\documentclass[a4paper,leqno,11pt]{amsart}

\usepackage{amsfonts,amssymb,verbatim,amsmath,amsthm,latexsym,textcomp,amscd}
\usepackage{latexsym,amsfonts,amssymb,epsfig,verbatim}
\usepackage{amsmath,amsthm,amssymb,latexsym,graphics,textcomp}
\usepackage{graphicx}
\usepackage{color}
\usepackage{url}
\usepackage{enumerate}
\usepackage[mathscr]{euscript}
\usepackage{mathtools} 
\usepackage{tikz}
\usetikzlibrary{cd}
\usepackage{dsfont}
\usetikzlibrary{matrix}
\usepackage{sseq}
\usepackage{hyperref}
\usepackage{xcolor}

\input xy
\xyoption{all}

\definecolor{maroon}{RGB}{128,0,0}

\DeclareMathOperator{\Hom}{Hom}
\DeclareMathOperator{\Ext}{Ext}

\theoremstyle{plain}
\newtheorem*{thma}{Theorem A}
\newtheorem*{thmb}{Theorem B}
\newtheorem*{thmc}{Theorem C}
\newtheorem*{thmd}{Theorem D}
\newtheorem*{thme}{Theorem E}
\newtheorem*{qn1}{Question I}
\newtheorem*{qn2}{Question II}
\newtheorem*{qn3}{Question III}
\newtheorem*{qn4}{Question IV}

\theoremstyle{definition}
\newtheorem{theorem}{Theorem}[section]
\newtheorem{prop}[theorem]{Proposition}
\newtheorem{lemma}[theorem]{Lemma}
\newtheorem{cor}[theorem]{Corollary}

\newtheorem{rmk}[theorem]{Remark}

\newtheorem{exam}[theorem]{Example}

\newtheorem{thm}[theorem]{Theorem}
\newtheorem{notation}[theorem]{Notation}

\theoremstyle{remark}
\swapnumbers
\newtheorem{ack}[theorem]{Acknowledgements}

\newcommand{\C}{{\mathbb C}}

\newcommand{\Hyp}{{\mathbb H}}

\newcommand{\Z}{{\mathbb{Z}}}

\newcommand{\E}{{\mathbb E}}
\newcommand{\W}{\mathbb W}

\newcommand{\Cok}{\mbox{Cok}}

\newcommand{\Top}{{\mbox{Top}}}

\newcommand{\F}{\mathbb F}

\newcommand\EE{{\mathcal E}}
\newcommand\FF{{\mathcal F}}

\newcommand\LL{{\mathcal L}}
\newcommand\MM{{\mathcal M}}
\newcommand\NN{{\mathcal N}}

\newcommand\PP{{\mathcal P}}

\newcommand\Sp{{\mathbb S}}

\newcommand{\G}{\mathbb G}
\newcommand{\even}{\tiny \mbox{even}}
\newcommand{\odd}{\tiny \mbox{odd}}
\newcommand{\oalpha}{\overline{\alpha}}
\newcommand{\obeta}{\overline{\beta}}
\newcommand{\Tel}{\mbox{Tel}}

\newcommand\PMF{{\PP\kern-2pt\MM\FF}}

\newcommand\PML{{\PP\kern-2pt\MM\LL}}

\newcommand\tr{\operatorname{tr}}

\newcommand{\fsubd}{\mathrel{{\scriptstyle\searrow}\kern-1ex^d\kern0.5ex}}
\newcommand{\bsubd}{\mathrel{{\scriptstyle\swarrow}\kern-1.6ex^d\kern0.8ex}}
\newcommand{\fsubeq}{\mathrel{\raise-.7ex\hbox{$\overset{\searrow}{=}$}}}
\newcommand{\bsubeq}{\mathrel{\raise-.7ex\hbox{$\overset{\swarrow}{=}$}}}

\newcommand{\tsh}[1]{\left\{\kern-.9ex\left\{#1\right\}\kern-.9ex\right\}}

\newcommand{\Ker}{\mbox{Ker}}
\newcommand{\HFPSS}{\mathit{HFPSS}}
\newcommand{\AHSS}{\mathit{AHSS}}

\newcommand{\pt}{\mbox{pt}}
\newcommand{\Cone}{\mbox{Cone}}
\newcommand{\Coker}{\mbox{Coker}}

\renewcommand{\Im}{\operatorname{Im}}

\makeatletter
\let\c@figure\c@theorem
\makeatother

\numberwithin{equation}{section}
\newenvironment{myeq}[1][]
{\stepcounter{theorem}\begin{equation}\tag{\thetheorem}{#1}}
	{\end{equation}}

\newenvironment{meq*}[1][]
{\stepcounter{theorem}\begin{equation*}\tag{\thetheorem}{#1}}
	{\end{equation*}}
\newtheorem{subsec}[theorem]{}
\newenvironment{mysubsection}[2][]
{\begin{subsec}\begin{upshape}\begin{bfseries}{#2.}
			\end{bfseries}{#1}}
		{\end{upshape}\end{subsec}}

\title[EO-theory of projective spaces]{Exotic structures on hyperbolic manifolds via the EO-theory of  Projective Spaces}

\author{Samik Basu}
\address{Stat-Math Unit\\ Indian Statistical Institute\\ Kolkata 700108\\ India}
\email{samik.basu2@gmail.com}
\author{Akashdwip Biswas}
\address{Stat-Math Unit\\ Indian Statistical Institute\\ Kolkata 700108\\ India}
\email{akashdwip1995@gmail.com}

\date{\today}
\hypersetup{
	colorlinks,
	citecolor=red,
	filecolor=black,
	linkcolor=blue,
	urlcolor=black
}

\begin{document}
\maketitle

\begin{abstract}
Farrell and Jones showed that negatively curved manifolds in dimension $\geq 5$ are topologically rigid in the sense that homotopy equivalence implies homeomorphism. In the smooth category, the result does not hold even for hyperbolic manifolds, and there are negatively curved manifolds $N$ homotopy equivalent to a given hyperbolic manifold $M$ but not diffeomorphic. For complex hyperbolic manifolds, the analogous construction is demonstrated only in dimensions of the form $8n+2$. In this paper, we prove this result in many other dimensions. The main approach involves the use of $EO$-theory of complex projective spaces to construct suitable examples of exotic spheres. These theories may be viewed as a version of real $K$-theory, and are defined at each prime $p$. The computation also yields positive results about the existence of   free smooth $S^1$ and $S^3$ actions on exotic spheres which do not bound parallelizable manifolds.
\end{abstract}

%\tableofcontents

\section{Introduction}
Hyperbolic manifolds possess strong rigidity properties which restricts the possibility of encountering exotic structures within them. The Mostow rigidity theorem \cite{Mos68} states that in dimensions $\geq 3$ these manifolds are determined from their fundamental group, that is, any isomorphism $\pi_1(M) \to \pi_1(N)$ for hyperbolic manifolds $M$ and $N$ is induced by a unique isometry $M\to N$.  The theorem was extended to rank one locally symmetric  spaces \cite{Mos73} and the ``compact" hypothesis was generalized  to the ``finite volume" case in  \cite{Pra73}. There are also similar rigidity phenomena for complex hyperbolic manifolds. In fact, the analogue for negatively curved K\"{a}hler manifolds was observed independently in \cite{Her91} and \cite{YZ91}. 

There is a topological version of the story which has been studied in great detail by Farrell and Jones \cite{FJ89b,FJ89c}. One of the topological rigidity results proved by them is that if $M$ is a closed connected manifold of dimension $\geq 5$ which is negatively curved, and $N$  is a closed manifold which is homotopy equivalent to $M$, then $N$ is homeomorphic to $M$ \cite{FJ90}. On the other hand, in the smooth category, Farrell and Jones constructed exotic smooth structures on real hyperbolic manifolds \cite{FJ89} in the following sense. Let $\Theta_m=\{\Sigma_1,\dots, \Sigma_k\}$ be the collection of exotic spheres \cite{KM63} of dimension $m$. For any real hyperbolic manifold $M$ of dimension $m$, there exists a finite cover $\hat{M}$ satisfying \\
1) The manifolds $\hat{M}\# \Sigma_1, \dots, \hat{M}\# \Sigma_k$ are pairwise non-diffeomorphic. \\
2) For each $\delta >0$, each of $\hat{M}\#\Sigma_1,\dots, \hat{M}\# \Sigma_k$ has a Riemannian metric whose sectional curvatures lie in the interval $(-1-\delta, -1+\delta)$. 

In the complex case, Farrell and Jones \cite{FJ94} proved an analogous result in complex dimension $m=4n+1$. They constructed manifolds $M$ and $N$ of dimension $2m$, with $M$ complex hyperbolic, $M$ and $N$ homeomorphic but not diffeomorphic, and for any $\epsilon >0$, $N$ has a negatively curved metric whose sectional curvatures lie in the interval $[-4-\epsilon, -1+\epsilon]$. The quaternionic case was also considered in \cite{AF04}, and they proved results in quaternionic dimension $2$, $4$, and $5$. For these values of $m$, there are manifolds $M$ and $N$ of dimension $4m$, such that $M$ is homeomorphic to $N$ but not diffeomorphic, $M$ is quaternionic hyperbolic, and $N$ possesses a negatively curved metric. The case of the Cayley projective plane was proved in \cite{AF03}. In this paper, we prove the following results which demonstrate the existence of exotic phenomena in certain other dimensions.
\begin{thma}  (Theorem \ref{exstrcxhyp} and Theorem  \ref{exstrquathyp}) \\
a)  Let $m= 36t+s$ for $s\in \{5, 10, 20\}$ and $ t\geq 0$, or 
\[m=(r+kp)(p(p-1)-1), \mbox{ where } \begin{cases} k\in \{1,\dots, r-1\} & \mbox{ if } r\leq \frac{p-1}{2}\\
 k\in \{1,\dots r-2\} & \mbox{ if }  \frac{p-1}{2}< r \leq p-1, \end{cases} \]  
for a prime $p\geq 5$. There are manifolds $M$ and $N$ of dimension $2m$, such that $M$ is complex hyperbolic, $M$ and $N$ are homeomorphic but not diffeomorphic, and $N$ supports a negatively curved Riemannian metric whose sectional curvatures may be chosen to lie in $[-4-\epsilon, -1+\epsilon]$ (for any given $\epsilon >0$). \\
b) Let $m\in \{ 18t +5, 18t+10\}$ or $m=(\frac{r+kp}{2})(p(p-1)-1)$ for a prime $p\geq 5$ and $r<p$, with  $ k\in \{1,\dots r-2\}$, and $k+r$ is even. There are manifolds $M$ and $N$ of dimension $4m$, such that, $M$ is quaternionic hyperbolic, $M$ and $N$ are not diffeomorphic, and $N$ supports a negatively curved Riemannian metric. 
\end{thma} 

\vspace*{0.2cm}

 In the complex case, comparing with the results of \cite{FJ94}, we observe that the results in dimension $2m$ for $m\in \{36t + 10, 36t+20\}$, and those $(r+kp)(p(p-1)-1)$ such that $4\nmid k+pr +1$ are new. In the quaternionic case, the results are mostly new as it was proved earlier only in quaternionic dimensions $2$, $4$, and $5$. In both the complex and the quaternionic cases, the smooth structure is changed by taking the connected sum with an exotic sphere. This operation changes the smooth structure on the chosen complex hyperbolic manifold (respectively quaternionic hyperbolic manifold) if the connected sum operation also changes the smooth structure on $\C P^m$ (respectively $\Hyp P^m$).  

\vspace*{0.2cm}

\begin{mysubsection}{The inertia groups} The collection $\Theta_m$ of smooth manifolds homeomorphic to $S^m$ is a group under connected sum. Given a smooth manifold $M$ of dimension $m$, the inertia group $I(M)$ is the subgroup of $\Theta_m$ comprising those $\Sigma$ that satisfy the condition $M\# \Sigma$ is diffeomorphic to $M$. One may also impose criteria on the diffeomorphism from $M\# \Sigma$ to $M$ and obtain : 
The homotopy inertia group $I_h(M)$ (respectively, concordance inertia group $I_c(M)$) consists of $\Sigma \in I(M)$ (respectively, $\Sigma \in I_c(M)$) such that the diffeomorphism is homotopic (respectively, concordant) to the identity. The computations for inertia groups are usually involved and there are hardly any systematic methods available to approach such a question. The usual method uses the infinite loop space $\Top/O$ and the formula \cite{KS77}
\begin{myeq}\label{ichtp}
 I_c(M)= \Ker \Big( \Theta_m \cong [S^m, \Top/O] \xrightarrow{q^\ast} [M,\Top/O]\Big),
\end{myeq}
where $q: M\to S^m$ is the map which quotients out the $(m-1)$-skeleton.

For projective spaces, most of the known computations are in low dimensions \cite{Kaw68}, where explicit formulas for attaching maps are useful in determining the concordance inertia group. In the case of $\C P^m$, one has $I_c(\C P^m) = I_h(\C P^m)=I(\C P^m)$ \cite{Kas15}, while for $\Hyp P^m$ one has $I_c(\Hyp P^m)= I_h(\Hyp P^m)$ but these are not equal to $I(\Hyp P^m)$ \cite{BK22}. In low dimensions, computations for the inertia group of $\C P^m$ are available in \cite{Kaw68}. For the inertia group in high dimensions, the results are usually of the following type\\
1) There is a sequence of exotic spheres that lie outside the inertia group. This is done for the dimensions $8n+2$ in \cite{FJ94}. \\
2) There is a sequence of exotic spheres that lie inside the inertia group. This is done in \cite{BK18} and \cite{BK22}.\\
 The approach in this paper yields results of the type 1). We prove
 \begin{thmb}  (Theorem \ref{inertcpn} and Theorem \ref{inerthpn}) \\
a) Let $m\in \{36t+5, 36t + 10, 36t+20\}$. There are $2$ exotic spheres $\Sigma_1$ and $\Sigma_2$ of dimension $2m$ that lie outside $I(\C P^m)$. \\
b) Suppose 
\[m=(r+kp)(p(p-1)-1), \mbox{ where } \begin{cases} k\in \{1,\dots, r-1\} & \mbox{ if } r\leq \frac{p-1}{2}\\
 k\in \{1,\dots r-2\} & \mbox{ if } r> \frac{p-1}{2}, \end{cases} \]  
for $p\geq 5$. There are $(p-1)$ exotic spheres $\Sigma_1, \dots, \Sigma_{p-1}$ of dimension $2m$ that lie outside $I(\C P^m)$. \\
c) Let $n\in\{18t+5,18t+10\}$. There are $2$ exotic spheres $\Sigma_1$ and $\Sigma_2$ of dimension $4n$ that lie outside $I_c(\Hyp P^m)$. \\
d) Let $m=(\frac{r+kp}{2})(p(p-1)-1)$ for a prime $p\geq 5$ and $r<p$, where  $ k\in \{1,\dots r-2\}$, and $k+r$ is even.   There are $p-1$ exotic spheres $\Sigma_1, \dots, \Sigma_{p-1}$ of dimension $4m$ that lie outside $ I_c(\Hyp P^m)$.
\end{thmb}

\vspace*{0.2cm}

These results are mostly new in the sense that results for these dimensions were not proved before. For $m=36t+5\equiv 1\pmod{4}$ or $m=(r+kp)(p(p-1)-1)$ with $4\mid k+pr+1$, the dimension matches with the results of \cite{FJ94}. However, here we show that there is a copy of $\F_p \subset \Theta_{2m}$ that maps non-trivially to $[\C P^m,\Top/O]$, and thus, the exotic spheres obtained here are different.  
\end{mysubsection}

\vspace*{0.2cm}

\begin{mysubsection}{Free smooth actions} 
The computations of inertia groups also yield results about free smooth $S^1$ and $S^3$ actions on exotic spheres. Hsiang \cite{Hsi66} observed that the exotic spheres coming from the $bP$-factor possess free smooth $S^1$-actions. Among the spheres arising from the $\Cok(J)$ factor, the question has been extensively studied in the unpublished work of Brumfiel \cite{Bru69}. From this approach, such a sphere $\Sigma$ of dimension $2k+1$, represented by a map $S^{2k+1} \to \Top/O$ lies in the image of the composite
\begin{myeq} \label{s1act}
S^{2k+1} \to S^{2k+1}/S^1 \simeq \C P^k \to \Top/O.
\end{myeq}
For $S^3$-actions, the argument is similar, and we have $\Sigma$ of dimension $4k+3$, represented by $S^{4k+3} \to \Top/O$ lies in the image of 
\begin{myeq}\label{s3act}
 S^{4k+3} \to S^{4k+3}/S^3 \simeq \Hyp P^k \to \Top/O.
\end{myeq}
This was used in \cite{BK22} to construct free smooth $S^3$-actions on exotic $15$-spheres, and in \cite{BQ26} to provide infinite families of non-$bP$-spheres with free smooth $S^1$ and $S^3$-actions. We show that our computations also yield analogous results, with a few new results.
\begin{thmc}  (Theorem \ref{smactexotic})
a) Let $l\in \{72t + 13, 72t+23, 72t+43\}$. There exists an exotic non-$bP$-sphere $\Sigma$ of dimension $l$ which has a free smooth  $S^1$-action. 
\\
b) Let $l=72t + 23$ or $l=72t+43$. There exists an exotic non-$bP$-sphere $\Sigma$ of dimension $l$ which has a free smooth $S^3$-action. 
\end{thmc}
\end{mysubsection}	

\vspace*{0.2cm}

Now, we describe our approach to prove the results described above. The underlying questions are about proving the triviality and non-triviality of various classes in $[X,\Top/O]$ where $X$ is either a sphere, a complex projective space, or a quaternionic projective space \eqref{ichtp}, \eqref{s1act}, \eqref{s3act}. As the results featuring in this paper are about non-$bP$-spheres, we consider the image under the map $\Top/O \to G/O$. Further we have the map $G \to G/O$ and the fact that the homotopy groups of $G$ are the stable homotopy groups of spheres. Thus, we inspect the following analogues of \eqref{ichtp}, \eqref{s1act}, \eqref{s3act} for which we may apply methods from stable homotopy theory.  
\begin{qn1}
Find classes $\chi \in \pi_{2m}^s$ such that the image of $\chi$ in the composite 
\[\C P^m \xrightarrow{q} S^{2m} \xrightarrow{\chi} \Sp, \]
is non-trivial.
\end{qn1}

\begin{qn2}
Find classes $\chi \in \pi_{4m}^s$ such that the image of $\chi$ in the composite 
\[\Hyp P^m \xrightarrow{q} S^{4m} \xrightarrow{\chi} \Sp, \]
is non-trivial.
\end{qn2}

\begin{qn3}
Find classes $\chi \in [\C P^m, \Sp]$ such that the image of $\chi$ in the composite 
\[S^{2m+1} \xrightarrow{\eta_{\C P^m}} \C P^m \xrightarrow{\chi}  \Sp, \]
is non-trivial.
\end{qn3}

\begin{qn4}
Find classes $\chi \in [\Hyp P^m, \Sp]$ such that the image of $\chi$ in the composite 
\[S^{4m+3} \xrightarrow{\nu_{\Hyp P^m}} \Hyp P^m \xrightarrow{\chi}  \Sp, \]
is non-trivial.
\end{qn4}

\vspace*{0.2cm}

The method we use to approach these questions is : find a suitable cohomology theory $E$, where the classes in consideration here map non-trivially under the unit map $\Sp \to E$. The choice of $E$ is critical as it is supposed to detect certain homotopy groups of spheres. A suitable example is real $K$-theory $KO$, which detects the class $\mu_{8n+2}$ \cite{Ada66} in $\pi_{8n+2}^s$.  

In this paper the cohomology theories considered by us are the higher $K$-theories of Hopkins and Miller. The Morava $E$-theory specta $E_n$ are a Landweber-exact cohomology theory (for every prime $p$) with homotopy groups given as \cite{Rez98}
\[ \pi_\ast E_n \cong W(\F_{p^n})[[u_1,\dots,u_{n-1}]][u^{\pm}],\]
where $W(\F_{p^n})$ is the ring of $p$-typical Witt vectors of $\F_{p^n}$, $|u_i|=0$, $|u|=2$. This is an $\E_\infty$-ring spectrum \cite{Rez98, GH04}, and it has an action of the extended Morava stabilizer group $\G_n = \Sp_n \rtimes Gal(\F_{p^n}/\F_p)$. The unit map $\Sp \to E_n$ factors through the homotopy fixed points of $E_n$ under this action. The entire homotopy fixed point spectrum is usually computationally difficult but the homotopy fixed points for finite subgroups are more manageable. These are also useful in detecting cohomology classes. In this paper, we consider $n=p-1$ for $p\geq 3$, and $E_n^{hG}$ where $G$ is the maximal finite subgroup of $\G_n$. Suggestively, the homotopy fixed point spectrum is denoted $EO_n$. We provide more details in \S \ref{eop-1}.   
 
For $p=2$ and $n=1$, the spectrum $E_1 \simeq K_2^{\wedge}$, the $2$-completed complex $K$-theory spectrum. The finite subgroup in this case is $C_2$ which acts on $K_2^\wedge$ by complex conjugation. The homotopy fixed points spectrum in this case is $KO_2^\wedge$, the $2$-completed real $K$-theory spectrum. The spectra $EO_{p-1}$ are in this sense ``higher" analogues of real $K$-theory. 

From the point of view of Questions I-IV, we need several inputs. Firstly, we require a computation of the homotopy groups of $EO_{p-1}$ which is available from \cite{Nav99, Nav10}.  Secondly, we need to identify classes in the Hurewicz image $h_\iota : \pi_\ast(\Sp)\to \pi_\ast(EO_{p-1})$. At the prime $3$, this follows from the results of \cite{BS23}, while for $p\geq 5$, we rely on the relevantly weaker results in \cite{Nav99}. Thirdly, we require a computation of $EO_{p-1}^\ast(\C P^m)$ and $EO_{p-1}^\ast(\Hyp P^m)$. In this paper, we provide a complete calculation of these cohomology rings.  

\vspace*{0.2cm}

\begin{mysubsection}{Homotopy fixed point spectral sequence} 
The main technique to compute $EO_{p-1}$-homology and cohomology is the homotopy fixed point spectral sequence \cite{Dug03}
\[ E_2^{s,t} = H^s(G; \pi_{s+t}(E_{p-1}\wedge X)) \implies \pi_t(EO_{p-1} \wedge X).\]
For $X =S^0$, the $E_2$-page of this spectral sequence has been computed by Hopkins and Miller (see Theorem \ref{hmcomp}). There is an exact sequence 
\[\pi_t E_{p-1} \xrightarrow{tr} H^s(G; \pi_t(E_{p-1})) \to \F_p[\alpha,\beta, \Delta^{\pm 1}]/(\alpha^2), \]
where $|\Delta|= (0, 2p(p-1)^2)$, $|\alpha|= (1,2p-2)$, $|\beta| = (2, 2p(p-1)-2)$. The homotopy fixed point spectral sequence receives a map from the Adams Novikov spectral sequence \cite{Rav78}
\[\Ext^{s,t}_{BP_\ast BP}(BP_\ast, BP_\ast(X)) \to   H^s(G; \pi_{s+t}(E_{p-1}\wedge X)), \] 
which is used to show that $\alpha$ and $\beta$ are the images of $\alpha_1$ and $\beta_1$ respectively \cite{Nav99}, and thus, they lie in the Hurewicz image. The differentials of the homotopy fixed point spectral sequence are also determined by this map. There are precisely two non-trivial differentials $d_{2p-1}$ and $d_{2(p-1)^2+1}$. Toda's computations \cite{Tod67, Tod68}  imply that $\alpha_1\beta_1^p =0$, and $\beta_1^{(p-1)^2+1} =0$. Now using $h_\iota(\alpha_1)=\alpha$ and $h_\iota(\beta_1)=\beta$, one only obtains that $\F_p[\alpha,\beta]/(\alpha^2, \alpha \beta^{p-1}, \beta^{(p-1)^2+1})$ lies in the image of $h_\iota$. Thus, given a prime $p$, $h_\iota$ contains elements in a small range of degrees. At the prime $3$, the computations of \cite{BS23} determine the precise Hurewicz image which also contains infinite families. One has 
\begin{myeq}\label{hurewiczeo2}
\begin{aligned}
& h_\iota(\alpha_1)=\alpha,\\
& h_\iota(\beta_1^{i-1}\beta_{9t+1}) = \Delta^{6t} \beta^i \quad t\geq 0, ~ 1\leq i \leq 4, \\
& h_\iota(\alpha_1\beta_{9t+3/3})= \Delta^{6t}\beta b, \\
& h_\iota(\beta_1^{i-1}\beta_{9t+6/3}) = \Delta^{6t+3} \beta^i \quad t\geq 0, ~ 1\leq i \leq 4, \\
& h_\iota(\alpha_1 \beta_{9t+7})= \Delta^{6t+3}\beta b.
\end{aligned}
\end{myeq}
\end{mysubsection} 

\vspace*{0.2cm} 

\begin{mysubsection}{The homotopy fixed point spectral sequence for even cell complexes}
The projective spaces $\C P^m$ and $\Hyp P^m$ are CW complexes with even dimensional cells. For a CW complex $X$ with even cells, the homotopy fixed point spectral sequence has a trigraded spectral sequence  converging to its $E_2$-term (see Theorem \ref{FSS})
%\begin{myeq}\label{filtspseq}
\[\EE^{s,t,r} = H^s\Big(G; \pi_{s+t}\big(E_{p-1}\wedge \frac{X^{(r)}}{X^{(r-1)}}\big)\Big) \implies H^s(G; \pi_{s+t}(E_{p-1} \wedge X)).\]
%\end{myeq}
This type of spectral sequence is fairly common in computations in stable homotopy theory. In the case where the computations are made over a Hopf algebroid this is called the algebraic Atiyah-Hirzebruch spectral sequence. In this paper, we call this the {\it Filtration spectral sequence}. We prove some properties of the  filtration spectral sequence in Theorem \ref{FSS}, Proposition \ref{diffl-att}. 
%1) The filtration spectral sequence for the $EO_{p-1}$-cohomology of $X$ is multiplicative. \\
%2)  Suppose $\gamma\in E_2\mbox{-page}$ of the homotopy fixed point spectral sequence of $X^{(r)}/X^{(r-1)})$ is represented in \eqref{filtspseq} by $y\in \EE_2^{0,t,r}$. Also suppose $\zeta\in E_2\mbox{-page}$ of the homotopy fixed point spectral sequence for $X^{(r+m)}/X^{(r+m-1)}$ is represented in \eqref{filtspseq} by $z \in \EE_2^{1,t-1,r+m}$. Assume that $d_i(\alpha)=0$ for $i<m$,  $d_m(\gamma)=\zeta$ in the Atiyah-Hirzebruch spectral sequence of $X$ for $EO_{p-1}$, and $d_i(y)=0$ for $i<m$ in the spectral sequence \eqref{filtspseq}. Then,
%	 $d_m(y)=z$ in the filtration spectral sequence. 
%
%\vspace*{0.2cm} 

The first example where we apply this setup is the computation of $KO^\ast(\C P^m)$ (see Theorem \ref{compko}). The $C_2$-action on $K$-theory is by complex conjugation, and $K^{hC_2} \simeq KO$. We  write down the precise action of $C_2$ on $K^\ast(\C P^m)$, and compute the filtration spectral sequence, and obtain the following expression, which matches the answer in \cite{Yam07}.  
\[\scalebox{0.9}{$KO^\ast(\C P^m)) \cong 	 \begin{cases}  \frac{\Z[t^{\pm 2},\alpha, \eta]}{(\alpha^2 - 4 t^2,\alpha\eta, 2\eta, \eta^3)} \oplus \Z\{t^s y^k, t^sy^j u\} \mbox{ for } 1\leq k \leq \lfloor\frac{m}{2}\rfloor, ~ 1\leq j <\lfloor\frac{m}{2}\rfloor,&\mbox{ if } m \mbox{ is even}\\
 \frac{\Z[t^{\pm 2},\alpha,\eta]}{(\alpha^2-4t^2,\alpha\eta, \eta^3, 2\eta)}\{1,z\} \oplus \Z\{t^s y^k, t^sy^j u\} \mbox{ for }  1\leq k \leq \lfloor\frac{m}{2}\rfloor,~1\leq j < \lfloor\frac{m}{2}\rfloor,  &\mbox{ if } m \mbox{ is odd}. \end{cases}$}  \] 
The ring $\frac{\Z[t^{\pm 2},\alpha, \eta]}{(\alpha^2 - 4 t^2,\alpha\eta, 2\eta, \eta^3)} \cong KO^\ast(\pt)$, and for $1\leq k \leq \lfloor \frac{m}{2}\rfloor$, $\oplus_{s\in \Z}\Z\{ t^s y^k, t^l u y^{k-1}\}$ is a $KO^\ast(\pt)$-submodule isomorphic to $K^\ast(\pt)$. This last identification is related to Wood's theorem that $KO\wedge C(\eta) \simeq K$. When $n= \infty$, one may also compare this with the splitting  \cite[\S 15]{GM95}
\begin{myeq}\label{kowedgecpinfty}
 KO\wedge \C P^\infty_+ \simeq KO\vee \bigvee_{k\geq 1} \Sigma^{4k-2} K.
\end{myeq}

\end{mysubsection}

\vspace*{0.2cm}

\begin{mysubsection}{$EO$-theories of projective spaces}
The splitting in \eqref{kowedgecpinfty} has been generalized to the $EO$-case in \cite{BC22}. They work in the more general situation of $C_p$ acting on $E_\Gamma$ where the underlying formal group law of $\Gamma$ is of a height that is divisible by $p-1$ ($p$ an odd prime), and considering $EO_\Gamma := E_\Gamma^{hC_p}$, and prove
\begin{myeq}\label{eoncpsplit}
EO_\Gamma \wedge \C P^\infty_+ \simeq \MM \vee \FF,
\end{myeq}
where $\MM$ is a compact $EO_\Gamma$-module, and $\FF$ is a wedge of copies of $E_\Gamma$. In this paper, we consider $n=p-1$, and the maximal finite subgroup $G$ instead of $C_p$. The following result follows from the computations in Theorem \ref{eo2cpinfty} and Theorem \ref{spliteoncpinfty}.
\begin{myeq}\label{eop-1split}
EO_{p-1} \wedge \C P^\infty_+ \simeq EO_{p-1} \vee \Big(\bigvee_{i=1}^{p-2} \Sigma^{2i} EO_{p-1} \wedge C^{[i]} \Big) \vee \bigvee_{j\geq 0} \Sigma^{2jp+2(p-1)} E_{p-1}^{h(G/C_p)} 
\end{myeq}
The first two factors in the splitting above forms the compact $EO_{p-1}$-module analogous to $\MM$ in \eqref{eoncpsplit} and the last factor forms the module analogous to $\FF$. The complex $C^{[i]}$ for $i\leq p-1$, is defined as the cellular spectrum with cells in degrees $0, 2p-2, 2(2p-2), \dots , i(2p-2)$ where the attaching map of the $j(2p-2)$-cell onto the $(j-1)(2p-2)$-cell is given by an unit multiple of $\alpha_1$. These criteria do not uniquely define the complex $C^{[i]}$, and the notation stands for any cellular spectrum satisfying this stated condition (see \S \ref{itercones} and \eqref{defc2alph} for a precise definition). The complexes $C^{[i]}$ bear a close resemblence to the $X_i$ of \cite[\S 2]{Cha20} which are in fact uniquely defined. The computations in \cite[Corollary 3.5]{Cha20} imply the following analogue of Wood's theorem for $EO_{p-1}$ 
\[ EO_{p-1} \wedge C^{[p-1]} \simeq E_{p-1}^{h(G/C_p)}.\]
\begin{exam}
Let $p=3$. Here there are three possible $C^{[i]}$s, namely, $S^0$, $C^{[1]}$ and $C^{[2]}$. The homotopy groups of $EO_2$ \eqref{eo2htpygp} are given by 
\[\pi_\ast(EO_2) \cong \frac{\F_3[\Delta^{\pm 3}, \alpha,\beta,b]}{(\alpha^2, \alpha\beta^2, b\beta^2, \alpha b -\beta^3, \beta^5)} \pmod{M_3}.\]
We also derive explicit expressions for ${EO_2}_\ast(C^{[1]})$ \eqref{einftyhfpsseo2calph} and ${EO_2}_\ast (C^{[2]})$ \eqref{einftyhfpsseo2c[2]}.
\[{EO_2}_\ast (C^{[1]})\cong \frac{{EO_2}_\ast \{\iota_0,[\Delta \iota_0], \kappa_7\}}{(\alpha\iota_0, \beta^4 \kappa_7, \beta^3 \iota_0, \beta^3 [\Delta \iota_0], \alpha \kappa_7 -\beta \iota_0, b \kappa_7 - \beta [\Delta \iota_0]) } \pmod{M_3,\chi_4}.\] 
\[{EO_2}_\ast (C^{[2]})\cong \frac{{EO_2}_\ast }{(\alpha, \beta, b) }\{\iota_0, [3\iota_4], [3\iota_8]\} \pmod{M_3}.\] 
\end{exam}

\vspace*{0.2cm}

We determine the $EO_{p-1}$-cohomology of $\C P^\infty$ in terms of generators and relations, and in order to do this, we also compute  $\pi_\ast(EO_{p-1}\wedge C^{[k]})$ explicitly (Proposition \ref{hfpssc[k]}).  
\begin{thmd}(Theorem \ref{eoncpinfty})
The $EO_{p-1}$-cohomology of $\C P^\infty$ is generated over $\pi_\ast EO_{p-1}$ by the classes 
\begin{align*}
&x^{jp} \mbox{ for } j\geq 1 \mbox{ of (cohomological) degree } 2jp ,\\
&x_{(kp,\Delta^i)} \mbox{ for } 1\leq i\leq k \leq p-2  \mbox{ of degree } 2kp - 2ip(p-1)^2, \mbox{ and }\\
& x_{(k,\Delta^i \alpha)} \mbox{ for } 1\leq k \leq p-2, 0\leq i \leq p-k-2, \mbox{ of degree } 2k - (2p-3) -2ip(p-1)^2.
\end{align*} 
The classes $x^{jp}$ are annihilated by $\beta$ for $j\geq p-1$, and $x^{kp}$ and $x_{(kp,\Delta^i)}$ are annihilated by $\beta^{(p-1)(p-1-k)+1}$ if $k<p-1$. The classes $x_{(k,\Delta^i\alpha)}$ are annihilated by $\beta^{(k+1)(p-1)}$.   
\end{thmd}

\vspace*{0.2cm}

We refer to \eqref{noteoncp} for the notations above. The analogue of Theorem D for quaternionic projective spaces involves only the generators in degrees that are multiples of $4$ (Theorem \ref{eonhpm}), and this implies the following splitting (Theorem \ref{spliteonhpinfty})
\begin{myeq} \label{eohpinfty}
EO_{p-1} \wedge \Hyp P^\infty_+ \simeq EO_{p-1} \vee \Big( \bigvee_{i=1}^{\frac{p-3}{2}} \Sigma^{4i} EO_{p-1} \wedge C^{[2i]} \Big) \vee \bigvee_{j\geq 0} \Sigma^{4jp + 2(p-1)} E_{p-1}^{h(G/C_p)}. 
\end{myeq}
Note that at $p=3$, in \eqref{eohpinfty} the middle term does not appear and we have  
\[EO_2\wedge \Hyp P^\infty_+ \simeq EO_2 \vee \bigvee_{j\geq 0} \Sigma^{12j+4} E_2^{h(G/C_3)}.\]

We now describe $EO_{p-1}^\ast(\C P^m)$ and $EO_{p-1}^\ast(\Hyp P^m)$ for finite $m$. One observes that the splittings in \eqref{eop-1split} and \eqref{eohpinfty} are skeletal in the sense that we obtain a splitting of $EO_{p-1} \wedge \C P^m$ (respectively $EO_{p-1}\wedge \Hyp P^m$) by taking the $(2m)$-skeleton (respectively $(4m)$-skeleton) of each summand. Now applying $F_{\scalebox{.65}{$EO_{p-1}-\mbox{mod}$}}(-,EO_{p-1})$ together with the observation that the Spanier-Whitehead dual of $C^{[i]}$ is a suspension of a $C^{[i]}$ one obtains splittings of $F(\C P^m, EO_{p-1})$ (Theorem \ref{spliteoncpm}) and $F(\Hyp P^m, EO_{p-1})$ (Theorem \ref{spliteonhpinfty}) into wedges of suspensions of $EO_{p-1}\wedge C^{[i]}$. 
The following describes how $EO_{p-1}^\ast(\C P^m)$ and $EO_{p-1}^\ast(\Hyp P^m)$ looks like for sufficiently large $m$, which feed into the proofs of Theorem A, Theorem B and Theorem C. 
\begin{thme} (Theorem \ref{eoncpm} and Theorem \ref{eonhpm}) \\
a) For $m\geq p(p-1)$, the $EO_{p-1}$-cohomology of $\C P^m$ 
is generated over $\pi_\ast EO_{p-1}$ by the generators of $EO_{p-1}^\ast(\C P^\infty)$ in cohomological degree $\leq 2m$ and the classes
\begin{align*}
&x^{k} \mbox{ for } m-p+1< k \leq m, \mbox{ } k\neq 0 \bmod{p}, \mbox{ and }\\
%&x_{(k-[k]_p(p-1),\Delta^i\)} \mbox{ for } 1\leq i\leq k \leq p-2  \mbox{ of degree } 2kp - 2ip(p-1)^2, \mbox{ and }\\
& x_{(k-[k]_p(p-1),\Delta^i \alpha)} \mbox{ for } m-p+1< k \leq m, \mbox{ } k\neq 0 \bmod{p}, \mbox{ } 0\leq i \leq p-[k]_p-2.
\end{align*} 
The classes $x^k$ are annihilated by $\beta^{(p-1)(p-1-[k]_p)+1}$ and the classes $x_{(k-[k]_p(p-1),\Delta^i \alpha)}$ are annihilated by $\beta^{([k]_p+1)(p-1)}$. \\
b)  For $m\geq \frac{p(p-1)}{2}$, the $EO_{p-1}$-cohomology of $\Hyp P^m$ 
is generated over $\pi_\ast EO_{p-1}$ by the generators of $EO_{p-1}^\ast(\Hyp P^\infty)$ in cohomological degree $\leq 4m$ and the classes
\begin{align*}
&y^{k} \mbox{ for } m-\frac{p-1}{2}< k \leq m, \mbox{ } k\neq 0 \bmod{p}, \mbox{ and }\\
%&x_{(k-[k]_p(p-1),\Delta^i\)} \mbox{ for } 1\leq i\leq k \leq p-2  \mbox{ of degree } 2kp - 2ip(p-1)^2, \mbox{ and }\\
& y_{(k-[2k]_p\frac{(p-1)}{2},\Delta^i \alpha)} \mbox{ for } m-\frac{p-1}{2}< k \leq m, \mbox{ } k\neq 0 \bmod{p}, \mbox{ } 0\leq i \leq p-[2k]_p-2.
\end{align*} 
The classes $y^k$ are annihilated by $\beta^{(p-1)(p-1-[2k]_p)+1}$ and the classes $y_{(k-[2k]_p\frac{(p-1)}{2},\Delta^i \alpha)}$ are annihilated by $\beta^{([2k]_p+1)(p-1)}$. 
\end{thme}

\vspace*{0.2cm}

The expressions in Theorem E also lead to a direct expression of $EO_{p-1}^\ast(\C P^m)$ and $EO_{p-1}^\ast(\Hyp P^m)$ in terms of generators and relations. These usually form a long list and are cumbersome to write down. In the example below, we write down the expressions at $p=3$.  
\begin{exam}
We now demonstrate the explicit formula at $p=3$. We have (Theorem \ref{eo2cpinfty})
\[EO_2^\ast(\C P^\infty) \cong \frac{\frac{{EO_2}_\ast [z, z_{(\Delta)}]}{((\alpha,\beta^3,\beta z) z, (\alpha, \beta z, \beta^3) z_{(\Delta)}, z_{(\Delta)}^3-\Delta^3z^3 )}\{1, x_{(\alpha)}, x_{(1,3)}, x_{(2,3)}\}}{ (\alpha x_{(\alpha)} - \beta z, bz-\beta^2 x_{(\alpha)}, bx_{(\alpha)}-\beta z_{(\Delta)}, (\beta^4,z) x_{(\alpha)}, (\alpha,\beta)x_{(1,3)}, (\alpha, \beta)x_{(2,3)}  )} ,\]
where the cohomological degree of the generators are given by $|z|=6$, $|x_{(\alpha)}|=-1$, $|x_{(1,3) }|=2$, $|x_{(2,3)}|=4$, and $|z_{(\Delta)}|=-18$. 
For $\Hyp P^\infty$, we have (Theorem \ref{eo2hpn}) 
\[EO_2^\ast(\Hyp P^\infty) \cong \frac{\frac{{EO_2}_\ast [y]}{((\alpha, \beta) y)}\{1, y_{(\Delta)},y_{(\Delta^2)}, x_{(2,3)}, x_{(4,3)}\}}{  ((\alpha, \beta ) y_{(\Delta)},(\alpha,\beta) y_{(\Delta^2)},  (\alpha,\beta)x_{(2,3)}, (\alpha, \beta)x_{(4,3)}  )} ,\]
with $|y|=12$, $|y_{(\Delta)}|=-12$, $|y_{(\Delta^2)}|=-36$, $|x_{(2,3)}|=4$, and $|x_{(4,3)}|=8$. In the case of finite $m$, it is simplest to write the formula of $EO_2^\ast(\C P^m)$ as a module over $EO_2^\ast(\C P^\infty)$. We have the following formulas (Theorem \ref{eo2cpn})
\[ EO_2^\ast(\C P^{3k}) \cong \frac{\frac{EO_2^\ast(\C P^\infty)}{(z^{k+1},z^k x_{(1,3)}, z^k x_{(2,3)})}\{1,x_{(3k-1)}\}}{(z^{k-1}x_{(2,3)}-3x_{(3k-1)}, x_{(1,3)}x_{(3k-1)}-3z^k, (z,z_{(\Delta)}, x_{(\alpha)}, x_{(2,3)})x_{(3k-1)}) },\]
\[EO_2^\ast(\C P^{3k+1}) \cong \frac{\frac{EO_2^\ast(\C P^\infty)}{(z^{k+1}, z^k x_{(2,3)})}\{1,x_{(3k-1,\alpha)}, x_{(3k+1)},x_{(3k+1,\Delta)}\}}{\begin{pmatrix} 
\alpha x_{(3k-1, \alpha)}-\beta x_{(3k+1)}, &  (\beta^4,z, I) x_{(3k-1,\alpha)}, &  (\alpha, z, \beta^3,I) x_{(3k+1)},  \\  
bx_{(3k-1,\alpha)} - \beta x_{(3k+1,\Delta)}, &z^k x_{(1,3)} - 3 x_{(3k+1)}, &(\alpha,z,b,\beta^3,I)x_{(3k+1,\Delta)},\\
bx_{(3k+1)}-\beta^2 x_{(3k-1,\alpha)}
\end{pmatrix}} ,\]
\[EO_2^\ast(\C P^{3k+2}) \cong \frac{\frac{EO_2^\ast(\C P^\infty)}{(z^{k+1})}\{1,x_{(3k-1,\alpha)}, x_{(3k+1)},x_{(3k+1,\Delta)}, x_{(3k+2)} \}}{\resizebox{.8\hsize}{!}{$\begin{pmatrix} 
\alpha x_{(3k-1, \alpha)}-\beta x_{(3k+1)}, &  (\beta^4,I) x_{(3k-1,\alpha)}, &  (\alpha, z, \beta^3, z_{(\Delta)}, x_{(2,3)}) x_{(3k+1)},  \\  
bx_{(3k-1,\alpha)} - \beta x_{(3k+1,\Delta)}, &z^k x_{(1,3)} - 3 x_{(3k+1)}, &(\alpha,z,b,\beta^3, z_{(\Delta)}, x_{(2,3)})x_{(3k+1,\Delta)},\\
 zx_{(3k-1,\alpha)}-\alpha x_{(3k+2)}, & (z,I) x_{(3k+2)}, & z^kx_{(2,3)}-3 x_{(3k+2)}, \\
x_{(\alpha)}x_{(3k+1)}-\alpha x_{(3k+2)}, & x_{(\alpha)} x_{(3k+1,\Delta)} - b x_{(3k+2)}, & x_{(1,3)} x_{(3k+1)}-3 x_{(3k+2)}  \\
bx_{(3k+1)}-\beta^2 x_{(3k-1,\alpha)}
\end{pmatrix}$}}  ,\]
where $I=(x_{(\alpha)}, x_{(1,3)},x_{(2,3)}, z_{(\Delta)}) \subset EO_2^\ast(\C P^\infty)$. 
The $EO_2$-cohomology of $\Hyp P^m$ for $m\geq 3$ is described by (Theorem \ref{eo2hpn})
\[ EO_2^\ast(\Hyp P^{3k}) \cong\frac{EO_2^\ast(\Hyp P^\infty)}{(y^{k+1},y^k x_{(2,3)}, y^k x_{(4,3)})},\]
\[ EO_2^\ast(\Hyp P^{3k+1}) \cong \frac{\frac{EO_2^\ast(\Hyp P^\infty)}{(y^{k+1}, y^k x_{(4,3)})}\{1,x_{(6k+2)}\}}{((y,I)x_{(6k+2)}, y^k x_{(2,3)} - 3 x_{(6k+2)})} ,\]
\[EO_2^\ast(\Hyp P^{3k+2}) \cong \frac{\frac{EO_2^\ast(\Hyp P^\infty)}{(y^{k+1}, y^k x_{(2,3)})}\{1,x_{(6k+2,\alpha)}, x_{(6k+4)},x_{(6k+4,\Delta)}\}}{\begin{pmatrix} 
\alpha x_{(6k+2, \alpha)}-\beta x_{(6k+4)}, &  (\beta^4,y, I) x_{(6k+2,\alpha)}, &  (\alpha,   \beta^3,y, I) x_{(6k+4)},  \\  
bx_{(6k+2,\alpha)} - \beta x_{(6k+4,\Delta)}, &y^k x_{(2,3)} - 3 x_{(6k+2)}, &(\alpha,b,\beta^3, y, I)x_{(6k+4,\Delta)},\\
 y^k x_{(4,3)}-3 x_{(6k+4)} , & bx_{(6k+4)}-\beta^2 x_{(6k+2,\alpha)} 
\end{pmatrix}}  ,\]
where $I=(y_{(\Delta)}, y_{(\Delta^2)}, x_{(2,3)},x_{(4,3)} ) \subset EO_2^\ast(\Hyp P^\infty)$. 
\end{exam}
\end{mysubsection}

\vspace*{0.2cm}

\begin{notation} Throughout the document, we use the following notations.
\begin{itemize}
\item All manifolds are assumed to be smooth, compact, and without boundary, unless otherwise stated.
\item The notation $\F_p$ stands for the field with $p$ elements for a prime $p$, and $\Z_p$ for the $p$-adic integers.
\item The notation $n$ stands for $p-1$ where $p$ is a prime. 
\item Write $[k]_p=(p-1)-k+[\frac{k}{p}]$ which equals $(p-1)-s$, where $s$ is the remainder when $k$ is divided by $p$.
\item The notations $O = \varinjlim O(m)$, $\Top=\varinjlim \Top(m)$, and $G=\varinjlim G(m)$ stands for the usual direct limits of the group of orthogonal transformations, homeomorphisms, and homotopy equivalences respectively. 
\item $\Theta_m$ is the collection of exotic spheres which is a group under connected sum. 
\item For a smooth manifold $M$,  
\begin{align*}
& I(M) = \mbox{ The inertia group of } M,\\
& I_h(M) = \mbox{ The homotopy inertia group of } M,\\
& I_c(M) = \mbox{ The concordance inertia group of } M.
\end{align*}
\item The usual quotient map $S^{2m+1} \to \C P^m$ is denoted by $\eta_{\C P^m}$, and the usual quotient map $S^{4m+3}\to \Hyp P^m$ is denoted by $\nu_{\Hyp P^m}$. 
\item The notation $\Sp$ stands for the sphere spectrum.  
\item For spectra $X $, $Y$, $F(X,Y)$ denotes the function spectrum.  
\item For a spectrum $X$, the Spanier-Whitehead dual $F(X,\Sp)$ of $X$ is denoted by $D(X)$. 
\item We denote the Hurewicz homomorphism $\pi_\ast(\Sp) \to \pi_\ast(EO_{p-1})$ by $h_\iota$. For a spectrum $X$, $[X,\Sp]\to [X,EO_{p-1}]$ is also denoted by $h_\iota$. 
\item We use the abbreviated forms :  HFPSS for the homotopy fixed point spectral sequence, AHSS for the Atiyah Hirzebruch spectral sequence, and FSS for the filtration spectral sequence. 
\item The notations $E_\ast ( EO_{p-1}\wedge X)$ and $E_\ast (EO_{p-1}^\ast(X))$ stands for the homotopy fixed point spectral sequences converging to ${EO_{p-1}}_\ast(X)$ and $EO_{p-1}^\ast(X)$ respectively. 
\item The notations $\EE_\ast ( EO_{p-1}\wedge X)$ and $\EE_\ast (EO_{p-1}^\ast(X))$ stands for the filtration  spectral sequences converging to $E_2(EO_{p-1}\wedge X)$ and $E_2(EO_{p-1}^\ast(X))$ respectively. 
\item The notation $C^{[i]}$ stands for a cellular spectrum  with cells in degrees $2j(p-1)$ for $0\leq j \leq i$, and that the attaching map of the cell in degree $2j(p-1)$ onto the cell in degree $2(j-1)(p-1)$ (by quotienting out the lower skeleton) is given by a unit multiple of $\alpha_1$. 
\end{itemize}
\end{notation}

\vspace*{0.2cm}

\begin{mysubsection}{Organization}
In \S \ref{spseqcons}, we discuss the relevant background on the homotopy fixed point spectral sequence, the filtration spectral sequence, and then use them in computations for real $K$-theory. The $3$-primary calculations of $EO_{p-1}^\ast(\C P^m)$ and $EO_{p-1}^\ast(\Hyp P^m)$ are discussed in \S \ref{higherkth3}. In \S \ref{higherkthgeq5}, we discuss the analogues at primes $\geq 5$. In \S \ref{applnexotic}, we point out consequences for exotic smooth structures on hyperbolic manifolds, and also for free smooth actions on exotic spheres. 
\end{mysubsection}

\vspace*{0.2cm}

\begin{ack}
The first author would like to thank Prasit Bhattacharya, Hans-Werner Henn, and Ramesh Kasilingam for useful conversations. The research of the first author was supported by ANRF grant ARG/2025/000434/MS.
\end{ack}

\vspace*{0.2cm}

	\section{The real K-theory of complex projective spaces}\label{spseqcons}
	We demonstrate a computation of the $KO$-cohomology of  $\C P^{m}$ using the homotopy fixed point spectral sequence, and write it as a ring with generators and relations, reproving the results of Yamaguchi \cite{Yam07}. This describes our method which we will later apply for the computation of the $EO$-theory of $\C P^m$.   The computation involves an interaction between the Atiyah Hirzebruch spectral sequence, the homotopy fixed point spectral sequence, and the filtration spectral sequence. This idea is described in the diagram below.
	\[\begin{tikzcd}
		{H^{*}\Big(X;H^{*}(C_2;KU^{*}(pt))}\Big) &&& {H^{*}(X;KO^{*}(pt))} \\
		\\
		{H^{*}(C_2;KU^{*}(X))} &&& {KO^{*}(X)}
		\arrow[Rightarrow, from=1-1, to=1-4]
		\arrow["FSS"', Rightarrow, from=1-1, to=3-1]
		\arrow["\AHSS", Rightarrow, from=1-4, to=3-4]
		\arrow["\HFPSS"', Rightarrow, from=3-1, to=3-4]
	\end{tikzcd}\]
	
The homotopy fixed point spectral sequence [HFPSS] is induced by an equivariant cellular filtration of $EG$. A cellular filtration of $X$ induces the Atiyah Hirzebruch spectral sequence [AHSS], as well as a filtration on the homotopy fixed points. This induces a filtration spectral sequence [FSS] on group cohomology that converges to the $E_2$-page of the homotopy fixed point spectral sequence.

	\begin{mysubsection}{Homotopy fixed point spectral sequence}
	The homotopy fixed point spectrum of a  $G$-spectrum $E$, which is denoted by $E^{hG}$, is defined as the $G$-fixed point spectrum of $F(EG_{+},E)$, where $EG$ is a contractible free $G$-space, and $F(*,*)$ denotes the function spectrum. A cellular filtration on $EG$ gives a tower of spectra whose homotopy limit is $E^{hG}=F(EG_{+},E)^{G}$. From this we obtain a spectral sequence, whose $2^{nd}$-page is independent of the choice of filtration. This spectral sequence converges conditionally to $({E^{hG}})^*(pt)$. 
	\begin{thm} \label{HFPSS}
For a topological space $X$, there is a conditionally convergent spectral sequence of the  form
%\begin{myeq}\label{hfpss} 
\[E_{2}^{s,t}=H^{s}(G;\pi_{t+s}\big(F(X,E)\big)) \implies \pi_t\big(F(X,E^{hG})\big).\]
%\end{myeq}
This spectral sequence satisfies the following properties \\ 
(1)  $d_r : E_{r}^{s,t}\to E_{r}^{s+r,t-1}$. \\ 
(2) The spectral sequence is multiplicative. \\
(3) The product on $E_2$-term is the usual cup product of group cohomology. 
	\end{thm}
	For a detailed proof we refer the reader to \cite[Theorem 6.1]{Dug03}.
	
\vspace*{0.2cm}

	As an example, we recall from \cite{HS14} the well-known computation of this spectral sequence for the $C_2$-equivariant spectrum $KU$, where $C_2$ is acting by complex conjugation. Since $KU^{hC_2}\simeq KO$, the spectral sequence takes the form 
	\[E_{2}^{s,t}=H^{s}(C_2;\pi_{t+s}(KU)) \implies \pi_t(KO).\]

The first step is to compute the action of $C_2$ on $\pi_\ast(KU)\cong \Z[\beta^{\pm}]$ with the Bott element $|\beta|=2$. Under the identification $\pi_{2}(KU) \cong \widetilde{KU}^0(S^2)$, $\beta= L-1$, where $L$ is the canonical line bundle over $\C P^1\cong S^2$. Writing $C_2=\langle \sigma \mid \sigma^2 =1 \rangle$, we directly calculate using the above identification that the $C_2$-action given by $\sigma(\beta)=-\beta$.  
		
Recall that the free $\Z[C_2]$-resolution $C_{*}$ of $\Z$ may be written as an algebra 
\[\Z[\sigma][u,v]/(\sigma^{2}-1,v^2), ~|u|=2,~|v|=1, ~d(u)=(1+\sigma)v,~ d(v)=1-\sigma.\]
The co-product on $C_\ast$ may be described as \cite[XII.7]{CE99}
\[\Delta(v)=v \otimes \sigma+ 1 \otimes v, \Delta(u)=u \otimes 1 + v \otimes \sigma v + 1 \otimes u.\]  
 The group cohomology $H^{*}(C_2;\pi_\ast KU)$ may be computed as the homology of the complex          $\Hom_{\Z[C_2]}(C_{*},\pi_\ast KU)$. The algebra structure comes from the co-product on $C_\ast$ and the algebra structure on $\pi_\ast KU$. The generating cocycles are 
  \[t= \beta^2 \in H^{0}(C_2;\pi_4KU), ~\eta=v\beta \in H^{1}(C_2;\pi_2 KU).\] 
 One then observes that 
 \begin{myeq} \label{gpcohku}
H^{s}(C_2;\pi_{t+s} KU)=\Z[t^{\pm},\eta]/(2\eta), ~ |t|=(0, 4), ~ |\eta|=(1,1).\end{myeq}
		The evident sparsity of this spectral sequence shows the first non-trivial differential is $d_3$ and it is shown in \cite{HS14} that $d_3(t)=\eta^{3},d_3(\eta)=0$. The multiplicative structure determines effect of $d_3$ completely. The spectral sequence collapses at the $E_4$-page for degree reasons, we have 
\[E_{\infty}^{*,*}=\Z[t^{\pm 2}]\{1,2t,\eta,\eta^{2}\}/(2\eta, 2 \eta^2).\] 
		Figure \ref{kohtpyfixpt} describes this spectral sequence.
	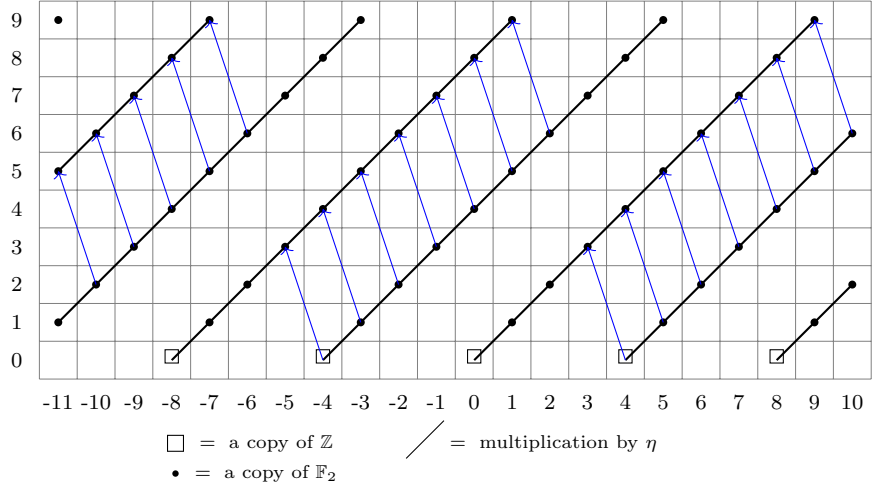
\begin{figure}
\centering	
	\begin{tikzpicture}
		\begin{scriptsize}
			% Define grid parameters
			\def\gridsize{10} % Number of squares per side
			\def\squaresize{0.5} % Size of each square in cm
			
			% Calculate total size
			\pgfmathsetmacro\totalsize{\gridsize*\squaresize}
			\pgfmathsetmacro\halfsize{\squaresize/2}
			% Draw the grid
			\draw[step=\squaresize cm,gray,very thin] (-6,0) grid (\totalsize,\totalsize);
			
			% Add coordinate labels (optional)
			\foreach \x in {-11,-10,...,10}
			\node at (\x*\squaresize-\halfsize,-0.3) {\x};
			\foreach \y in {0,1,...,9}
			\node at (-6.3,\y*\squaresize+\halfsize) {\y};

\node[draw, rectangle, minimum size=6pt, inner sep=0pt] at (-4.2,-.85) { };
\node[] at ( -3,-.85)[font=\tiny]{$= \mbox{ a copy of } \Z$};		

\fill(-4.2,-1.25)circle(1.2pt) { };
\node[] at ( -3,-1.25)[font=\tiny]{$= \mbox{ a copy of } \F_2$};		

%\fill(-1.17,-0.5)circle(1.2pt) { };
%\fill(-.65,-0.15)circle(1.2pt) { };
\draw[-,black] (-1.15,-1.05) -- (-.65,-.55); 
\node[] at ( 0.8,-.85)[font=\tiny]{$= \mbox{ multiplication by } \eta$};				

			\foreach \x in {-8,-4,0,4,8} {
				\node[draw, rectangle, minimum size=5pt, inner sep=0pt] at (\x*\squaresize-\halfsize,0.3) {};
			}
			
			\foreach \y in {1,2,3,4,5,6,7,8,9} {
				\fill (-4.25+\y*\squaresize,\y*\squaresize+\halfsize) circle(1.5pt) {};
				\fill (-2.25+\y*\squaresize,\y*\squaresize+\halfsize) circle(1.5pt) {};
				\fill (-.25+\y*\squaresize,\y*\squaresize+\halfsize) circle(1.5pt) {};
				\fill (-6.25+\y*\squaresize,\y*\squaresize+\halfsize) circle(1.5pt) {};
			}
			\foreach \y in {1,2,3,4,5,6} {
				\fill (1.75+\y*\squaresize,\y*\squaresize+\halfsize) circle(1.5pt) {};
			}
			
			\foreach \y in {1,2} {
				\fill (3.75+\y*\squaresize,\y*\squaresize+\halfsize) circle(1.5pt) {};
			}
			\foreach \y in {5,6,7,8,9} {				
				\fill (-8.25+\y*\squaresize,\y*\squaresize+\halfsize) circle(1.5pt) {};
			}
			\foreach \y in {9} {
				\fill (-10.25+\y*\squaresize,\y*\squaresize+\halfsize) circle(1.5pt) {};
			}
			
			\foreach \y in {0,4,8} {
				\draw[-, solid, thick] (-4.25,\y*\squaresize+\halfsize) to (-3.75,\y*\squaresize+\squaresize+\halfsize) ;
				\draw[-, solid, thick] (-2.25,\y*\squaresize+\halfsize) to (-1.75,\y*\squaresize+\squaresize+\halfsize) ;
				\draw[-, solid, thick] (-.25,\y*\squaresize+\halfsize) to (.25,\y*\squaresize+\squaresize+\halfsize) ;
				\draw[-, solid, thick] (1.75,\y*\squaresize+\halfsize) to (2.25,\y*\squaresize+\squaresize+\halfsize) ;
				\draw[-, solid, thick] (3.75,\y*\squaresize+\halfsize) to (4.25,\y*\squaresize+\squaresize+\halfsize) ;
			}
			
			\foreach \y in {1,5} {
				\draw[-, solid, thick] (-5.75,\y*\squaresize+\halfsize) to (-5.25,\y*\squaresize+\squaresize+\halfsize) ;
				\draw[-, solid, thick] (-3.75,\y*\squaresize+\halfsize) to (-3.25,\y*\squaresize+\squaresize+\halfsize) ;
				\draw[-, solid, thick] (-1.75,\y*\squaresize+\halfsize) to (-1.25,\y*\squaresize+\squaresize+\halfsize) ;
				\draw[-, solid, thick] (.25,\y*\squaresize+\halfsize) to (.75,\y*\squaresize+\squaresize+\halfsize) ;
				\draw[-, solid, thick] (2.25,\y*\squaresize+\halfsize) to (2.75,\y*\squaresize+\squaresize+\halfsize) ;
				\draw[-, solid, thick] (4.25,\y*\squaresize+\halfsize) to (4.75,\y*\squaresize+\squaresize+\halfsize) ;
			}
			
			\foreach \y in {2,6} {
				\draw[-, solid, thick] (-5.25,\y*\squaresize+\halfsize) to (-4.75,\y*\squaresize+\squaresize+\halfsize) ;
				\draw[-, solid, thick] (-3.25,\y*\squaresize+\halfsize) to (-2.75,\y*\squaresize+\squaresize+\halfsize) ;
				\draw[-, solid, thick] (-1.25,\y*\squaresize+\halfsize) to (-.75,\y*\squaresize+\squaresize+\halfsize) ;
				\draw[-, solid, thick] (.75,\y*\squaresize+\halfsize) to (1.25,\y*\squaresize+\squaresize+\halfsize) ;
				\draw[-, solid, thick] (2.75,\y*\squaresize+\halfsize) to (3.25,\y*\squaresize+\squaresize+\halfsize) ;
				%\draw[-, solid, thick] (4.25,\y*\squaresize+\halfsize) to (4.75,\y*\squaresize+\squaresize+\halfsize) ;
			}
			\foreach \y in {3,7} {
				\draw[-, solid, thick] (-4.75,\y*\squaresize+\halfsize) to (-4.25,\y*\squaresize+\squaresize+\halfsize) ;
				\draw[-, solid, thick] (-2.75,\y*\squaresize+\halfsize) to (-2.25,\y*\squaresize+\squaresize+\halfsize) ;
				\draw[-, solid, thick] (-.75,\y*\squaresize+\halfsize) to (-.25,\y*\squaresize+\squaresize+\halfsize) ;
				\draw[-, solid, thick] (1.25,\y*\squaresize+\halfsize) to (1.75,\y*\squaresize+\squaresize+\halfsize) ;
				\draw[-, solid, thick] (3.25,\y*\squaresize+\halfsize) to (3.75,\y*\squaresize+\squaresize+\halfsize) ;
				%\draw[-, solid, thick] (4.25,\y*\squaresize+\halfsize) to (4.75,\y*\squaresize+\squaresize+\halfsize) ;
			}
			
			\foreach \y in {0,1,2,3,4,5,6} {
				\draw[->,blue] (1.75+\y*\squaresize,\y*\squaresize+\halfsize) -- (1.25+\y*\squaresize,1.5+\y*\squaresize+\halfsize); 
			}
			\foreach \y in {0,1,2,3,4,5,6} {
				\draw[->,blue] (-2.25+\y*\squaresize,\y*\squaresize+\halfsize) -- (-2.75+\y*\squaresize,1.5+\y*\squaresize+\halfsize); 
			}
			\foreach \y in {2,3,4,5,6} {
				\draw[->,blue] (-6.25+\y*\squaresize,\y*\squaresize+\halfsize) -- (-6.75+\y*\squaresize,1.5+\y*\squaresize+\halfsize); 
			}
			
		\end{scriptsize}
	\end{tikzpicture}
\caption{ The homotopy fixed point spectral sequence of $KO$. The blue arrows describe $d_3$.}
\label{kohtpyfixpt}
\end{figure}
	
	\end{mysubsection}

\vspace*{0.2cm}	

	\begin{mysubsection}{Filtration spectral sequence}
Let $E$ be a ring $G$-spectrum (which is a ring spectrum such that $G$ acts by algebra maps)  and $X$ be a CW complex with skeleta $X^{(k)}$. We observe that the homotopy fixed point spectral sequence gets an induced filtration, and yields multiplicative trigraded spectral sequences which converge to the pages of the homotopy fixed point spectral sequence. For a fixed cellular filtration of $EG= \bigcup EG^{(s)}$, we have the bifiltered object 
\[ Y_{s,r} = F^G((EG^{(s)} \times X^{(r)})_+, E),\]
and the skeletal inclusions give two maps 
\[i_{s,r} : Y_{s,r} \to Y_{s-1,r}, ~ j_{s,r} : Y_{s,r} \to Y_{s,r-1},\]
with the property 
%\begin{myeq}\label{ijcomm}
\[j_{s-1,r}\circ i_{s,r} = i_{s,r-1}\circ j_{s,r}.\]
%\end{myeq}
 We denote the fibers of the above maps as follows 
\[ C_{s,r}^{(1)} = \mbox{Fib}(i_{s,r})  \simeq F^G(EG^{(s)}/EG^{(s-1)} \wedge X^{(r)}_+,E),\]
\[ C_{s,r}^{(2)} = \mbox{Fib}(j_{s,r})  \simeq F^G(EG^{(s)}_+ \wedge X^{(r)}/X^{(r-1)},E),\]
 and 
 \begin{align*}
  C_{s,r} &= \mbox{Fib}(i_{s,r}:C_{s,r}^{(2)} \to C_{s-1,r}^{(2)}) \\
        &\simeq \mbox{Fib}(j_{s,r} : C_{s,r}^{(1)} \to C_{s,r-1}^{(1)}) \\
        & \simeq F^G(EG^{(s)}/EG^{(s-1)} \wedge X^{(r)}/X^{(r-1)},E).
 \end{align*}
 We then obtain a square of spectral sequences  
 \begin{myeq} \label{fsse1sq}
 \begin{tikzcd}
		E_1^{s,t,r} = \pi_{t}(C_{s,r}) &&& E_1^{t,r} = \pi_{t}F(\frac{X^{(r)}}{X^{(r-1)}} ,E^{hG}) \\
		\\
		E_1^{s,t} = \pi_tF^G(\frac{EG^{(s)}}{EG^{(s-1)}}\wedge X_+,E) &&&  \pi_{t}F(X_+,E^{hG})
		\arrow["\HFPSS(X^{(r)}/X^{(r-1)})",Rightarrow, from=1-1, to=1-4]
		\arrow["\AHSS( E)", Rightarrow, from=1-1, to=3-1]
		\arrow["\AHSS (E^{hG})", Rightarrow, from=1-4, to=3-4]
		\arrow["\HFPSS(X_+)"', Rightarrow, from=3-1, to=3-4]
	\end{tikzcd}
	\end{myeq}
 The top row and the bottom row of \eqref{fsse1sq} are suitable homotopy fixed point spectral sequences. 
 \begin{thm}\label{FSSE1}
		The spectral sequences in \eqref{fsse1sq} are conditionally convergent and multiplicative %have the following properties \\
	\end{thm}
	
	\begin{proof}
	The multiplicative structure on the spectral sequences follow from \cite{Dug03}.  
	\end{proof}

\vspace*{0.2cm}
 
 We now upgrade the square of spectral sequences \eqref{fsse1sq} to the $E_2$-page. Let $\EE_2^{s,t,r}$ be defined by 
 \[ \EE_2^{s,t,r}= \frac{\Ker\Big(d_1^{\HFPSS} : \pi_t(C_{s,r}) \to \pi_{t-1}(C_{s+1,r})\Big)}{\Im\Big(d_1^{\HFPSS} : \pi_{t+1}(C_{s-1,r})\to \pi_t(C_{s,r})\Big)}\cong H^s\Big(G; \pi_{t+s}\big(F({X^{(r)}}/{X^{(r-1)}},E)\big)\Big).\]
 This is the $E_2$-page of the spectral sequence HFPSS$(X^{(r)}/X^{(r-1)})$. We also write 
 \begin{align*}
 E_2^{s,t} &= \frac{\Ker\Big(d_1^{\HFPSS} : \pi_t(F^G(\frac{EG^{(s)}}{EG^{(s-1)}}\wedge X_+,E)) \to \pi_{t-1}(F^G(\frac{EG^{(s+1)}}{EG^{(s)}}\wedge X_+,E))\Big)}{\Im\Big(d_1^{\HFPSS} : \pi_{t+1}(F^G(\frac{EG^{(s-1)}}{EG^{(s-2)}}\wedge X_+,E))\to \pi_t(F^G(\frac{EG^{(s)}}{EG^{(s-1)}}\wedge X_+,E))\Big)} \\
   & \cong H^s\Big(G; \pi_{t+s}\big(F(X_+,E)\big)\Big),
 \end{align*}
 which is the $E_2$-page of the spectral sequence HFPSS$(X_+)$. We now assume that $E$ is an even ring spectrum (that is, the homotopy groups of $E$ are concentrated in even degrees) and that the cellular filtration of $X$ has cells only in even degrees. Under these hypotheses, we show that Theorem \ref{FSSE1} implies that AHSS$(E)$ induces a spectral sequence 
 \[\EE_2^{s,t,r} \implies \EE_2^{s,t}.\]
 This will be called the filtration spectral sequence (FSS). 
	\begin{thm}\label{FSS}
Suppose that $E$ is an even ring spectrum, and that $X$ has cells only in even degrees. In such a situation, there is a multiplicative conditionally convergent spectral sequence  
$$ \mbox{FSS} : \EE_2^{s,t,r} = H^{s}(G;E^{-t-s}(X^{(r)}/X^{(r-1)})) \implies H^{s}(G;E^{-t-s}(X))$$
 with differential $d_m: \EE_{m}^{s,t,r} \rightarrow \EE_{m}^{s+1,t-1,r+m}$. 
	\end{thm}
	
\begin{proof}
As $E$ is an even spectrum, and $X$ is a CW complex with only even cells, the spectral sequence AHSS$(E)$ in \eqref{fsse1sq} degenerates at the $E^1$-page. The same criterion implies that the sequence
\[ 0 \to E^\ast(X^{(r)}/X^{(r-1)}) \to E^\ast(X^{(r)}) \to E^\ast(X^{(r-1)})\to 0\]
is a short exact sequence of $\Z[G]$-modules. One obtains long exact sequences on group cohomology leading to the exact couple
\[\xymatrix{ \bigoplus_{s,r} H^s(G;E^\ast(X^{(r)}) \ar[rr] &&   \bigoplus_{s,r} H^s(G;E^\ast(X^{(r)})  \ar[ld] \\ 
& \bigoplus_{s,r} H^s(G;E^\ast(X^{(r)}/X^{(r-1)}). \ar[lu]}\] 
This gives us the conditionally convergent spectral sequence 
\[\EE_2^{s,t,r} = H^{s}(G;E^{-t-s}(X^{(r)}/X^{(r-1)})) \implies H^{s}(G;E^{-t-s}(X)).\]
As the differentials in AHSS$(E)$ are $0$, this spectral sequence is obtained from the differential graded algebra $C^\ast(G; E^\ast(X))$ that is filtered by the inclusions $E^\ast(X/X^{(r)}) \to E^\ast(X)$ as in \cite[Theorem 6.1]{Dug03}. The differential equals the differential on the cochains  in group cohomology so it follows that this is a filtered differential graded algebra in the sense of \cite[Theorem 2.14]{Mcc01}. It follows that this spectral sequence is multiplicative. 
\end{proof}	
	
\vspace*{0.2cm}

Recall that for a $G$-module $M$, 
\[ H^0(G;M) \cong M^{G} = \{ m\in M \mid g.m=m ~\forall~ g\in G\}.\]
We also fix the filtration $EG= \cup_{s\geq 0} EG^{(s)}$ such that $EG^{(0)}$ is the discrete space $G$, and $EG^{(1)}/EG^{(0)}$ is a free $G$-CW-complex whose underlying space is a wedge of circles.
The condition that $X$ has only even cells implies that among the maps 
\[[S^t \wedge EG^{(1)}_+\wedge X,E]^G \to [S^t \wedge X,E],\quad [S^t  \wedge \frac{EG^{(1)}}{EG^{(0)}}\wedge  X, E]^G \xrightarrow{j_1} [S^t \wedge EG^{(1)}_+\wedge X,E]^G,\]
one is an isomorphism, and the other is $0$.  
\begin{prop} \label{diffl-att}
 Suppose $\alpha\in E_1^{t,r}$ is represented by $a$ in $\EE_2^{0,t,r}$ in HFPSS($X^{(r)}/X^{(r-1)})$, and that $\beta\in E_1^{t-1,r+m}$ is represented by $b$ in $\EE_2^{1,t-1,r+m}$ in HFPSS($X^{(r+m)}/X^{(r+m-1)})$. Assume that $d_i(\alpha)=0$ for $i<m$,  $d_m(\alpha)=\beta$ in the spectral sequence AHSS($E^{hG}$), and $d_i(a)=0$ for $i<m$ in the spectral sequence FSS. Then,
	 $d_m(a)=b$ in the spectral sequence FSS.
\end{prop}
\begin{proof}
Recall that $H^0(G; E^\ast(Y)) \cong E^\ast(Y)^G$. The fact that $\alpha \in E_1^{t,r}$ is represented by $a\in \EE_2^{0,t,r}= E^{-t}(X^{(r)}/X^{(r-1)})^G$ means that  the map 
$ \phi : E^{hG} \to E$ sends $\alpha$ to $a$. The fact that $E$ is even, and that $X$ has only even dimensional cells implies that $t$ is even (otherwise both $a$ and $b$ are forced to be $0$).  The formula $d_m(\alpha) = \beta$ means that under 
\[ {E^{hG}}^{-t}(X^{(r)}/X^{(r-1)}) \rightarrow {E^{hG}}^{-t}(X^{(r)}) \leftarrow {E^{hG}}^{-t}(X^{(r+m)})\rightarrow {E^{hG}}^{-t+1}(X^{(r+m+1)}/X^{(r+m)})\]
there is a class $\hat{\beta}\in {E^{hG}}^{-t}(X^{(r+m)})$ which maps to $\beta \in {E^{hG}}^{-t+1}(X^{(r+m+1)}/X^{(r+m)})$   and maps to the image of $\alpha$ in ${E^{hG}}^{-t}(X^{(r)}) $.  The fact that $\beta$ is represented by $b\in \EE_2^{1,t-1,r+m}$ means that the map $\phi$ sends $\beta$ to $0$. Denote the map $E^{hG} \to F^G(EG^{(1)}_+,E)$ by $\psi$. This means that $\psi(\beta)=j_1(b)$. Consider the map 
$\partial :\Sigma^{-1} X^{(r+m+1)}/X^{(r+m)}\to X^{(r+m)}.$
We observe that in diagram \eqref{diffl}, $\overline{\phi(\hat{\beta})}=b$. 
\begin{myeq}\label{diffl}
\scalebox{0.68}{$\xymatrix{ S^t \wedge \big(\Sigma^{-1} \frac{X^{(r+m+1)}}{X^{(r+m)}}\big) \wedge \big(\Sigma^{-1} \frac{EG^{(1)}}{EG^{(0)}}\big) \ar[r] \ar[d] &S^t \wedge X^{(r+m)}_+ \wedge \big(\Sigma^{-1} \frac{EG^{(1)}}{EG^{(0)}}\big) \ar[r] \ar[d]  & S^t \wedge X^{(r+m+1)}_+ \wedge \big(\Sigma^{-1} \frac{EG^{(1)}}{EG^{(0)}}\big) \ar[r] \ar[d]  & S^t \wedge  \frac{X^{(r+m+1)}}{X^{(r+m)}} \wedge \big(\Sigma^{-1} \frac{EG^{(1)}}{EG^{(0)}}\big)  \ar@{-->}[ddd]^{\overline{\phi(\hat{\beta})}} \\
S^t \wedge \big(\Sigma^{-1} \frac{X^{(r+m+1)}}{X^{(r+m)}}\big) \wedge EG^{(0)}_+\ar[d] \ar[r] & S^t \wedge X^{(r+m)}_+ \wedge  EG^{(0)}_+  \ar[r] \ar[d] \ar@/^2.25pc/[rrdd]^-{\phi(\hat{\beta})} & S^t \wedge X^{(r+m+1)}_+ \wedge  EG^{(0)}_+ \ar[d] \ar@/^2pc/[rdd] \\
S^t \wedge \big(\Sigma^{-1} \frac{X^{(r+m+1)}}{X^{(r+m)}}\big) \wedge EG^{(1)}_+\ar[d] \ar[r] & S^t \wedge X^{(r+m)}_+ \wedge  EG^{(1)}_+  \ar[r]\ar[rrd]^{\psi(\hat{\beta})} & S^t \wedge X^{(r+m+1)}_+ \wedge  EG^{(1)}_+\\ 
S^t \wedge \big( \frac{X^{(r+m+1)}}{X^{(r+m)}}\big) \wedge \big(\Sigma^{-1}\frac{EG^{(1)}}{EG^{(0)}}\big) \ar@{-->}[rrr]_{\psi(\partial(\hat{\beta}))=b} &&& E.  } $} 
\end{myeq}
As $E$ is even, and $X$ has only even cells, we have 
\[ [S^t \wedge \Sigma^{-1} \frac{X^{(r+m+1)}}{X^{(r+m)}} \wedge EG^{(1)}_+,E]^{G} \cong [S^t \wedge  \frac{X^{(r+m+1)}}{X^{(r+m)}} \wedge \Sigma^{-1}\frac{EG^{(1)}}{EG^{(0)}}, E]^G.\]
 The bottom row of \eqref{diffl} writes $\psi(\partial(\hat{\beta}))$ as the obstruction to extending $\psi(\hat{\beta})$ to $S^t\wedge X^{(r+m+1)}_+ \wedge EG^{(1)}_+$. On the other hand, the same obstruction is realized as the right vertical arrow of the diagram, implying
\[\delta(\phi(\hat{\beta}))= \psi(\partial(\hat{\beta}))=\partial (\psi(\hat{\beta}))=b.\]
 Thus, in the sequence of maps
\[\scalebox{0.85}{$H^0(G; E^{-t}(X^{(r)}/X^{(r-1)})) \to H^0(G; E^{-t}(X^{(r)})) \leftarrow H^0(G; E^{-t}(X^{(r+m)})) \xrightarrow{\delta} H^1(G; E^{-t}(X^{(r+m+1)}/X^{(r+m)}))$},\]
 $\phi(\hat{\beta})\in H^0(G; E^{-t}(X^{(r+m)})$ maps to $b \in H^1(G; E^{-t}(X^{(r+m+1)}/X^{(r+m)}))$   and maps to the image of $a\in H^0(G; E^{-t}(X^{(r)}/X^{(r-1)})) $ in $H^0(G; E^{-t}(X^{(r)}))  $. Hence, $d_m(a) = b$.
\end{proof}

\vspace*{0.2cm}

We now note that the filtration spectral sequence FSS is obtained via a filtration of differential graded algebras. We write $H^\ast(G;E^\ast(X))$ as the cohomology of the DGA $C^\ast(G;E^\ast(X))$. Under the hypothesis that $E$ is even, and $X$ is a CW complex with cells in even degrees, we have 
\[ C^\ast(G; E^\ast(X)) \cong \varprojlim_m C^\ast(G;E^\ast(X^{(m)})),\]
which a limit of DGAs. This means that the results of \cite[\S 4]{May69} (see also \cite[A1.4.10 - A1.4.12]{Rav86}) applies to this situation. 
	\end{mysubsection}
	
\vspace*{0.2cm}

	\begin{mysubsection}{The filtration spectral sequence for the $K$-theory of $\C P^m$}
	We apply  Theorem \ref{FSS} to the case where $E=KU$ and $X=\mathbb{C}P^m_{+}$. The structure of $KU^*(\C P^m)$ may be understood from the fact that $KU$ is a complex oriented cohomology theory.	Recall that such a theory $E$ is given by a homotopy commutative ring spectrum $E$ with an orientation class $x$ in ${\tilde{E^{2}}}(\C P^{\infty})$ such that when restricted to $\mathbb{C}P^1$, it becomes a generator of free rank one ${\pi_{0}}(E)$-module $\tilde{E^2}(\mathbb{C}P^1)=\tilde{E^2}(S^2)$. Then, we have the following \cite{Ada74},
	\[E^{*}(\mathbb{C}P^{\infty})=E^{*}(pt)[[x]],\quad E^{*}(\mathbb{C}P^{m})=E^{*}(pt)[x]/(x^{m+1}) .\]
	For $KU$, a choice of the  orientation class is given by $x=(L-1)\beta^{-1}$, where $L$ denotes the canonical line bundle over $\C P^{\infty}$. The filtration spectral sequence
\[\EE_2^{s,t,r} = H^{s}(C_2;KU^{-t-s}((\C P^m)^{(r)}/(\C P^m)^{(r-1)})) \implies H^{s}(C_2;KU^{-t-s}(\C P^m))\]
satisfies 
\[\EE_2^{s,t,r} = \begin{cases} 
H^{s}(C_2;KU^{-t-s}(S^{2k})) & \mbox{ if } r=2k \mbox{ for } k\leq n\\
0 & \mbox{ if } r \mbox{ is odd or } > 2n. \end{cases} \]
Using the multiplicative structure and the formula \eqref{gpcohku}, we have the isomorphism
\[ H^{s}(C_2;KU^{-t-s}(S^{2k})) \cong \Sigma^{-2k} H^{s}(C_2;KU^{-t-s}(\pt))\{x^k\} \cong \Z[t^{\pm}, \eta]/(2\eta)\{x^k\}.\]
Thus we have 
\[\EE_2^{s,t,r}\cong \Z[t^{\pm},\eta,x]/(2\eta, x^{n+1}), ~ |t|=(0,4,0), ~ |\eta|= (1,1,0), ~ |x|= (0,-2,2).\]
	
We now compute $d(x)$ in the filtration spectral sequence. Note that $\C P^2 \simeq \Cone(\eta : S^3 \to S^2)$, and the Hopf map $\eta$ is detected in $\pi_1 KO$. This implies by Proposition \ref{diffl-att} that $d_2(x) = \eta x^2$ in the filtration spectral sequence, and we have the following Proposition using the multiplicative structure.
\begin{prop}\label{diffl-fss-ku}
	In the filtration spectral sequence, 
		$d_2(a x^k)=\eta a x^{k+1} $ when $k$ is odd, $d_2(a x^k)=0$ when $k$ is even.
	\end{prop}

\vspace*{0.2cm}

\begin{exam}
There is an alternate method to make the same computation. In this case we may directly compute the $C_2$-action on  $KU^{*}(\C P^{m})$. Write $C_2=\langle \sigma\mid \sigma^2= 1\rangle$. Then $\sigma$ sends $L$ to its conjugate bundle $\bar{L}$. Therefore, we have
	%\begin{align*}
	\[	\sigma (x) =(-\beta^{-1})(\bar{L}-1)	%&=\beta^{-1}(1-\frac{1}{L})\\
		=\beta^{-1}(\frac{L-1}{L})
		=x(\frac{1}{L-1+1})
	%	&=x[1-(L-1)+(L-1^2)\cdots]\\
		=\sum_{i\geq 0}(-1)^{i}(\beta)^{i}(x)^{i+1}. \]
	%\end{align*}
The differentials are computed by applying $\delta$ given by the formula $\delta(\tau)=\sigma( \tau) - \tau$ and looking at the first skeleton where the image is non-trivial. Doing this process for $x$ yields 
\[ \delta(x) = \sum_{i\geq 1} (-1)^i 	\beta^i x^{i+1},\]
and so it follows that $d_2(x)=\eta x^2$.
\end{exam}

\vspace*{0.2cm}

	Now we use Proposition \ref{diffl-fss-ku} to determine the $E_2$-page of homotopy fixed point spectral sequence. 
	Note that under $d_2$ all  $\eta$-multiples of $x^k$ either belong to the image of $d_2$ (if $k$ is even) or map non-trivially under $d_2$ (if $k$ is odd), except for the case $k=m$ and $m$ is odd. Thus, only $t^{ s} x^{2j}$ and $2 t^{ s} x^{2j-1}$ survives to $\EE_3$ among multiples of $x^k$ for $0<k<m$.  In summary the $\EE_3$-page of the filtration sequence has the following form : 
	\[\EE_{3}= \begin{cases}  \Z[t^\pm,\eta]/(2\eta) \oplus \Z[2t^{ s}x^{2k-1}] \oplus \Z[t^{ s} x^{2k}] & k\leq \frac{m}{2}, \mbox{ if } m \mbox{ is even}\\
 \Z[t^\pm,\eta]/(2\eta)\{1,x^m\} \oplus \Z[2t^{ s}x^{2k-1}] \oplus \Z[t^{ s} x^{2k}] & k\leq \frac{m}{2}, \mbox{ if } m \mbox{ is odd}. \end{cases}  \] 
Note for degree reasons, there are no further differentials if $m$ is even as other than multiples of $x^0$, the non-trivial elements lie in the $s=0$ line. If $m$ is odd, by degree reasons, the only possible differentials must be of the form $d(x^j)=\eta^r t^s x^m$ if $j$ is even or $d(2x^j)= \eta^r t^s x^m$ if $j$ is odd. 
In this case, let $i: \mathbb{C}P^m \rightarrow \mathbb{C}P^{m+1}$ denote the inclusion,  and note that $i$ induces a map between fitration spectral sequences,
	\[\EE_r(\C P^{m+1}) \rightarrow \EE_r(\C P^m).\] 
The map takes the class $x$ to $x$, and also preserves the multiplicative structure. Further,	from the computation of the differentials in the even case,  we see that the classes $x^j$ ($j$ even) and $2x_j$ ($j$ odd) are images of permanent cycles. Therefore, even in this case, there are no more non-trivial differentials.  
In summary, we have the following result.	
\begin{prop} \label{einffss-ko}
		The $E_{\infty}$-page of the filtration spectral sequence for the $C_2$-action on the $K$-theory of  $\C P^m$ is given by 
	\[\EE_\infty (K(\C P^m)) = \begin{cases}  \Z[t^\pm,\eta]/(2\eta) \oplus \Z[2t^{ s}x^{2k-1}] \oplus \Z[t^{ s} x^{2k}] & k\leq \frac{m}{2}, \mbox{ if } m \mbox{ is even}\\
 \Z[t^\pm,\eta]/(2\eta)\{1,x^m\} \oplus \Z[2t^{ s}x^{2k-1}] \oplus \Z[t^{ s} x^{2k}] & k\leq \frac{m}{2}, \mbox{ if } m \mbox{ is odd}. \end{cases}  \] 
	\end{prop}
	\end{mysubsection}

\vspace*{0.2cm}
	
	\begin{mysubsection}{Calculation of $KO^{*}(\mathbb{C}P^{m})$}
	The calculation of filtration spectral sequence in Proposition \ref{einffss-ko} provides us the $E_2$-page of the homotopy fixed point spectral sequence. In this section we will compute all the differentials of the homotopy fixed point spectral sequence finally computing $KO^{*}(\C P^m)$. We use the following notation 
\[ y \in E_2^{0,-4} \mbox{ denotes the class } x^2 \mbox{ in } \EE_\infty, \mbox{ and satisfies } y^{\lfloor \frac{m}{2}\rfloor+1}=0.\]
\[ u \in E_2^{0,-2} \mbox{ denotes the class } 2x \mbox{ in } \EE_\infty., \mbox{ and satisfies } u^2=4y.\]
\[ z \in E_2^{0,-2m} \mbox{ denotes the class } x^m \mbox{ in } \EE_\infty \mbox{ if } m \mbox{ is odd}, \mbox{ and satisfies } 2z= y^{\lfloor \frac{m}{2}\rfloor}u.\]
Now we have the following formula for the $E_2$-page of the homotopy fixed point spectral sequence.
\begin{prop} \label{e2ss-ko}
		The $E_{2}$-page of the homotopy fixed point  spectral sequence for the $C_2$-action on the $K$-theory of  $\C P^m$ is given by 
	\[E_2 (K(\C P^m)) = \begin{cases}  \frac{\Z[t^\pm,\eta,y,u]}{(2\eta,y^{\lfloor \frac{m}{2}\rfloor+1},\eta y, \eta u, u^2-4y)} , &\mbox{ if } m \mbox{ is even}\\
\frac{ \Z[t^\pm,\eta,y,u,z]}{(2\eta,y^{\lfloor \frac{m}{2}\rfloor+1},\eta y, \eta u, u^2-4y, 2z-y^{\lfloor\frac{m}{2}\rfloor} u)}, &\mbox{ if } m \mbox{ is odd}. \end{cases}  \] 
	\end{prop}

\vspace*{0.2cm}

As in the homotopy fixed point spectral sequence for $K$-theory, we have $d_3(t)= \eta^3$. This means $2t$ and $t^2$ survives to $E_4$. Let $\alpha$ denote the class $2t$. Therefore, we have the following formula for the $E_4$-page
\[\scalebox{0.9}{$E_4(K(\C P^m)) \cong 	 \begin{cases}  \frac{\Z[t^{\pm 2},\alpha, \eta]}{(\alpha^2 - 4 t^2, \alpha\eta, 2\eta, \eta^3)} \oplus \Z\{t^s y^k, t^sy^j u\} \mbox{ for } 1\leq k \leq \lfloor\frac{m}{2}\rfloor, ~ 1\leq j <\lfloor\frac{m}{2}\rfloor,&\mbox{ if } m \mbox{ is even}\\
 \frac{\Z[t^{\pm 2},\alpha,\eta]}{(\alpha^2-4t^2, \alpha\eta, \eta^3, 2\eta)}\{1,z\} \oplus \Z\{t^s y^k, t^sy^j u\} \mbox{ for }  1\leq k \leq \lfloor\frac{m}{2}\rfloor,~1\leq j < \lfloor\frac{m}{2}\rfloor,  &\mbox{ if } m \mbox{ is odd}. \end{cases}$}  \] 
Now all the non-trivial classes in this spectral sequence  lies in  $E_4^{\leq 2,\ast}$. This means that there are no further non-zero differentials and $E_4=E_\infty$. We may now read off the computation of $KO^\ast(\C P^m)$ from the $E_\infty$-page. 
\begin{thm}\label{compko}
The real $K$-theory of $\C P^m$ is given by 
\[\scalebox{0.9}{$KO^\ast(\C P^m)) \cong 	 \begin{cases}  \frac{\Z[t^{\pm 2},\alpha, \eta]}{(\alpha^2 - 4 t^2,\alpha\eta, 2\eta, \eta^3)} \oplus \Z\{t^s y^k, t^sy^j u\} \mbox{ for } 1\leq k \leq \lfloor\frac{m}{2}\rfloor, ~ 1\leq j <\lfloor\frac{m}{2}\rfloor,&\mbox{ if } m \mbox{ is even}\\
 \frac{\Z[t^{\pm 2},\alpha,\eta]}{(\alpha^2-4t^2,\alpha\eta, \eta^3, 2\eta)}\{1,z\} \oplus \Z\{t^s y^k, t^sy^j u\} \mbox{ for }  1\leq k \leq \lfloor\frac{m}{2}\rfloor,~1\leq j < \lfloor\frac{m}{2}\rfloor,  &\mbox{ if } m \mbox{ is odd}. \end{cases}$}  \] 
The ring $\frac{\Z[t^{\pm 2},\alpha, \eta]}{(\alpha^2 - 4 t^2,\alpha\eta, 2\eta, \eta^3)} \cong KO^\ast(\pt)$, and for $1\leq k \leq \lfloor \frac{m}{2}\rfloor$, $\oplus_{s\in \Z}\Z\{ t^s y^k, t^l u y^{k-1}\}$ is a $KO^\ast(\pt)$-submodule isomorphic to $K^\ast(\pt)$. 
\end{thm}
\end{mysubsection}

	\vspace*{0.2cm}
	
	\section{Higher Real-K theory of complex projective spaces at the prime $3$} \label{higherkth3}

Our main computations involve the $E$-cohomology of $\C P^m$ where $E$ is the homotopy fixed points of a finite group acting on the Morava $E$-theory spectrum. In this section, we compute this at the prime $3$ which is slightly easier than the computation at $p\geq 5$. 

\begin{mysubsection}{The homotopy groups of $EO_{p-1}$} \label{eop-1} Recall that the coefficient ring for the Morava $E$-theory $E_n$ is  
\[{E_n}_\ast \cong \W(\F_{p^n})[[u_1,\cdots, u_{n-1}]][u^{\pm}], \quad |u_i|=0,\quad |u|=2,\]
where $\W(\F_{p^n})$ is the ring of $p$-typical Witt vectors in $\F_{p^n}$ \cite{Ser79}. This is an $\E_\infty$-ring spectrum \cite{Rez98, GH04}, and possesses an action of the (extended) Morava stabilizer group $\G_n = \Sp_n\rtimes Gal(\F_{p^n}/\F_p)$ by $\E_\infty$-maps. We restrict our attention to $n=p-1$ (and $p\geq 3$), and the group $G$ as the maximal finite subgroup of $\G_n$ containing an element of order $p$. One notes that $G$ is of order $p.(p-1)^3$ (see \cite{Hew95} for details on finite subgroups of $\Sp_n$), and it fits into an extension (see \cite[Proposition 20]{Hen07}) 
\[1\to C_p\rtimes C_{p-1}^2 \to G \to Gal(\F_{p^{p-1}}/\F_p)\to 1.\]
From now on, we refer to $EO_{p-1}$ as the homotopy fixed point spectrum $E_{p-1}^{hG}$, whose homotopy groups are computed via the homotopy fixed point spectral sequence. The $E_2$-term was computed by Hopkins and Miller (in unpublished literature). The formula for the $G$-action on ${E_{p-1}}_\ast$ may be understood from \cite{DH04}. Published accounts of computations of the $E_2$-page  are available in \cite{Nav99, Nav10} (see also \cite{HMS17}, and \cite{GHMR05} for $p=3$). 
\begin{theorem} (Hopkins-Miller) \label{hmcomp}
There is an exact sequence
\[ {E_{p-1}}_\ast \xrightarrow{\tr} H^\ast(G;{E_{p-1}}_\ast) \to \F_p[\alpha,\beta,\Delta^{\pm 1}]/(\alpha^2),\]
with $|\Delta|=(0,2pn^2)$, $|\alpha|=(1, 2p-3)$, $|\beta|= (2,2p(p-1)-2)$.
\end{theorem}

\vspace*{0.2cm}

There is a map from the the $E_2$-page of the Adams Novikov spectral sequence to homotopy fixed point spectral sequence, 
\[{{\Ext}}^{s,t}_{BP_*BP}(BP_{*},BP^{*}(X)) \rightarrow H^{s}(G;({E_{p-1}})^{t}(X)).\]  
When $X$ is the sphere spectrum this map sends $\alpha_{1}$  to $\alpha$, $\beta_1$  to $\beta$ and lastly $\beta_{p/p}$  to $\Delta \beta$ (see \cite{Rav78} or \cite[Chapter 6, \S 4]{Rav86}. It follows that $\alpha, \beta$ are permanent cycles as $\alpha_{1},\beta_1$ are so in the Adams Novikov spectral sequence. The differentials in the homotopy fixed point spectral sequence may then be computed as (see \cite[Chapter 7]{Nav99})
\begin{myeq}\label{diffhfpsseo}
d_{2p-1}(\Delta)= \alpha \beta^{p-1}, \quad d_{2(p-1)^2+1}(\Delta^{p-1} \alpha)= \beta^{(p-1)^2 +1}. 
\end{myeq}
The formula \eqref{diffhfpsseo} together with the fact that all the differentials are $0$ on $M_p= \Im(\tr)$, implies that there are no further differentials due to degree reasons. This gives us the $E_\infty$-page for the homotopy fixed point spectral sequence \cite[Chapter 7]{Nav99}. 
\begin{myeq}\label{einftyhfpsseo}
E_\infty^{\even, \ast} \cong \frac{\F_p[\Delta^{\pm p}][\beta]}{(\beta^{(p-1)^2+1})} \pmod{M_p}, \quad E_\infty^{\odd, \ast} \cong \frac{\alpha \F_p[\Delta^{\pm p}]\{1,\Delta,\dots, \Delta^{p-2}\}[\beta]}{(\beta^{p-1})}.
\end{myeq}
In the remainder of the section we restrict to $p=3$, so that in this case 
\[|\alpha|=(1,3), \quad |\beta|=(2,10),\quad |\Delta|=(0,24),\] 
 the differentials are given by (see Figure \ref{eo2htpyfixpt})
\begin{myeq}\label{dhfpsseo2}
d_5(\Delta)= \alpha \beta^2, \quad d_9(\Delta^2 \alpha)= \beta^5,
\end{myeq}
and the $E_\infty$-page is given by
 \begin{myeq} \label{eo2htpygp}
\begin{aligned}
&E_\infty^{\even, \ast} \cong \frac{\F_3[\Delta^{\pm 3}][\beta]}{(\beta^5)} \cong \F_3[\Delta^{\pm 3}]\{1,\beta, \beta^2,\beta^3,\beta^4\} \pmod{M_3}, \\
& E_\infty^{\odd, \ast} \cong \frac{\F_3[\Delta^{\pm 3}]\{\alpha,b\}[\beta]}{(\beta^2)} \cong \F_3[\Delta^{\pm 3}]\{\alpha, b, \alpha \beta, b \beta\},
\end{aligned}
\end{myeq}
where $b$ denotes the class $[\alpha\Delta]$. There is a multiplicative extension $\alpha b = \beta^3$. This may be deduced from the Toda bracket formula $\beta=\langle \alpha, \alpha, \alpha \rangle$ via the juggling formulas \cite[A1.4.6]{Rav86}.
\[\alpha b = \alpha \langle \alpha, \alpha, \beta^2 \rangle  =  \langle \alpha, \alpha, \alpha \rangle \beta^2 = \beta^3. \]
\begin{figure}
\centering
\begin{tikzpicture}
		\begin{scriptsize}
			% Define grid parameters
			\def\gridsize{27} % Number of squares per side
			\def\squaresize{0.35} % Size of each square in cm
			
			% Calculate total size
			\pgfmathsetmacro\totalsize{\gridsize*\squaresize}
                    
			\pgfmathsetmacro\halfsize{\squaresize/2}
			% Draw the grid
			\draw[step=\squaresize cm,gray,very thin] (-6,0) grid (\totalsize,14*\squaresize);
			
			% Add coordinate labels (optional)
%			\foreach \x in {-14,...,30}
%			\node at (\x*\squaresize-\halfsize,-0.3) {\tiny \x};
%		%	\foreach \y in {0,...,20}
			%\node at (-6.3,\y*\squaresize+\halfsize) {\y};
\node[draw, rectangle, minimum size=6pt, inner sep=0pt] at (-4.2,-0.45) { };
\node[] at ( -3,-0.45)[font=\tiny]{$= \mbox{ a copy of } \Z_3^{\wedge}$};		

\fill(-4.2,-0.75)circle(1.2pt) { };
\node[] at ( -3,-0.75)[font=\tiny]{$= \mbox{ a copy of } \F_3$};		

%\fill(-1.17,-0.5)circle(1.2pt) { };
%\fill(-.65,-0.15)circle(1.2pt) { };
\draw[-,black] (-1.17,-.5) -- (-.65,-.15); 
\node[] at ( 0.8,-0.45)[font=\tiny]{$= \mbox{ multiplication by } \alpha$};		

\draw[-,brown] (3.05,-.85) -- (4.8,-.15); 
\node[] at ( 6,-0.45)[font=\tiny]{$= \mbox{ multiplication by } \beta$};	
			\foreach \x in {0} {
				\node[draw, rectangle, minimum size=6pt, inner sep=0pt] at (\x*\squaresize-\halfsize,0.15) { };
\node[] at (\x*\squaresize -\halfsize+ 0.2,0.15)[font=\tiny]{$1$};
			}

			\foreach \x in {12}{
				\node[draw, rectangle, minimum size=6pt, inner sep=0pt] at (\x*\squaresize-\halfsize,0.15) [font=\tiny]{ };
\node[] at (\x*\squaresize -\halfsize+ 0.22,0.15)[font=\tiny]{$\Delta$};
		   }	
		   \foreach \x in {-12}{
		   	\node[draw, rectangle, minimum size=6pt, inner sep=0pt] at (\x*\squaresize-\halfsize,0.15)[font=\tiny] {};
\node[] at (\x*\squaresize -\halfsize+ 0.4,0.15)[font=\tiny]{$\Delta^{-1}$};
		   }
		   	\foreach \x in {24}{
		   		\node[draw, rectangle, minimum size=6pt, inner sep=0pt] at (\x*\squaresize-\halfsize,0.15) { };
\node[] at (\x*\squaresize -\halfsize+ 0.3,0.15)[font=\tiny]{$\Delta^2$};
		   	}
	\foreach \y in {1} {
		\fill (\y*\squaresize,\y*\squaresize+\halfsize) circle(1.2pt) {};
		\fill (4.2+\y*\squaresize,\y*\squaresize+\halfsize) circle(1.2pt) {};
		\fill (-4.2+\y*\squaresize,\y*\squaresize+\halfsize) circle(1.2pt) {};
		\fill (8.4+\y*\squaresize,\y*\squaresize+\halfsize) circle(1.2pt) {};
	}

\node[ ] at (.53,.52)[font=\tiny]{$\alpha$};
\draw[-,black] (-.175,.15) -- (.35,.525); 
\draw[-,black] (-4.2-.175,.15) -- (-4.2+.35,.525); 
\draw[-,black] (4.2-.175,.15) -- (4.2+.35,.525); 
\draw[-,black] (8.4-.175,.15) -- (8.4+.35,.525); 

	\foreach \y in {2}{
	\fill (2*\y*\squaresize+\halfsize,\y*\squaresize+\halfsize) circle(1.2pt) {};
	\fill (4.2+2*\y*\squaresize+\halfsize,\y*\squaresize+\halfsize) circle(1.2pt) {};
	\fill (-4.2+2*\y*\squaresize+\halfsize,\y*\squaresize+\halfsize) circle(1.2pt) {};}

\node[ ] at (1.75,.87)[font=\tiny]{$\beta$};
\draw[-,brown] (-.175,.15) -- (1.58,.87); 
\draw[-,brown] (-4.2-.175,.15) -- (-4.2+1.58,.875); 
\draw[-,brown] (4.2-.175,.15) -- (4.2+1.58,.875);

	\foreach \y in {3} {
		\fill (2*\y*\squaresize,\y*\squaresize+\halfsize) circle(1.2pt) {};
		\fill (4.2+2*\y*\squaresize,\y*\squaresize+\halfsize) circle(1.2pt) {};
		\fill (-4.2+2*\y*\squaresize,\y*\squaresize+\halfsize) circle(1.2pt) {};
			}

\draw[-,black] (1.58,.87) -- (2.1,1.22); 
\draw[-,black] (-2.62,.87) -- (-4.2+2.1,1.22); 
\draw[-,black] (4.2+1.58,.87) -- (4.2+2.1,1.22); 	

\draw[-,brown] (.35,.525) -- (2.1,1.22); 
\draw[-,brown] (-4.2+.35,.52) -- (-4.2+2.1,1.22); 
\draw[-,brown] (4.2+.35,.52) -- (4.2+2.1,1.22);

\foreach \y in {4}{
		\fill (2*\y*\squaresize+\squaresize+\halfsize,\y*\squaresize+\halfsize) circle(1.2pt) {};
		\fill (4.2+2*\y*\squaresize+\squaresize+\halfsize,\y*\squaresize+\halfsize) circle(1.2pt) {};
		\fill (-4.2+2*\y*\squaresize+\squaresize+\halfsize,\y*\squaresize+\halfsize) circle(1.2pt) {};
			\fill (-8.4+2*\y*\squaresize+\squaresize+\halfsize,\y*\squaresize+\halfsize) circle(1.2pt) {};
	}
			
\draw[-,brown] (1.58,.87) -- (3.32,1.57); 
\draw[-,brown] (-4.2+1.58,.87) -- (-4.2+3.32,1.57); 
\draw[-,brown] (4.2+1.58,.87) -- (4.2+3.32,1.57);

		\foreach \y in {5} {
		\fill (2*\y*\squaresize+\squaresize,\y*\squaresize+\halfsize) circle(1.2pt) {};
		\fill (4.2+2*\y*\squaresize + \squaresize,\y*\squaresize+\halfsize) circle(1.2pt) {};
		\fill (-4.2+2*\y*\squaresize+\squaresize,\y*\squaresize+\halfsize) circle(1.2pt) {};
		\fill (-8.4+2*\y*\squaresize+\squaresize,\y*\squaresize+\halfsize) circle(1.2pt) {};
		
	}

\draw[-,black] (3.32,1.57) -- (3.85,1.92); 
\draw[-,black] (-4.2+3.32,1.57) -- (-4.2+3.85,1.92); 
\draw[-,black] (4.2+3.32,1.57) -- (4.2+3.85,1.92);
\draw[-,black](-8.4+3.32, 1.57)--(-8.4+3.85,1.92); 	

\draw[-,brown] (2.1,1.22) -- (3.85,1.92); 
\draw[-,brown] (-4.2+2.1,1.22) -- (-4.2+3.85,1.92); 
\draw[-,brown] (4.2+2.1,1.22) -- (4.2+3.85,1.92);

		\foreach \y in {6}{
		\fill (2*\squaresize+\halfsize+2*\y*\squaresize,\y*\squaresize+\halfsize) circle(1.2pt) {};
\fill (4.2+2*\squaresize+\halfsize+2*\y*\squaresize,\y*\squaresize+\halfsize) circle(1.2pt) {};
	    \fill (-4.2+2*\squaresize+\halfsize+2*\y*\squaresize,\y*\squaresize+\halfsize) circle(1.2pt) {};
		\fill (-8.4+2*\squaresize+\halfsize+2*\y*\squaresize,\y*\squaresize+\halfsize) circle(1.2pt) {};
	}
	
\draw[-,brown] (3.32,1.57) -- (5.07,2.27); 
\draw[-,brown] (4.2+3.32,1.57) -- (4.2+5.07,2.27); 
\draw[-,brown] (-4.2+3.32,1.57) -- (-4.2+5.07,2.27); 
\draw[-,brown] (-8.4+3.32,1.57) -- (-8.4+5.07,2.27);

	\foreach \y in {7} {
		\fill (2*\squaresize+2*\y*\squaresize,\y*\squaresize+\halfsize) circle(1.2pt) {};
		\fill (-4.2+2*\squaresize+2*\y*\squaresize,\y*\squaresize+\halfsize) circle(1.2pt) {};
		\fill (-8.4+2*\squaresize+2*\y*\squaresize,\y*\squaresize+\halfsize) circle(1.2pt) {};	
	}
	
\draw[-,black] (5.07,2.27) -- (5.6,2.62); 
\draw[-,black] (-4.2+5.07,2.27) -- (-4.2+5.6,2.62); 
\draw[-,black](-8.4+5.07, 2.27)--(-8.4+5.6,2.62); 	

\draw[-,brown] (3.85,1.92) -- (5.6,2.62); 
\draw[-,brown] (-4.2+3.85,1.92) -- (-4.2+5.6,2.62); 
\draw[-,brown] (-8.4+3.85,1.92) -- (-8.4+5.6,2.62);

		\foreach \y in {8}{
		\fill (3.5*\squaresize+2*\y*\squaresize,\y*\squaresize+\halfsize) circle(1.2pt) {};
		\fill (-4.2+3.5*\squaresize+2*\y*\squaresize,\y*\squaresize+\halfsize) circle(1.2pt) {};
		\fill (-8.4+3.5*\squaresize+2*\y*\squaresize,\y*\squaresize+\halfsize) circle(1.2pt) {};
		\fill (-12.6+3.5*\squaresize+2*\y*\squaresize,\y*\squaresize+\halfsize) circle(1.2pt) {};
	}
\draw[-,brown] (5.07,2.27) -- (6.82,2.97); 
%\draw[-,brown] (4.2+3.32,1.57) -- (4.2+5.07,2.27); 
\draw[-,brown] (-4.2+5.07,2.27) -- (-4.2+6.82,2.97); 
\draw[-,brown] (-8.4+5.07,2.27) -- (-8.4+6.82,2.97);
	
		\foreach \y in {9} {
		\fill (3*\squaresize+2*\y*\squaresize,\y*\squaresize+\halfsize) circle(1.2pt) {};
		\fill (-4.2+3*\squaresize+2*\y*\squaresize,\y*\squaresize+\halfsize) circle(1.2pt) {};
		\fill (-8.4+3*\squaresize+2*\y*\squaresize,\y*\squaresize+\halfsize) circle(1.2pt) {};
		\fill (-12.6+3*\squaresize+2*\y*\squaresize,\y*\squaresize+\halfsize) circle(1.2pt) {};
	}
\draw[-,black] (6.82,2.97) -- (7.35,3.32); 
\draw[-,black] (-4.2+6.82,2.97) -- (-4.2+7.35,3.32); 
\draw[-,black](-8.4+6.82, 2.97)--(-8.4+7.35,3.32); 	
\draw[-,black](-12.6+6.82, 2.97)--(-12.6+7.35,3.32);

\draw[-,brown] (5.6,2.62) -- (7.35,3.32); 
\draw[-,brown] (-4.2+5.6,2.62) -- (-4.2+7.35,3.32); 
\draw[-,brown] (-8.4+5.6,2.62) -- (-8.4+7.35,3.32);	

	\foreach \y in {10}{
		\fill (4.5*\squaresize+2*\y*\squaresize,\y*\squaresize+\halfsize) circle(1.2pt) {};
		\fill (-4.2+4.5*\squaresize+2*\y*\squaresize,\y*\squaresize+\halfsize) circle(1.2pt) {};
		\fill (-8.4+4.5*\squaresize+2*\y*\squaresize,\y*\squaresize+\halfsize) circle(1.2pt) {};
		\fill (-12.6+4.5*\squaresize+2*\y*\squaresize,\y*\squaresize+\halfsize) circle(1.2pt) {};
	}
	
\draw[-,brown] (6.82,2.97) -- (8.57,3.67); 
\draw[-,brown] (-12.6+6.82,2.97) -- (-12.6+8.57,3.67); 
\draw[-,brown] (-4.2+6.82,2.97) -- (-4.2+8.57,3.67); 
\draw[-,brown] (-8.4+6.82,2.97) -- (-8.4+8.57,3.67);
	
	\foreach \y in {11} {
		\fill (4*\squaresize+2*\y*\squaresize,\y*\squaresize+\halfsize) circle(1.2pt) {};
		\fill (-4.2+4*\squaresize+2*\y*\squaresize,\y*\squaresize+\halfsize) circle(1.2pt) {};
		\fill (-8.4+4*\squaresize+2*\y*\squaresize,\y*\squaresize+\halfsize) circle(1.2pt) {};
		\fill (-12.6+4*\squaresize+2*\y*\squaresize,\y*\squaresize+\halfsize) circle(1.2pt) {};
	}
	
\draw[-,black] (8.57,3.67) -- (9.07,4.02); 
\draw[-,black] (-4.2+8.57,3.67) -- (-4.2+9.07,4.02); 
\draw[-,black](-8.4+8.57, 3.67)--(-8.4+9.07,4.02); 	
\draw[-,black](-12.6+8.57, 3.67)--(-12.6+9.07,4.02);

\draw[-,brown] (7.35,3.32) -- (9.07,4.02); 
\draw[-,brown] (-4.2+7.35,3.32) -- (-4.2+9.07,4.02); 
\draw[-,brown] (-8.4+7.35,3.32) -- (-8.4+9.07,4.02);	
\draw[-,brown] (-12.6+7.35,3.32) -- (-12.6+9.07,4.02);

\foreach \y in {12}{
		%\fill (-16.8+5.5*\squaresize+2*\y*\squaresize,\y*\squaresize+\halfsize) circle(1.2pt) {};
		\fill (-4.2+5.5*\squaresize+2*\y*\squaresize,\y*\squaresize+\halfsize) circle(1.2pt) {};
		\fill (-8.4+5.5*\squaresize+2*\y*\squaresize,\y*\squaresize+\halfsize) circle(1.2pt) {};
		\fill (-12.6+5.5*\squaresize+2*\y*\squaresize,\y*\squaresize+\halfsize) circle(1.2pt) {};
	}	

%\draw[-,brown] (6.82,2.97) -- (8.57,3.67); 
\draw[-,brown] (-12.6+8.57,3.67) -- (-12.6+10.32,4.37); 
\draw[-,brown] (-4.2+8.57,3.67) -- (-4.2+10.32,4.37); 
\draw[-,brown] (-8.4+8.57,3.67) -- (-8.4+10.32,4.37);

	\foreach \y in {13} {
		\fill (-16.8+5*\squaresize+2*\y*\squaresize,\y*\squaresize+\halfsize) circle(1.2pt) {};
		\fill (-4.2+5*\squaresize+2*\y*\squaresize,\y*\squaresize+\halfsize) circle(1.2pt) {};
		\fill (-8.4+5*\squaresize+2*\y*\squaresize,\y*\squaresize+\halfsize) circle(1.2pt) {};
		\fill (-12.6+5*\squaresize+2*\y*\squaresize,\y*\squaresize+\halfsize) circle(1.2pt) {};
	}
	
%\draw[-,black] (8.57,3.67) -- (9.07,4.02); 
\draw[-,black] (-4.2+10.32,4.37) -- (-4.2+10.84,4.72); 
\draw[-,black](-8.4+10.32, 4.37)--(-8.4+10.84,4.72); 	
\draw[-,black](-12.6+10.32, 4.37)--(-12.6+10.84,4.72);

%\draw[-,brown] (7.35,3.32) -- (9.07,4.02); 
\draw[-,brown] (-4.2+9.09,4.02) -- (-4.2+10.84,4.72); 
\draw[-,brown] (-8.4+9.09,4.02) -- (-8.4+10.84,4.72);	
\draw[-,brown] (-12.6+9.09,4.02) -- (-12.6+10.84,4.72);

	\foreach \y in {0,2,4,6} {
		\draw[->,blue] (4.03+2.5*\y*\squaresize,\y*\squaresize+\halfsize) -- (3.85+2.5*\y*\squaresize,1.75+\y*\squaresize+\halfsize); 
	}
\foreach \y in {0} {
		\draw[->,blue] (4.2+4.03+2.5*\y*\squaresize,\y*\squaresize+\halfsize) -- (4.2+3.85+2.5*\y*\squaresize,1.75+\y*\squaresize+\halfsize); 
	}
\foreach \y in {0,2,4,6,8} {
		\draw[->,blue] (-8.4+4.03+2.5*\y*\squaresize,\y*\squaresize+\halfsize) -- (-8.4+3.85+2.5*\y*\squaresize,1.75+\y*\squaresize+\halfsize); 
	}
\foreach \y in {4,6,8} {
		\draw[->,blue] (-12.6+4.03+2.5*\y*\squaresize,\y*\squaresize+\halfsize) -- (-12.6+3.85+2.5*\y*\squaresize,1.75+\y*\squaresize+\halfsize); 
	}

	\foreach \y in {1} {
		\draw[->,red] (7.87+2.5*\y*\squaresize,\y*\squaresize+\halfsize) -- (7.7+2.5*\y*\squaresize,3.15+\y*\squaresize+\halfsize); 
	}

\foreach \y in {1,3} {
		\draw[->,red] (-12.6+7.87+2.5*\y*\squaresize,\y*\squaresize+\halfsize) -- (-12.6+ 7.7+ 2.5*\y*\squaresize, 3.15+\y*\squaresize+\halfsize); 
	}

%	\foreach \y in {10} {
%		\draw[->,blue] (-9.5+\y*\squaresize,\y*\squaresize+\halfsize) -- (-10.05+\y*\squaresize,2.4+\y*\squaresize+\halfsize); 
%	}

		\end{scriptsize}
	\end{tikzpicture}
\caption{ The homotopy fixed point spectral sequence of $EO_2$. The blue arrows describe $d_5$, and the red arrows describe $d_9$.}
\label{eo2htpyfixpt}
\end{figure}
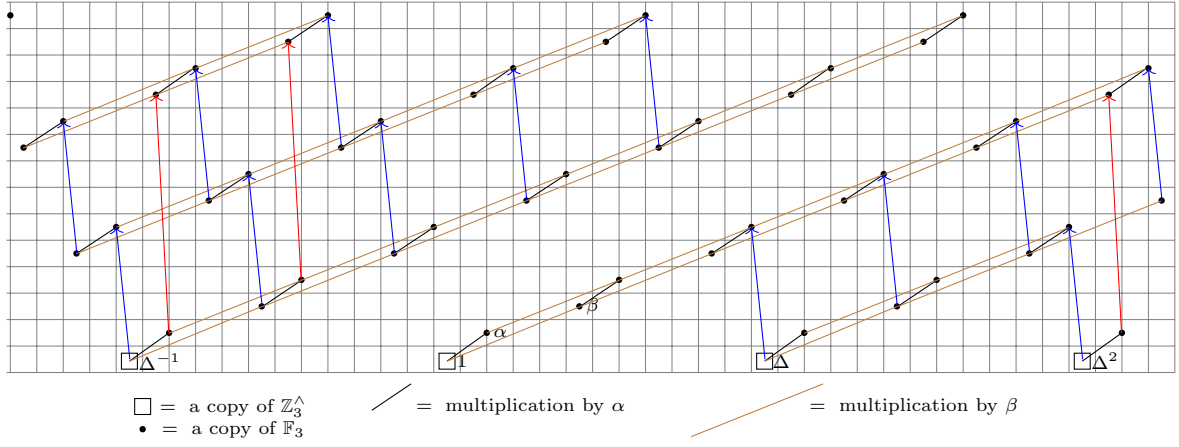

\end{mysubsection}

\vspace*{0.2cm}

	\begin{mysubsection}{The Hurewicz image for $EO_2$} Let $\rho : \pi_\ast(\Sp) \to \pi_\ast(EO_2)$ be the Hurewicz homomorphism. The image of $\rho$ is precisely computed in \cite{BS23} (also stated without proof in \cite{Hen14}). In order to read off the image of $\rho$ from these references, one must recall the link between $EO_2$ and $tmf$, namely, $EO_2 \simeq L_{K(2)}TMF$ at the prime $3$ (see for example \cite[Remark 6.6.11]{Beh20}. We note this in the following theorem
\begin{thm}\label{hureo2}
The image of $\rho : \pi_\ast(\Sp_3^\land) \to \pi_\ast(EO_2)$ consists of the $\Z_3^\wedge$ generated by $1$ in degree $0$, and the classes 
\[\alpha, \quad \Delta^{3t}\beta^i \mbox{ for } 1\leq i \leq 4, \quad \Delta^{3t}\alpha \beta, \quad \Delta^{3t}\beta b,\]
for $t\geq 0$.
\end{thm}
\end{mysubsection}

\vspace*{0.2cm}

\begin{mysubsection}{The homotopy groups of $EO_2 \wedge C(\alpha_1)$}
Let $C(\alpha_1)$ be the $p$-local spectrum given by the mapping cone of $\alpha_1 : S^{2p-3} \to S^0$. This is a cellular spectrum with a cell in degree $0$, and a cell in degree $2p-2$, with $\alpha_1$ as the attaching map. For a spectrum $X$, we use the notation $\EE_\ast^{s,t,r}(EO_{p-1}\wedge X)$ for the filtration spectral sequence (Theorem \ref{FSS}), and $E_\ast^{s,t}(EO_{p-1}\wedge X)$ for the homotopy fixed point spectral sequence. At the prime $3$, the filtration spectral sequence for $EO_2 \wedge C(\alpha_1)$ has the second page
\begin{myeq}\label{eo2calphfss}
\EE_2^{s,t,r}(EO_2 \wedge C(\alpha_1))= H^s(G;\pi_{t+s} E_2)\{\iota_0, \iota_4\}, \quad |\iota_0|= (0,0,0), \quad |\iota_4| = (0,4,-4).
\end{myeq}
The filtration spectral sequence $\EE_{2}^{s,t,r}$ is concentrated at $r=0,-4$ and the only possible differential is $d_{4}$. As the $EO_2$-Hurewicz image of $\alpha_1$ is $\alpha$, we must have $d_4(\iota_4)= \alpha \iota_0$ in the $EO_2$-Atiyah-Hirzebruch spectral sequence for $EO_2 \wedge C(\alpha_1)$.  Now Proposition \ref{diffl-att} determines $d_{4}$.
	\begin{prop}\label{d4eo2calphfss}
		The differential $d_{4}$ in $\EE_2^{s,t,r}(EO_2 \wedge C(\alpha_1))$ is given by $d_{4}(\iota_{4})=\alpha \iota_{0}$.
	\end{prop}    

\vspace*{0.2cm}	

We conclude  (see Figure \ref{eo2calphfiltspseq})
\[\EE_\infty(EO_2 \wedge C(\alpha_1)) \cong \EE_5(EO_2\wedge C(\alpha_1)) \cong \frac{H^{*}(G,{E_{2}}_{*})\{\iota_0, [\alpha\iota_4], [3\iota_4]\}}{(\alpha \iota_0,\alpha  [\alpha\iota_4], 3 [\alpha\iota_4],\beta  [3\iota_4],\alpha  [3\iota_4])} \]  
	%Here $\ba:= [\alpha.b], \bb:=[3b]$.
In this expression, and also henceforth in the rest of the section, we suppress the multiples of elements of $M_3$ in the formulas. From this formula we may write down the $E_2$-page of the homotopy fixed point spectral sequence by solving the extension problems. In this case we have 
\begin{myeq} \label{extfsscalpha}
\begin{aligned}
\alpha  [\alpha\iota_4] & = \alpha \langle \iota_0, \alpha, \alpha \rangle \\
                                 & = \iota_0 \langle \alpha, \alpha, \alpha \rangle \\
                                 & = \beta \iota_0.
\end{aligned}
\end{myeq}

\vspace*{0.2cm}

\begin{figure}
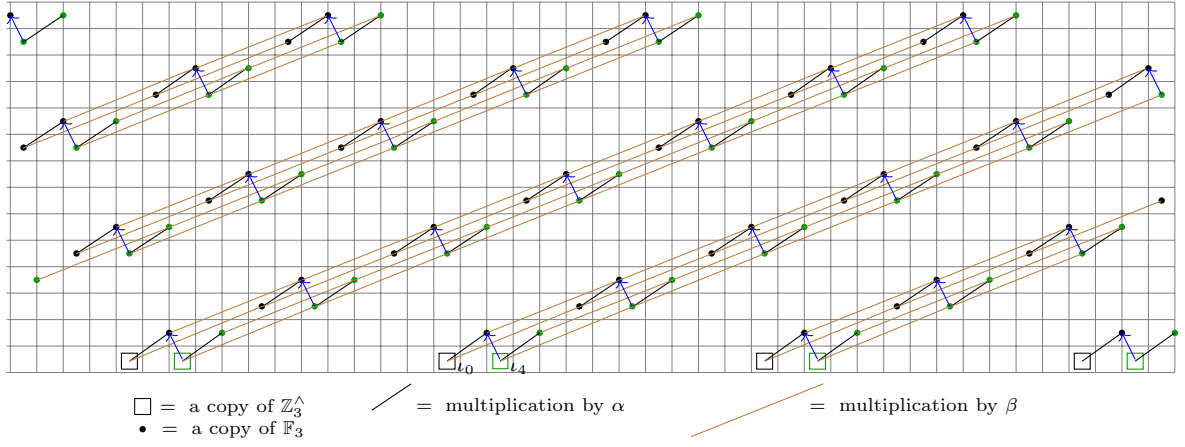

\centering

\caption{ The filtration spectral sequence of $EO_2\wedge C(\alpha_1)$. The blue arrows describe $d_4$, and the green entries denote the multiples of $\iota_4$.}
\label{eo2calphfiltspseq}
\end{figure}

\vspace*{0.2cm}

\begin{exam}\label{chcxfss} 
The filtration spectral sequence is essentially an algebraic computation of the $E_2$-term of the homotopy fixed point spectral sequence arising from the cellular filtration on $C^\ast(G;\pi_\ast(E_2\wedge C(\alpha_1)))$. This chain complex may be described as 
\[C^\ast(G;\pi_\ast(E_2))\{\iota_0,\iota_4\}, \quad d(\iota_0)=0, \quad d(\iota_4)=\oalpha\iota_0,\]
where $\oalpha\in C^1(G;\pi_4E_2)$ is a cocycle representing $\alpha$. Let $y \in C^1(G;\pi_8 E_2)$ be such that $d(y)=\oalpha^2$. Note that the class $\beta$ is represented by  $\obeta=y\oalpha+\oalpha y$. One can easily check that $[\alpha \iota_4]$ is represented by $\oalpha\iota_4 + y\iota_0$. Thus $\alpha [\alpha \iota_4]$ is represented by $\oalpha^2\iota_4 + \oalpha y \iota_0 $.  We now check 
\[
 d(y \iota_4) = \oalpha^2 \iota_4 - y \oalpha \iota_0  = (\oalpha^2 \iota_4 + \oalpha y \iota_0) - \obeta \iota_0 \]
\[\implies \alpha [\alpha \iota_4] = \beta \iota_0,\]
proving \eqref{extfsscalpha}.
\end{exam}

\vspace*{0.2cm}

%\begin{exam} 
%We may also derive \eqref{extfsscalpha} directly from the Adams-Novikov spectral sequence using the map 
%		\[{\Ext_{BP_*BP}}(BP_*,BP_*C(\alpha_1)) \to H^*(G,{E_2}_*C(\alpha_1)).\]
%		We check that the relation holds in the Adams-Novikov spectral sequence by considering explicit relations in the cobar-complex. The element $t_1|\iota_4+\frac{1}{2}{t_1}^2|\iota_0$ is an element of ${\Ext_{BP_*BP}}(BP_*,BP_*C(\alpha_1))$ which maps to $[\alpha \iota_4]$ in the right hand side. The class $\alpha_1$ is represented by $t_1$ and $\beta_1$ is represented by  $\frac{1}{2}(t_1|{t_1}^2+{t_1}^2|t_1)$  in ${\Ext_{BP_*BP}}(BP_*,BP_*)$.
%Then, the following relation holds in the cobar-complex up to coboundaries
%		\[t_1|t_1|\iota_4+\frac{1}{2} t_1|{t_1}^2|\iota_0=\frac{1}{2}[t_1|{t_1}^2+{t_1}^2|t_1]|\iota_0.\] 
%To see this,  consider the element ${t_1}^2|\iota_4+{t_1}^3|\iota_0$. We have,
%		\begin{align*}
%			\overline{\Delta}({t_1}^2|\iota_4+{t_1}^3|\iota_0)&=2t_1|t_1|\iota_4-{t_1}^2|t_1|\iota_0+3t_1|{t_1}^2|\iota_0+3{t_1}^2|t_1|\iota_0 \mbox{ as } \overline{\Delta}(\iota_4)= t_1|\iota_0 \\
%			&=2(t_1|t_1|\iota_4+\frac{1}{2}t_1|{t_1}^2|\iota_0)+2[t_1|{t_1}^2+{t_1}^2|t_1]|\iota_0.
%		\end{align*}
%Mapping to the $E_2$-page of homotopy fixed point spectral sequence we have, 
%$$2\alpha [\alpha\iota_4]+4\beta\iota_0=0 \implies \alpha [\alpha \iota_4]= \beta \iota_0. $$
%\end{exam}
%	
%\vspace*{0.2cm}	
	
We now compute the homotopy fixed point spectral sequence for $EO_2\wedge C(\alpha_1)$ (see Figure \ref{eo2calphhtpyfixspseq}). This is a module over the homotopy fixed point spectral sequence for $EO_2$. We write $\kappa_7= [\alpha \iota_4]$, $\chi_4=[3\iota_4]$, and using \eqref{extfsscalpha} we have,
	\[E_{2}(EO_2 \wedge C(\alpha_1))\cong\frac{H^{*}(G;{E_2}_{*}) \{\iota_0,\kappa_7,\chi_4 \}}{(\alpha \iota_0,\alpha \kappa_7-\beta \iota_0,3\kappa_7,\alpha \chi_4, \beta \chi_4)}, \quad |\kappa_7|= (1,7), |\chi_4|=(0,4).\]
Now we determine the differentials. The cofiber sequence
	\begin{myeq} \label{cofcalpha1}
S^{0}\xrightarrow{\iota_0} C(\alpha_1) \xrightarrow{j} S^4\xrightarrow{\alpha_1} S^1
\end{myeq}
implies that $\iota_0$ and $\chi$ are permanent cycles, as they lie in the Hurewicz image $\pi_\ast (C(\alpha_1)) \to {EO_2}_\ast(C(\alpha_1))$. As the other generating classes are in positive grading, it suffices to work $\pmod{M_3}$ (see \eqref{einftyhfpsseo}). Note that 
\[E_{2}(EO_2 \wedge C(\alpha_1))\cong \frac{\F_3[\Delta^{\pm 1}, \alpha,\beta] \{\iota_0,\kappa_7\}}{(\alpha^2,\alpha \iota_0,\alpha \kappa_7-\beta \iota_0)} \pmod{(M_3,\chi_4)}.\]
 Also from the cofiber sequence \eqref{cofcalpha1}, we note that $[\alpha \iota_4]$ represents a class in $\pi_7(EO_2\wedge C(\alpha_1))$, and thus, $\kappa_7$ is a permanent cycle in $E_{\ast}(EO_2 \wedge C(\alpha_1))$. From \eqref{dhfpsseo2}, we know that $d_5(\Delta)=\alpha \beta^2$, so we have
\begin{myeq} \label{d5hfpsscalph}
d_5 (\Delta \kappa_7) = \alpha \beta^2 \kappa_7 = \beta^3 \iota_0.
\end{myeq}
This implies 
\[d_5(\Delta^t \kappa_7)= t\alpha \beta^2 \Delta^{t-1} \kappa_7= t\Delta^{t-1} \beta^3 \iota_0, \implies d_5(\Delta^{3t} \kappa_7)=0.\]
The classes that survive to the $6^{th}$-page are $\Delta^{\pm t} \beta^s \iota_0=\Delta^{\pm t}\alpha \beta^s \kappa_7$ for $0\leq s\leq 2$, $t\in \Z$, and $t \equiv 0 \mbox{ or } 1 \pmod{3}$, $\Delta^{\pm 3t+2} \beta^s \iota_0$ for $s\geq 0$ and $t\in \Z$,  and $\Delta^{\pm 3t} \beta^s \kappa_7$ for $s\geq 0$ and $t\in \Z$. It follows that 
\[E_{6}(EO_2 \wedge C(\alpha_1))\cong \frac{\F_3[\Delta^{\pm 1}] \{ \iota_0, \beta \iota_0, \beta^2 \iota_0\}\oplus \F_3[\beta]\{\Delta^{3t+2}\iota_0\}\oplus \frac{\F_3[\Delta^{\pm 3}, \alpha,\beta] \{\kappa_7\}}{(\alpha^2, \alpha \beta^2)}}{(\alpha \kappa_7-\beta \iota_0)} \pmod{(M_3,\chi_4)}.\]
The next non-trivial differential is $d_9$. From \eqref{dhfpsseo2}, we know that $d_9(\Delta^2 \alpha)=\beta^5$. We thus have
\begin{myeq}\label{d9hfpsscalph}
d_9(\Delta^2 \beta \iota_0) = d_9(\Delta^2 \alpha \kappa_7) = \beta^5 \kappa_7.
\end{myeq}
Therefore,  we must have 
\[d_9(\Delta^2 \iota_0)=\beta^4 \kappa_7.\] 
This implies 
\[d_9(\Delta^{3t+2} \iota_0)= \Delta^{3t} \beta^4 \kappa_7.\]
Now we compute the $10^{th}$ page and observe that for degree reasons, the spectral sequence has no further non-trivial differentials.
\begin{myeq}\label{einftyhfpsseo2calph}
 E_\infty(EO_2 \wedge C(\alpha_1))\cong \frac{\F_3\{\Delta^{3t}, \Delta^{3t+1}\}\otimes \F_3 \{ \iota_0, \beta \iota_0, \beta^2 \iota_0\}\oplus \frac{\F_3[\Delta^{\pm 3}, \alpha,\beta] \{\kappa_7\}}{(\alpha^2, \alpha \beta^2,\beta^4)}}{(\alpha \kappa_7-\beta \iota_0)} \pmod{(M_3,\chi_4)}.
\end{myeq}
The homotopy groups $\pi_\ast(EO_2 \wedge C(\alpha_1))$ has one hidden multiplicative extension. We observe that $b\iota_0$ belongs to the cokernel of multiplication by $\alpha$, and so via the cofiber sequence \eqref{cofcalpha1}, it is a non-trivial class in $\pi_{27}(EO_2 \wedge C(\alpha_1))$. For degree reasons, we must have $b\iota_0 = \beta^2 \kappa_7$. 

\begin{figure}
\centering

\caption{ The  homotopy fixed point spectral sequence of $EO_2\wedge C(\alpha_1)$ $\pmod{M_3,\chi_4}$. The blue arrows describe $d_5$, and the red arrows describe $d_9$.}
\label{eo2calphhtpyfixspseq}
\end{figure}

\vspace*{0.2cm}

\end{mysubsection}

\vspace*{0.2cm}

\begin{mysubsection}{The homotopy groups of $EO_2 \wedge C(\alpha_1,2 \alpha_1)$}
Let $C(\alpha_1)$ denote the cofiber of $\alpha_1:S^{3} \rightarrow S^0$, and  consider the cofiber sequence 
\[S^0 \xrightarrow{\iota_0} C(\alpha_1) \xrightarrow{j_4} S^4.\] 
 After attaching an $8$-cell to  $C(\alpha_1)$ using the  map $2\bar{\alpha}_1: S^7 \rightarrow C(2\alpha_1)$ defined via \eqref{defc2alph},  the resulting complex will be denoted by $C(\alpha_1,2\alpha_1)$. 
	% https://q.uiver.app/#q=WzAsNyxbMCwxLCJTXnswfSJdLFsxLDEsIkMoMlxcYWxwaGEpIl0sWzIsMSwiU157NH0iXSxbMiwzXSxbMiwyLCJTXns3fSJdLFszLDEsIlNeezF9Il0sWzEsMCwiQygyXFxhbHBoYSxcXGFscGhhKSJdLFswLDFdLFsxLDJdLFs0LDIsIlxcYWxwaGEiLDJdLFsyLDUsIjJcXGFscGhhIl0sWzQsNSwiMCIsMl0sWzQsMSwiXFxiYXJcXGFscGhhIiwyLHsic3R5bGUiOnsiYm9keSI6eyJuYW1lIjoic3F1aWdnbHkifX19XSxbMSw2XV0=
	\begin{myeq}\label{defc2alph}
\begin{tikzcd}
		& {C(\alpha_1,2\alpha_1)} \\
		{S^{0}} & {C(\alpha_1)} & {S^{4}} & {S^{1}} \\
		&& {S^{7}} \\
		&& {}
		\arrow[from=2-1, to=2-2]
		\arrow[from=2-2, to=1-2]
		\arrow[from=2-2, to=2-3]
		\arrow["{\alpha_1}", from=2-3, to=2-4]
		\arrow["{2\bar\alpha_1}"', from=3-3, to=2-2]
		\arrow["2\alpha_1"', from=3-3, to=2-3]
		\arrow["0"', from=3-3, to=2-4]
	\end{tikzcd}
\end{myeq}
	
	\begin{rmk}
		The lift $2\bar{\alpha}_1$ defined in \eqref{defc2alph} is not unique as $\pi_{7}(C(\alpha_1))$ is an extension of $\pi_{3}(S^0)$ by the group $\pi_{7}(S^0)$, both of which are non-zero as they are generated by $\alpha_1$ and $\alpha_{2}$ respectively. Nonetheless all such spectra will be denoted by $C(\alpha_1,2\alpha_1)$ even though they are not homotopy equivalent. Also note that we have analogous definitions of $C(i\alpha_1, j\alpha_1)$ for all $i,j \in \{1,2\}$. 
	\end{rmk}
	
	\vspace*{0.2cm}

The spectra $C(\alpha_1,2\alpha_1)$ have cells in degrees $0,4,8$, hence the $2^{nd}$-page of the filtration spectral sequence is given by 
\[\EE_2^{s,t,r}(EO_2 \wedge C(\alpha_1,2\alpha_1))= H^s(G;\pi_{t+s} E_2)\{\iota_0, \iota_4,\iota_8\},\] 
\[ |\iota_0|= (0,0,0), \quad |\iota_4| = (0,4,-4), \quad |\iota_8|=(0,8,-8).\]
The non-trivial  differentials are $d_{4}$ and $d_{8}$ (see Figure \ref{eo2calph2alphfiltspseq}). Applying Proposition \ref{diffl-att} as in the case of $C(\alpha_1)$ we obtain 
\[d_4(\iota_4)=\alpha \iota_0, \quad d_4(\iota_8)=2 \alpha \iota_4.\]
We now obtain 
\[\EE_5(EO_2\wedge C(\alpha_1,2\alpha_1)) \cong \frac{H^{*}(G;{E_{2}}_{*})\{\iota_0, [3\iota_4],[\alpha\iota_8], [3\iota_8]\}}{(\alpha \iota_0,\alpha  [\alpha\iota_8], 3 [\alpha\iota_8],\beta  [3\iota_4],\alpha  [3\iota_4],\beta  [3\iota_8],\alpha  [3\iota_8])} \]  
The next differential is $d_8$ which is computed in the following proposition.	
	\begin{prop}\label{d8fsseo}
		In the filtration spectral sequence for $C(\alpha_1,2\alpha_1)$ we have $d_{8}([\alpha \iota_8])=\beta \iota_{0}$. 
			\end{prop}
	
	\begin{proof}
		The diagram given below is the exact couple that leads to the filtration spectral sequence of $C(\alpha_1, 2\alpha_1)$. 
		\[\begin{tikzcd}
			{E_{2}(EO_2\wedge S^0)} & {E_{2}(EO_2\wedge C(\alpha_1))} & {E_2(EO_2 \wedge C(\alpha_1,2\alpha_1))} \\
			{E_{2}(EO_2 \wedge S^4)} && {E_2(EO_2\wedge S^8)}
			\arrow["\delta", from=1-2, to=2-1]
			\arrow["{i_1}", from=1-1, to=1-2]
			\arrow["\delta", from=1-3, to=2-3]
			\arrow["{i_2}", from=1-2, to=1-3]
			\arrow[from=2-3, to=1-2]
			\arrow[from=2-1, to=1-1]
		\end{tikzcd}\]
 Following the notation of Example \ref{chcxfss}, consider the chain complex of $C^\ast(G;\pi_\ast E_2)$-modules 
\[ C^\ast(G;\pi_\ast(E_2))\{\iota_0,\iota_4,\iota_8\}, \quad d(\iota_0)=0, \quad d(\iota_4)=\oalpha\iota_0, \quad d(\iota_8)=2\oalpha\iota_4+2y\iota_0.\]
The filtration spectral sequence $\EE_\ast(EO_2\wedge C(\alpha_1,2\alpha_1))$ is formed by filtering this chain complex via the inclusions $\iota_0 \in \{\iota_0,\iota_4\} \subset \{\iota_0,\iota_4,\iota_8\}$. We now find a representative for $[\alpha \iota_8]$ in the $5^{th}$-page of the filtration spectral sequence. Note 
\[ d(\oalpha \iota_8)= -2\oalpha^2 \iota_4 - 2\oalpha y \iota_0.\]
In order to make this a cycle under $d_4$, one adds a copy of $2y\iota_4$. Therefore the class $[\alpha \iota_8]$ is represented by $\oalpha \iota_8 + 2y \iota_4$ in the $5^{th}$-page. Now we compute 
\[ d(\oalpha \iota_8 + 2y \iota_4)= -2\oalpha y \iota_0  -2y\oalpha \iota_0 = -2 \obeta \iota_0.\] 
As the class $\beta$ has order $3$, we have $d_8([\alpha \iota_8])= -2 \beta \iota_0= \beta \iota_0$.	
	\end{proof}

\vspace{0.2cm}

Proposition \ref{d8fsseo} shows that the filtration spectral sequence for $EO_2 \wedge C(\alpha_1,2\alpha_1)$ is concentrated at the $s=0$ line from the $9^{th}$-page onwards. This implies that there are no non-trivial differentials for the homotopy fixed point spectral sequence as well, and this gives a computation of $E_\infty(EO_2\wedge C(\alpha_1, 2\alpha_1))$.
\begin{myeq}\label{einftyhfpsseo2c[2]}
E_\infty(EO_2\wedge C(\alpha_1,2\alpha_1)) \cong \frac{H^{*}(G;{E_{2}}_{*})}{(\alpha,\beta)}\{\iota_0, [3\iota_4], [3\iota_8]\} 
\end{myeq}

\vspace*{0.2cm}

\begin{figure}
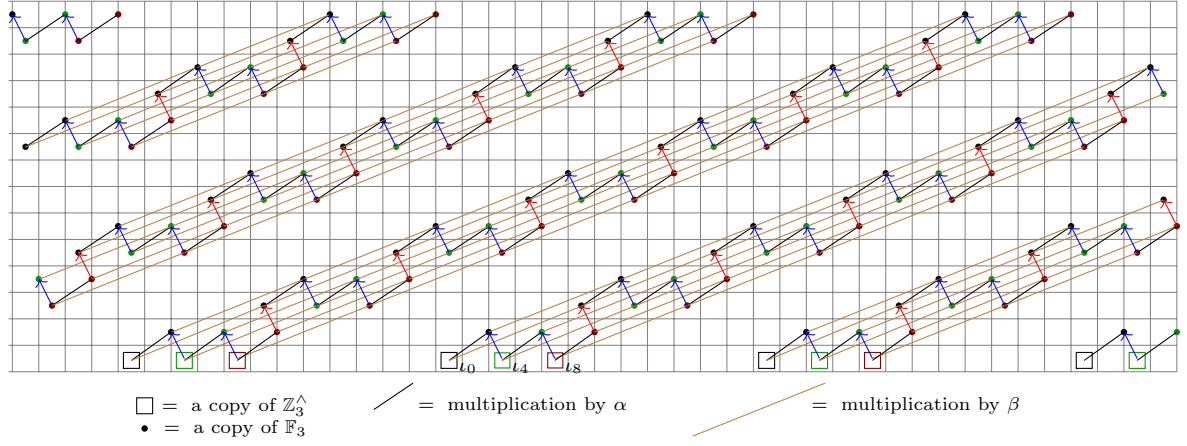

\centering

\caption{ The filtration spectral sequence of $EO_2\wedge C(\alpha_1,2\alpha_1)$. The green dots represent multiples of $\iota_4$, and the maroon dots represent multiples of $\iota_8$. The blue arrows describe $d_4$, and the red arrows describe $d_8$.}
\label{eo2calph2alphfiltspseq}
\end{figure}

\vspace*{0.2cm}

\end{mysubsection}

\vspace*{0.2cm}

	\begin{mysubsection}{Computation of ${EO_{2}^{*}(\C P^{m})}$}
%	We compute the $EO_{2}$-cohomology of $\C P^m$. 
As $EO_2$ is $3$-complete, we may replace $\C P^m$ by its $3$-completion ${\C P^m}^\wedge_{3}$, and for this section, all the relevant spectra will be $3$-complete.
	
We follow the method in \cite{Bau03} which computes the additive structure of $tmf^\ast(\C P^\infty)$, and adapt the arguments for $EO_2$. The complex conjugation of line bundles gives a map $c : \C P^\infty \to \C P^\infty$. Now consider $\frac{1+c}{2}: \Sigma^{\infty}_+ \C P^{\infty} \to  \Sigma^{\infty}_+ \C P^{\infty}$ which is an idempotent. This splits  $\Sigma^{\infty}_+ (\C P^{\infty})$  as the wedge $P_{0} \vee P_{1}$, where these two spectra are given by $P_0=\Tel(\frac{1+c}{2})$ and $P_1=\Tel(\frac{1-c}{2})$ ($\Tel$ is the mapping telescope functor). From the action of $c$ on $H^\ast(\C P^\infty;\F_3) \cong \F_3[x]$ as $c(x)=-x$, we may read off $H^\ast(P_0;\F_3)\subset H^\ast(\C P^\infty;\F_3)$ and $H^\ast(P_1;\F_3) \subset H^\ast(\C P^\infty; \F_3)$. 
\[H^{*}(P_0;\F_3)=\F_3 \{1, x^2,\cdots,x^{2k},\cdots \} , \quad H^{*}(P_1;\F_3)= \F_3 \{x,x^3,\cdots,x^{2k+1},\cdots\} \] 
We will write $P^{m}_{i}$ for the $2m$-th skeleta of $P_{i}$. 
	
	We will filter both $P_0$ and $P_1$ by blocks of $C(2\alpha_1,\alpha_1)$. The action of the Steenrod powers $\PP^i$ on the cohomology classes $x^n$ are given by 
\[\PP^{i}(x^m)= \binom{m}{i} x^{m+2i}.\]
Combining this with the fact that $\PP^1$ detects the class $\alpha_1$, we obtain
\[P^{6}_{0}\simeq \Sigma^{4} C(2\alpha_1,\alpha_1), \quad P^{6k}_{0}/P^{6(k-1)}_{0} \simeq \Sigma^{12k-8}C(2\alpha_1,\alpha_1) \mbox{ for } k\geq 2.\] 
The analogous result for $P_1$ is
 \[P^{3}_{1} \simeq \Sigma^{2} C(\alpha_1) , \quad P^{6k+3}_{1}/ P^{6k-3}_{1} \simeq \Sigma^{12k-2} C(2\alpha_1,\alpha_1) \mbox{ for } k\geq 1.\]

\vspace*{0.2cm}

We now compute $EO_2^\ast(\C P^m)$ for $1\leq n \leq \infty$ using the homotopy fixed point spectral sequence 
\[E_2(EO_2^\ast(\C P^m)) = H^s(G; E_2^{-s-t}(\C P^m)) \implies EO_2^{-t}(\C P^m).\]
The filtration spectral sequence takes the form 
\[\EE_2(EO_2^\ast(\C P^m)) = H^s(G;\pi_{s+t}(E_2)[x]/(x^{n+1}), \quad |x|= (0,-2,2)  \implies H^s(G; E_2^{s+t}(\C P^m)). \]
Theorem \ref{FSS} implies that the spectral sequence $\EE_\ast(EO_2^\ast(\C P^m))$ is multiplicative for $1\leq m \leq \infty$, and as the canonical inclusions $\C P^m \to \C P^{m+k}$ respect the filtration, this induces a map of spectral sequences $\EE_\ast(EO_2^\ast(\C P^{m+k})) \to \EE_\ast(EO_2^\ast(\C P^m))$. The following proposition computes the first differential. 
\begin{prop}\label{d4eo2fss}
The differentials $d_2$ and $d_3$ are $0$ in $\EE_\ast(EO_2^\ast(\C P^m))$, and $d_4$ is determined by $d_4(x)=\alpha x^3$ via the multiplicative structure.
\end{prop}

\begin{proof}
We directly have $d_2(x)=0$ and $d_3(x)=0$ by degree reasons, and this implies that $\EE_2(EO_2^\ast(\C P^m)) = \EE_4(EO_2^\ast(\C P^m))$. For computing $d_4$, it suffices to consider $m=3$, where we note that the formula for $\PP^i$ implies that 
\[ {\C P^3}_3^\wedge \simeq S^4 \vee P_1^3 \simeq S^4 \vee \Sigma^2 C(\alpha_1).\]
We thus have 
\[F(\C P^3, EO_2) \simeq EO_2 \wedge D(\C P^3) \simeq \Sigma^{-4} EO_2 \vee \Sigma^{-2} EO_2 \wedge D(C(\alpha_1)),\]
and then, we use $D(C(\alpha_1)) \simeq \Sigma^{-4} C(\alpha_1)$, to obtain 
\[F(\C P^3, EO_2) \simeq \Sigma^{-4} EO_2 \vee \Sigma^{-6} EO_2 \wedge C(\alpha_1).\] 
The class $x \in \pi_{-2}F(\C P^3, EO_2)$ is the class $\Sigma^{-6} \iota_4$ \eqref{eo2calphfss} in the above equivalence. Now the proof is complete by Proposition \ref{d4eo2calphfss}.
\end{proof}

\vspace*{0.2cm}

Following Proposition \ref{d4eo2fss}, we have
\[d_4(x^j)= j \alpha x^{j+2} = \begin{cases} 
\alpha x^{j+2} & \mbox{if } j \equiv 1 \pmod{3} \\
2 \alpha x^{j+2} & \mbox{if } j \equiv 2 \pmod{3}\\
0 & \mbox{if } j \equiv 0 \pmod{3}. 
\end{cases} \]
We thus have the following formula for the $5^{th}$-page
\begin{myeq} \label{e5fsseo2cp}
\begin{aligned}
\EE_5\Big((EO_2^\ast(\C P^m)\Big) & \cong  H^{*}(G;{E_{2}}_{*})\{ 1 \} \oplus  \frac{H^{*}(G;{E_{2}}_{*})\{[3x],[\alpha x]\}}{(\alpha  [\alpha x], 3 [\alpha x],\beta  [3x],\alpha  [3x])} \\
& \oplus  \frac{H^{*}(G;{E_{2}}_{*})\{ [3x^{3j+1}]\}}{(\beta  [3x^{3j+1}],\alpha  [3x^{3j+1}])} \mbox{  if } 3(j+1) \leq m \\ 
& \oplus \frac{H^{*}(G;{E_{2}}_{*})\{[3x^{3j+2}],[\alpha x^{3j+2}]\}}{(\alpha  [\alpha x^{3j+2}], 3 [\alpha x^{3j+2}],\beta  [3x^{3j+2}],\alpha  [3x^{3j+2}])} \mbox{  if }  3j+4 \leq m \\
& \oplus \frac{H^{*}(G;{E_{2}}_{*})\{x^{3j}\}}{(\alpha x^{3j})} \mbox{  if } 3j \leq m \\ 
& \oplus  \frac{H^{*}(G;{E_{2}}_{*})\{x^{3j+1}\}}{(\alpha x^{3j+1})} \mbox{  if } 3j+1 = m \mbox{ or } m-1\\ 
& \oplus  H^{*}(G;{E_{2}}_{*})\{ x^{3j+2} \} \mbox{  if } 3j+2 = m \mbox{ or } m-1.
\end{aligned}
\end{myeq}

\vspace*{0.2cm}

From the splitting of $\C P^\infty$ into $P_0$ and $P_1$, we observe that the differentials in the filtration spectral sequence $\EE_\ast(EO_2^\ast(\C P^\infty))$ must respect the skeletal filtration of the $P_i$. Using the map $\EE_\ast EO_2^\ast(\C P^\infty) \to \EE_\ast(EO_2^\ast(\C P^m))$, we see that this must also hold in the latter case. This means that the next possible non-trivial differential is $d_8$.   
\begin{prop}\label{d8eo2fss}
The differential $d_8$ is given by  $d_8([\alpha x^{3j+2}]) = \beta x^{3(j+2)}$ for $3(j+2) \leq m$, and that it is trivial on the other generators. 
\end{prop}	

\begin{proof}
Using the maps $\C P^m \to \C P^{m+k}$ we note that it suffices to compute $d_8$ on a class supported on $\C P^m$ by considering the spectral sequence for $\C P^{m+8}$. Now, we have the maps  
\[q_{6k}: \C P^{6k} \to P_0^{6k} \to P_0^{6k}/P_0^{6k-6} \simeq \Sigma^{12k - 8} C(2\alpha_1, \alpha_1),\]
\[q_{6k+3}: \C P^{6k+3} \to P_1^{6k+3} \to P_1^{6k+3}/P_1^{6k-3} \simeq \Sigma^{12k - 2} C(2\alpha_1, \alpha_1).\]
%Using these identifications, we readily deduce 
%\[ \C P^{n+8}/\C P^{n-1} \simeq \begin{cases}
%\Sigma^{2n} C(2\alpha_1,\alpha_1) \vee  S^{2n+2} \vee S^{2n+4} & \mbox{if } n\equiv -1 \pmod{3} \\
%\Sigma^{2n} C(2\alpha_1,\alpha_1) \vee  S^{2n+2} \vee S^{2n+4} & \mbox{if } n\equiv 1 \pmod{3} \\
%\Sigma^{2n} C(2\alpha_1,\alpha_1) \vee  S^{2n+2} \vee S^{2n+4} & \mbox{if } n\equiv 0 \pmod{3}.\end{cases} \]
This induces $q_{3k}^\ast$ 
\[ F(\C P^{3k}, EO_2) \leftarrow F(\Sigma^{6k-4} C(2\alpha_1, \alpha_1), EO_2) \simeq EO_2 \wedge \Sigma^{4-6k} C(\alpha_1, 2 \alpha_1). \]
Now Proposition \ref{d8fsseo} applies to show that $d_8([\alpha x^{3j+2}])= \beta x^{3(j+2)}$ if $m\geq 3j+6$.  The $d_8$ is $0$ on the other generators by degree reasons as the relevant filtration degree does not have any classes in positive $s$-degree.  
\end{proof}

\vspace*{0.2cm}

We now apply Proposition \ref{d8eo2fss} to \eqref{e5fsseo2cp} to obtain the following formula for the $9^{th}$-page
\begin{myeq} \label{e9fsseo2cp}
\begin{aligned}
\EE_9(EO_2^\ast(\C P^m) & \cong  H^{*}(G;{E_{2}}_{*})\{ 1 \} \oplus  \frac{H^{*}(G;{E_{2}}_{*})\{[ \alpha x]\}}{(\alpha  [\alpha x], 3 [\alpha x])} \oplus \frac{H^{*}(G;{E_{2}}_{*})\{x^{3}\}}{(\alpha x^{3})} \\
& \oplus  \frac{H^{*}(G;{E_{2}}_{*})\{ [3x^{j}]\}}{(\beta  [3x^{j}],\alpha  [3x^{j}])} \mbox{  if }  j \leq m-3 \mbox{ and } 3\nmid j \\ 
& \oplus \frac{H^{*}(G;{E_{2}}_{*})\{x^{3j}\}}{(\alpha x^{3j}, \beta x^{3j})} \mbox{  if } 6 \leq 3j \leq m \\ 
& \oplus \frac{H^{*}(G;{E_{2}}_{*})\{[\alpha x^{3j+2}]\}}{(\alpha  [\alpha x^{3j+2}], 3 [\alpha x^{3j+2}])} \mbox{  if }  3j+2 = m-2 \mbox{ or } m-3\\
& \oplus  \frac{H^{*}(G;{E_{2}}_{*})\{x^{3j+1}\}}{(\alpha x^{3j+1})} \mbox{  if } 3j+1 = m \mbox{ or } m-1\\ 
& \oplus  H^{*}(G;{E_{2}}_{*})\{ x^{3j+2} \} \mbox{  if } 3j+2 = m \mbox{ or } m-1.
\end{aligned}
\end{myeq}

\vspace*{0.2cm}

In the case $m=\infty$, we have the following expression 
\[
\begin{aligned}
\EE_9(EO_2^\ast(\C P^\infty)) & \cong  H^{*}(G;{E_{2}}_{*})\{ 1 \} \oplus  \frac{H^{*}(G;{E_{2}}_{*})\{[ \alpha x]\}}{(\alpha  [\alpha x], 3 [\alpha x])} \oplus \frac{H^{*}(G;{E_{2}}_{*})\{x^{3}\}}{(\alpha x^{3})} \\
& \oplus  \frac{H^{*}(G;{E_{2}}_{*})\{ [3x^{j}]\}}{(\beta  [3x^{j}],\alpha  [3x^{j}])} \mbox{  if }  3\nmid j \oplus \frac{H^{*}(G;{E_{2}}_{*})\{x^{3j}\}}{(\alpha x^{3j}, \beta x^{3j})} \mbox{  if } j\geq 2.
\end{aligned}
\]
Using the multiplicative structure, we may write $z=x^{3}$ and obtain  
\[ \EE_9(EO_2^\ast(\C P^\infty))  \cong    \frac{\frac{H^\ast(G;{E_2}_\ast)[z]}{(\alpha z, \beta z^2)} \{[3x], [3x^2], [\alpha x] \}}{((\alpha,3,z) [\alpha x],  (\alpha, \beta)[3x], (\alpha,\beta)[3x^2])}.\]
In this expression the only terms that are in positive $s$-degree are $\beta^t z$ and $\beta^t [\alpha x]$. It follows that all the other differentials in the filtration spectral sequence are $0$, and so, this is the $\EE_\infty$-page. For finite $n$, it follows that all classes in the first $3$ lines of \eqref{e9fsseo2cp} survive to the $\EE_\infty$-page. For the other classes, note that $d_{\geq 9}$ is $0$ as the image lies in $r$-filtration $> 2m$. Therefore, $\EE_9(EO_2^\ast(\C P^m)) = \EE_\infty(EO_2^\ast(\C P^m))$. 

\begin{prop}\label{einffsseo2cp}
Define $A_{3,k}$ as the algebra 
\[A_{3,k}= \frac{\frac{H^\ast(G;{E_2}_\ast)[z]}{(\alpha z, \beta z^2, z^{k+1})} \{[3x], [3x^2], [\alpha x]\}}{((\alpha,3,z) [\alpha x], (\alpha, \beta)[3x], (\alpha,\beta)[3x^2])}.\]
In this notation, the $\EE_\infty$-page of the filtration spectral sequence is given by 
\[\begin{aligned} 
&\EE_\infty(EO_2^\ast(\C P^{3k}))  \cong  \frac{\big(A_{3,k}\oplus H^\ast(G;{E_2}_\ast)\{[x^{3k-1}]\}\big)}{(z^{k-1} [3x^2]-3[x^{3k-1}], z[x^{3k-1}], z^k[3x], z^k[3x^2])},\\
&\EE_\infty(EO_2^\ast(\C P^{3k+1}))  \cong  \frac{A_{3,k}\oplus \frac{H^\ast(G;{E_2}_\ast)\{[\alpha x^{3k-1}]\}}{((\alpha,3)[\alpha x^{3k-1}])} \oplus \frac{H^\ast(G;{E_2}_\ast)\{[x^{3k+1}]\}}{(\alpha [x^{3k+1}])}} {(z^{k} [3x]-3[x^{3k+1}], z[\alpha x^{3k-1}], z [x^{3k+1}], z^k[3x^2])},\\
& \EE_\infty(EO_2^\ast(\C P^{3k+2}))  \cong \frac{A_{3,k}\oplus \frac{H^\ast(G;{E_2}_\ast)\{[\alpha x^{3k-1}]\}}{((\alpha,3)[\alpha x^{3k-1}])} \oplus \frac{H^\ast(G;{E_2}_\ast)\{[x^{3k+1}]\}}{(\alpha[x^{3k+1}])}\oplus H^\ast(G;{E_2}_\ast)\{[x^{3k+2}]\}}{\begin{pmatrix} 
z^k[3x]-3[x^{3k+1}], & z^{k} [3x^2]-3[x^{3k+2}], & z[x^{3k+1}],\\
 z[x^{3k+2}], & z[\alpha x^{3k-1}]-\alpha[x^{3k+2}]
\end{pmatrix}}, \\
 &\EE_\infty(EO_2^\ast(\C P^\infty))\cong   A_{3,\infty} \cong \frac{\frac{H^\ast(G;{E_2}_\ast)[z]}{(\alpha z, \beta z^2)} \{[3x], [3x^2], [\alpha x]\}}{((\alpha,3,z) [\alpha x], (\alpha, \beta)[3x], (\alpha,\beta)[3x^2])}.
\end{aligned}\]
 \end{prop}

\vspace*{0.2cm}

Let $E_2(EO_2^\ast(\C P^m))$ denote the $2^{nd}$-page of the homotopy fixed point spectral sequence. 

\begin{notation} \label{noteo2cpn}
We use the following notations
\[ x_{(\alpha)} := [\alpha x] \in E_2^{1, 1}(EO_2^\ast(\C P^m)), \quad x_{(j,3)}:= [3 x^j] \in E_2^{0,-2j}(EO_2^\ast(\C P^m)) \mbox{ for } j\in \{1,2\} ,\]
\[ x_{(j)}:= [x^j] \in E_2^{0,-2j}(EO_2^\ast(\C P^m)) \mbox{ if } j=n \mbox{ or } n-1, 3\nmid j,\]
 \[ x_{(3k-1, \alpha)}:= [\alpha x^{3k-1}] \in E_2^{1,5-6k}(EO_2^\ast(\C P^{3k+2})) \mbox{ and } E_2^{1,5-6k}(EO_2^\ast(\C P^{3k+1})),\]
\end{notation}

\vspace*{0.2cm}

 As in \eqref{extfsscalpha}, we have the multiplicative extensions 
\[\alpha x_{(\alpha)} = \beta z, \alpha x_{(3k-1, \alpha)}= \beta x_{(3k+1)}. \]
We first compute the spectral sequence for $n=\infty$. In the above notation 
\[E_2(EO_2^\ast(\C P^\infty))  \cong    \frac{\frac{H^\ast(G;{E_2}_\ast)[z]}{(\alpha z, \beta z^2)} \{x_{(1,3)}, x_{(2,3)}, x_{(\alpha)} \}}{(\alpha x_{(\alpha)}-\beta z, (3,z) x_{(\alpha)}, (\alpha, \beta)x_{(1,3)}, (\alpha,\beta) x_{(2,3)})}.\]
Here, we note that the only classes with positive $s$-degree are multiples of $1$, $x_{(\alpha)}$, and $z$, and that due to degree reasons, $z$, $x_{(\alpha)}$, $x_{(1,3)}$ and $x_{(2,3)}$ are permanent cycles. We note that the first non-trivial differential is $d_5$, and from \eqref{dhfpsseo2} and \eqref{d5hfpsscalph}, we have 
\[ d_5(\Delta)= \alpha \beta^2, \quad d_5(\Delta x_{(\alpha)})= \beta^3 z.\]
Thus, we have 
\begin{myeq}\label{e6hfpsseo2cpinf}
E_6(EO_2^\ast(\C P^\infty))  \cong  \frac{\frac{ \F_3[\alpha, \beta,\Delta^{\pm 3},z]}{(\alpha^2, \alpha \beta^2, \alpha z, \beta z^2, \beta^3 z) }{\resizebox{.3\hsize}{!}{$\{1,[\Delta \alpha], [\Delta^2 \alpha], [\Delta z], [\Delta^2 z], x_{(\alpha)}\}$}}}{\resizebox{.5\hsize}{!}{$\begin{pmatrix} 
\alpha x_{(\alpha)}-\beta z, & (\alpha,z,\beta^2) [\Delta \alpha] , & (\alpha, z)[\Delta^2 \alpha], \\
(\alpha, \beta z, \beta^3) [\Delta z],& (\alpha, \beta z, \beta^3)[\Delta^2 z], & zx_{(\alpha)}
\end{pmatrix}$}} \pmod{(M_3,x_{(1,3)}, x_{(2,3)})}. 
\end{myeq}
The next possible differential is $d_9$. We use \eqref{dhfpsseo2} and \eqref{d9hfpsscalph} to obtain 
\[d_9([\Delta^2 \alpha]=\beta^5, \quad d_9([\Delta^2 z])=\beta^4 x_{(\alpha)}. \]
Using this we compute the $10^{th}$-page as (recall the notation $b=[\Delta \alpha]$)
\begin{myeq}\label{e10hfpsseo2cpinf}
E_{10}(EO_2^\ast(\C P^\infty))  \cong  \frac{\frac{ \F_3[\alpha, \beta,\Delta^{\pm 3},z]}{(\alpha^2, \alpha \beta^2, \beta^5, \alpha z, \beta z^2, \beta^3 z) }\{1,b, [\Delta z],  x_{(\alpha)}\}}{\resizebox{.4\hsize}{!}{ $\begin{pmatrix}
\alpha x_{(\alpha)}-\beta z,& \beta^4 x_{(\alpha)}, &(\alpha,z,\beta^2) b , \\
 (\alpha, \beta z, \beta^3) [\Delta z],& zx_{(\alpha)}
\end{pmatrix}$}} \pmod{(M_3,x_{(1,3)}, x_{(2,3)})}.
\end{myeq}
We have no further room for differentials, and so $E_{10}(EO_2^\ast(\C P^\infty)) \cong E_\infty(EO_2^\ast \C P^\infty))$. There is a multiplicative extension given by $\alpha b = \beta^3$. We summarize this computation in the following theorem using the notation
\[z_{(\Delta)} := [\Delta z] \in EO_2^{-18}(\C P^\infty).\] 
\begin{thm} \label{eo2cpinfty}
The $EO_2$-cohomology of $\C P^\infty$ is described by 
\[EO_2^\ast(\C P^\infty) \cong \frac{\frac{{EO_2}_\ast [z, z_{(\Delta)}]}{((\alpha,\beta^3,\beta z) z, (\alpha, \beta z, \beta^3) z_{(\Delta)}, z_{(\Delta)}^3-\Delta^3z^3 )}\{1, x_{(\alpha)}, x_{(1,3)}, x_{(2,3)}\}}{ (\alpha x_{(\alpha)} - \beta z, bz-\beta^2 x_{(\alpha)}, b x_{(\alpha)} - \beta z_{(\Delta)}, (\beta^4,z) x_{(\alpha)}, (\alpha,\beta)x_{(1,3)}, (\alpha, \beta)x_{(2,3)}  )} ,\]
where the cohomological degree of the generators are given by $|z|=6$, $|x_{(\alpha)}|=-1$, $|x_{(1,3) }|=2$, $|x_{(2,3)}|=4$, and $|z_{(\Delta)}|=-18$. 
\end{thm}  

\vspace*{0.2cm}

For $m<\infty$, one may proceed similarly and derive the analogues of \eqref{e6hfpsseo2cpinf} and \eqref{e10hfpsseo2cpinf}. The map $\C P^m \to \C P^\infty$ gives a multiplicative map $E_\ast(EO_2^\ast(\C P^\infty))\to E_\ast(EO_2^\ast(\C P^m))$, and this allows us to read off the differentials on the multiples of $x_{(\alpha)}$ and $z$. For the other differentials, we consider maps $\C P^m \to \C P^m/\C P^{m-4}$, which induces $E_2(EO_2^\ast(\C P^m/\C P^{m-4}))\to E_2(EO_2^\ast(\C P^m))$. Now $\C P^m/\C P^{m-4}$ is a wedge of $C(\alpha_1)$s and spheres, so we can understand the differentials from \eqref{d5hfpsscalph} and \eqref{d9hfpsscalph}. We obtain following Proposition \ref{einffsseo2cp} (see Notation \ref{noteo2cpn})
\[ x_{(3k-1)}, \mbox{ } x_{(3k+1)}, \mbox{ and } x_{(3k+2)}  \mbox{  are permanent cycles}, \]
\[  \alpha x_{(3k-1,\alpha)} = \beta x_{(3k+1)}, \quad d_5(\Delta x_{(3k-1, \alpha)})= \beta^3 x_{(3k+1)}, d_9 ([\Delta^2 x_{(3k+1)}])= \beta^4 x_{(3k-1, \alpha)}.\]
We denote $x_{(3k+1,\Delta)}:=[\Delta x_{(3k+1)}]\in EO_2^{6k-18}(\C P^{3k+e})$ for $e\in \{1,2\}$.  The relation $\alpha x_{(3k-1,\alpha)} = \beta x_{(3k+1)}$ also gives us the relation $b x_{(3k-1,\alpha)}=\beta x_{(3k+1,\Delta)}$.  
We now compute the $E_\infty$-page and deduce the following result
\begin{thm}\label{eo2cpn}
The $EO_2$-cohomology of $\C P^m$ for $m\geq 4$ is described by
\[ EO_2^\ast(\C P^{3k}) \cong \frac{\frac{EO_2^\ast(\C P^\infty)}{(z^{k+1},z^k x_{(1,3)}, z^k x_{(2,3)})}\{1,x_{(3k-1)}\}}{(z^{k-1}x_{(2,3)}-3x_{(3k-1)}, x_{(1,3)}x_{(3k-1)}-3z^k, (z,z_{(\Delta)}, x_{(\alpha)}, x_{(2,3)})x_{(3k-1)}) },\]
\[EO_2^\ast(\C P^{3k+1}) \cong \frac{\frac{EO_2^\ast(\C P^\infty)}{(z^{k+1}, z^k x_{(2,3)})}\{1,x_{(3k-1,\alpha)}, x_{(3k+1)},x_{(3k+1,\Delta)}\}}{\begin{pmatrix} 
\alpha x_{(3k-1, \alpha)}-\beta x_{(3k+1)}, &  (\beta^4,z, I) x_{(3k-1,\alpha)}, &  (\alpha,  z, \beta^3,I) x_{(3k+1)},  \\  
bx_{(3k-1,\alpha)} - \beta x_{(3k+1,\Delta)}, &z^k x_{(1,3)} - 3 x_{(3k+1)}, &(\alpha,z,b,\beta^3,I)x_{(3k+1,\Delta)}, \\
bx_{(3k+1)}-\beta^2 x_{(3k-1,\alpha)}
\end{pmatrix}} ,\]
\[EO_2^\ast(\C P^{3k+2}) \cong \frac{\frac{EO_2^\ast(\C P^\infty)}{(z^{k+1})}\{1,x_{(3k-1,\alpha)}, x_{(3k+1)},x_{(3k+1,\Delta)}, x_{(3k+2)} \}}{\resizebox{.8\hsize}{!}{$\begin{pmatrix} 
\alpha x_{(3k-1, \alpha)}-\beta x_{(3k+1)}, &  (\beta^4,I) x_{(3k-1,\alpha)}, &  (\alpha,  z, \beta^3, z_{(\Delta)}, x_{(2,3)}) x_{(3k+1)},  \\  
bx_{(3k-1,\alpha)} - \beta x_{(3k+1,\Delta)}, &z^k x_{(1,3)} - 3 x_{(3k+1)}, &(\alpha,z,b,\beta^3, z_{(\Delta)}, x_{(2,3)})x_{(3k+1,\Delta)},\\
 zx_{(3k-1,\alpha)}-\alpha x_{(3k+2)}, & (z,I) x_{(3k+2)}, & z^kx_{(2,3)}-3 x_{(3k+2)}, \\
x_{(\alpha)}x_{(3k+1)}-\alpha x_{(3k+2)}, & x_{(\alpha)} x_{(3k+1,\Delta)} - b x_{(3k+2)}, & x_{(1,3)} x_{(3k+1)}-3 x_{(3k+2)}  \\
bx_{(3k+1)}-\beta^2 x_{(3k-1,\alpha)}
\end{pmatrix}$}}  ,\]
where $I=(x_{(\alpha)}, x_{(1,3)},x_{(2,3)}, z_{(\Delta)}) \subset EO_2^\ast(\C P^\infty)$. 
\end{thm}

	\end{mysubsection}

\vspace*{0.2cm}

\begin{mysubsection}{The $EO_2$-cohomology of $\Hyp P^m$} The cohomology $EO_2^\ast(\Hyp P^m)$ of the quaternionic projective space may be directly read off from $EO_2^\ast(\C P^m)$. In the case $n=\infty$, the spectrum $\Sigma^\infty \Hyp P^\infty$ is the summand $P_0$ of $\Sigma^\infty \C P^\infty$. It follows that $EO_2^\ast \Hyp P^m$ is precisely the subring of $EO_2^\ast \C P^m$ comprising the even generators. We use the following notations
\[
\begin{aligned} 
& y \in EO_2^{12}(\Hyp P^\infty) \mbox{ corresponds to the class } z^2 \in EO_2^{12}(\C P^\infty)\\
&y_{(\Delta)} \in EO_2^{-12}(\Hyp P^\infty) \mbox{ corresponds to the class } zz_{(\Delta)} \in EO_2^{-12}(\C P^\infty)\\
&y_{(\Delta^2)} \in EO_2^{-36}(\Hyp P^\infty) \mbox{ corresponds to the class } z_{(\Delta)}^2 \in EO_2^{-36}(\C P^\infty)\\
& x_{(4,3)} \in EO_2^{8}(\Hyp P^\infty) \mbox{ corresponds to the class } zx_{(1,3)} \in EO_2^{8}(\C P^\infty)\\
&x_{(6k+2)} \in EO_2^{12k+4}(\Hyp P^{3k+1}) \mbox{ corresponds to the class } x_{(6k+2)} \in EO_2^{12k+4}(\C P^{6k+2})\\
&x_{(6k+4)} \in EO_2^{12k+8}(\Hyp P^{3k+2}) \mbox{ corresponds to the class } x_{(6k+4)} \in EO_2^{12k+8}(\C P^{6k+4})\\
&x_{(6k+4,\Delta)} \in EO_2^{12k-16}(\Hyp P^{3k+2}) \mbox{ corresponds to the class } x_{(6k+4,\Delta)} \in EO_2^{12k-16}(\C P^{6k+4})\\
\end{aligned}
\]

\vspace*{0.2cm}

In terms of this notation, we have the following result that may be read off from Theorem \ref{eo2cpn} and Theorem \ref{eo2cpinfty}.
\begin{thm}\label{eo2hpn}
The $EO_2$-cohomology of $\Hyp P^\infty$ is described by 
\[EO_2^\ast(\Hyp P^\infty) \cong \frac{\frac{{EO_2}_\ast [y]}{((\alpha, \beta) y)}\{1, y_{(\Delta)},y_{(\Delta^2)}, x_{(2,3)}, x_{(4,3)}\}}{  ((\alpha, \beta ) y_{(\Delta)},(\alpha,\beta) y_{(\Delta^2)},  (\alpha,\beta)x_{(2,3)}, (\alpha, \beta)x_{(4,3)}  )} .\]

The $EO_2$-cohomology of $\Hyp P^m$ for $n\geq 3$ is described by
\[ EO_2^\ast(\Hyp P^{3k}) \cong\frac{EO_2^\ast(\Hyp P^\infty)}{(y^{k+1},y^k x_{(2,3)}, y^k x_{(4,3)})},\]
\[ EO_2^\ast(\Hyp P^{3k+1}) \cong \frac{\frac{EO_2^\ast(\Hyp P^\infty)}{(y^{k+1}, y^k x_{(4,3)})}\{1,x_{(6k+2)}\}}{((y,I)x_{(6k+2)}, y^k x_{(2,3)} - 3 x_{(6k+2)})} ,\]
\[EO_2^\ast(\Hyp P^{3k+2}) \cong \frac{\frac{EO_2^\ast(\Hyp P^\infty)}{(y^{k+1}, y^k x_{(2,3)})}\{1,x_{(6k+2,\alpha)}, x_{(6k+4)},x_{(6k+4,\Delta)}\}}{\begin{pmatrix} 
\alpha x_{(6k+2, \alpha)}-\beta x_{(6k+4)}, &  (\beta^4,y, I) x_{(6k+2,\alpha)}, &  (\alpha,   \beta^3,y, I) x_{(6k+4)},  \\  
bx_{(6k+2,\alpha)} - \beta x_{(6k+4,\Delta)}, &y^k x_{(2,3)} - 3 x_{(6k+2)}, &(\alpha,b,\beta^3, y, I)x_{(6k+4,\Delta)},\\
 y^k x_{(4,3)}-3 x_{(6k+4)} , & bx_{(6k+4)}-\beta^2 x_{(6k+2,\alpha)} 
\end{pmatrix}}  ,\]
where $I=(y_{(\Delta)}, y_{(\Delta^2)}, x_{(2,3)},x_{(4,3)} ) \subset EO_2^\ast(\Hyp P^\infty)$. 
\end{thm}
\end{mysubsection}

\vspace*{0.2cm}

\section{Higher Real-K theory of complex projective spaces at primes $\geq 5$} \label{higherkthgeq5}

We now compute $EO_{p-1}^\ast(\C P^m)$ for $p\geq 5$. The method is analogous to the $p=3$ case, in the sense that we apply the filtration spectral sequence to compute $H^\ast(C_p ; E_{p-1}(\C P^m))$, and then deduce the differentials of the homotopy fixed point spectral sequence from those of the homotopy fixed point spectral sequence of $S^0$. As in the $p=3$ case, we first analyze the $EO_{p-1}$-homology of complexes formed via iterated cofibers of $\alpha_1$. Throughout this section, we use $n=p-1$, and $\simeq$ for equalities up to a unit in $\F_p$. We recall the homotopy fixed point spectral sequence for $EO_n$
\[ E_2(EO_n) \cong H^\ast(G;{E_n}_\ast) \cong \F_p[\alpha,\beta,\Delta^{\pm 1}]/(\alpha^2) \pmod{M_p},\]
\[d_{2n+1}(\Delta)= \alpha \beta^n, \quad d_{2n^2 +1}(\Delta^{p-1}\alpha)=\beta^{n^2+1},\]
\[E_\infty^{\even, \ast}(EO_n) \cong \frac{\F_p[\Delta^{\pm p}][\beta]}{(\beta^{n^2+1})} \pmod{M_p}, \quad E_\infty^{\odd, \ast}(EO_n) \cong \frac{\alpha \F_p[\Delta^{\pm p}]\{1,\Delta,\dots, \Delta^{p-2}\}[\beta]}{(\beta^n)}.\]
The map from the Adams Novikov spectral sequence to the homotopy fixed point spectral sequence of $EO_n$ sends $\beta_{p/p}$ to $\Delta \beta$. Now from the arguments in \cite[Page 139]{Rav86}, we note $\epsilon^{(i)}=\alpha_1 \beta_{p/p}^i \mapsto [\Delta^i \alpha]\beta^i$. The equation $\alpha_1 \epsilon^{(p-2)}\simeq \beta_1^{1+(p-2)p}$ implies the multiplicative extension
\begin{myeq} \label{extnalpheon}
\alpha \beta^{n-1} [\Delta^{p-2}\alpha] \simeq \beta^{n^2} \implies \alpha [\Delta^{p-2} \alpha] \simeq \beta^{n^2 - n +1}.
\end{myeq} 
Further for degree reasons, $\alpha [\Delta^j \alpha]=0$ for $j\leq p-3$. Analogously the calculation of the Massey products (with $p-i-1$ $\alpha_1$s)
\[\langle \alpha_1, \dots, \alpha_1 \epsilon^{(i)}\rangle \simeq \beta_1^{1+ip},\]
yields the following formula with Toda brackets in $\pi_\ast EO_n$
\begin{myeq}\label{todabreon}
\langle \alpha, \dots, \alpha, [\Delta^{p-j}\alpha]\rangle \simeq \beta^{n^2 - (j-1)n+1}, \quad 1\leq j \leq p-3, 
\end{myeq}
where there are ($j-1$) terms equal to $\alpha$ in the expression above. 
\vspace{0.2cm}

\begin{mysubsection}{The homotopy groups of $EO_n\wedge C(\alpha_{1})$}
	The $\EE_{2}$-page of the filtration spectral sequence for $C(\alpha_{1})$ is given by,
	\[\EE_{2}(EO_n\wedge C(\alpha_{1}))\cong H^*(G;{E_n}_*)\{\iota_{0},\iota_{2n}\}\]
	The only possible non trivial differential is given by, $d_{2n}(\iota_{2n})=\alpha \iota_0$. This implies that the $\EE_{\infty}$-page of the filtration spectral sequence is given by
	\[\EE_{\infty}(EO_n \wedge C(\alpha_1)) \cong \frac{H^*(G;{E_n}_*)\{[\iota_{0}],[p\iota_{2n}],[\alpha \iota_{2n}]\}}{\big(\alpha[\iota_{0}],\alpha[\alpha\iota_{2n}],(\alpha,\beta)[p\iota_{2n}]\big)}.\]
	Unlike the case when $p=3$, for $p \geq 5$, there are no multiplicative extensions for degree reasons. 
% such as $\alpha[\alpha \iota_{2n}]=\beta \iota_{0}$ in the $E_2$-page of the homotopy fixed point spectral sequence of $C(\alpha_{1})$. 
Let $\kappa_{2n,\alpha}:=[\alpha \iota_{2n}]$, $\chi_{2n}:=[p \iota_{2n}]$ with $ |\kappa_{2n,\alpha}|=(1,4n-1)$, and $|\chi_{2n}|=(0,2n) $.  Thus the $E_2$-page of the homotopy fixed point spectral sequence can be written as
	\[E_{2}(EO_n\wedge C(\alpha_{1})) \cong \frac{H^*(G;{E_n}_*)\{\iota_{0}, \kappa_{2n,\alpha},\chi_{2n}\}}{\big((\alpha,\beta)\chi_{2n}, \alpha \kappa_{2n,\alpha}, \alpha \iota_{0} \big)} \cong \F_p[\Delta^{\pm},\beta]\{\iota_{0},\kappa_{2n,\alpha}\} \pmod {M_p, \chi_{2n}}. \] 
	From the cofiber sequence
	\[S^0 \xrightarrow{\iota_{0}} C(\alpha_{1}) \xrightarrow{j} S^{2n}\xrightarrow{\alpha_1}S^1, \]
	it follows that $\iota_{0}, \chi_{2n}$ are permanent cycles as they lie in the Hurewicz image $\pi_{*}C(\alpha_{1})\to {EO_n}_*C(\alpha_{1})$. A direct calculation using $\alpha^2=0$ and \eqref{extnalpheon} in $\pi_\ast EO_n$ yields $\pmod{M_p}$
\[\Ker( EO_n \wedge S^{2n} \xrightarrow{\alpha_1} EO_n \wedge S^1)= \F_p[\Delta^{\pm p}]\otimes\big[\F_p \{\beta^n,\dots, \beta^{n^2}\} \oplus \frac{\F_p[\beta]}{(\beta^n)}\{\alpha,[\Delta \alpha], \dots, [\Delta^{p-3}\alpha]\}\big], \]
\[\Coker( EO_n \wedge S^{2n-1} \xrightarrow{\alpha_1} EO_n)= \F_p[\Delta^{\pm p}]\otimes\big[\frac{\F_p[\beta]}{(\beta^{n^2-n+1})} \oplus \frac{\F_p[\beta]}{(\beta^n)}\{[\Delta \alpha], \dots, [\Delta^{p-2}\alpha]\}\big]. \]
This shows that the following are permanent cycles
\[ \iota_0, \kappa_{2n,\alpha}, \Delta^j \kappa_{2n,\alpha} \mbox{ for } 1\leq j \leq p-3. \]
For degree reasons, the class $\beta^n \iota_{2n}$ in the $\Ker$ part is represented by the class $\Delta \iota_0$ up to a unit of $\F_p$, which implies that it must be a permanent cycle as well. 
%
%
%$\kappa_{2n,\alpha}$ is also a permanent cycle. Applying the multiplicative extension \eqref{extnalpheon}, we see that $\Delta^j \kappa_{2n,\alpha}$ is a permanent cycle for $j\leq p-3$ and $\Delta^{p-2}\kappa_{2n,\alpha}$ must support a differential.  In order to figure out the differentials in $E_\ast(EO_n\wedge C(\alpha_1))$ it suffices to work modulo $M_p$. Note that,
%	\[E_{2}(EO_n\wedge C(\alpha_{1})) \cong \F_p[\Delta^{\pm},\beta]\{\iota_{0},\kappa_{2n,\alpha}\} \pmod {M_p, \chi_{2n}}.\]
As $\Delta^{\pm p}$, $\alpha$ and $\beta$ are permanent cycles in $E_2(EO_n)\cong H^\ast(G;{E_n}_\ast)$, and $E_2(EO_n\wedge C(\alpha_1))$ is a module over $E_2(EO_n)$, the products of these classes with $\iota_0$ and $\Delta^j \kappa_{2n,\alpha}$ for $j\leq p-3$, are permanent cycles as well. 

We now note that $\beta^{2n} \Delta^j \kappa_{2n,\alpha}$ cannot survive to $E_\infty$. This is because, the class represented by $\Delta^j\kappa_{2n,\alpha}$ is an odd degree class, and in $\pi_\ast(EO_n)$, $\beta^n (\mbox{odd degree class} ) =0$. Therefore, $\beta^n \Delta^j\kappa_{2n,\alpha}$ belongs to the image of $\pi_\ast EO_n \to \pi_\ast EO_n \wedge C(\alpha_1)$, and thus becomes $0$ after a further multiplication by $\beta^n$. Hence, $\beta^{2n}\kappa_{2n,\alpha}$ is the target of a differential, and for degree reasons the only possibility is 
\[ d_{4n+1}(\Delta^{j+2}\iota_0)\simeq \beta^{2n} \Delta^j \kappa_{2n,\alpha}.\] 
We obtain 
\[\begin{aligned}
E_{4n+2}(EO_n\wedge C(\alpha_1))= \F_p[\Delta^{\pm p}, \beta] \otimes \Big[ & \F_p\{ \iota_0, \Delta \iota_0, \Delta^{p-2}\kappa_{2n,\alpha}, \Delta^{p-1}\kappa_{2n,\alpha}\} \oplus \\
& \frac{\F_p\{\Delta^j \kappa_{2n,\alpha} \mid 0\leq j \leq p-3 \}}{(\beta^{2n})}\Big].
\end{aligned} \]
From the formula of the cokernel of $\alpha_1$, we see that $\beta^{n^2-n+1}\iota_0$ must be killed by a differential, and so must $\beta^{n^2 - n+1}\Delta \iota_0$. Again by degree reasons, the only way this can happen is 
\[ d_{2n^2-2n+1}(\Delta^{p-2} \kappa_{2n,\alpha})\simeq \beta^{n^2 - n +1}\iota_0, \quad d_{2n^2-2n+1}(\Delta^{p-1}\kappa_{2n,\alpha})\simeq\beta^{n^2-n+1}\Delta \iota_0.\]
This implies 
\begin{myeq}\label{compc[k]at1}
\resizebox{13.5cm}{!}{$ E_{2n^2-2n+2}(EO_n\wedge C(\alpha_1))= \F_p[\Delta^{\pm p}, \beta] \otimes \Big[ \frac{\F_p\{ \iota_0, \Delta \iota_0\}}{(\beta^{n^2-n+1})} \oplus  \frac{\F_p\{\Delta^j \kappa_{2n,\alpha} \mid 0\leq j \leq p-3 \}}{(\beta^{2n})}\Big]$}.
 \end{myeq}
There cannot be any further differentials, so this is also the $E_\infty$-page. From this expression, we deduce the structure of $\pi_\ast(EO_n \wedge C(\alpha_1))$ by figuring out the multiplicative extensions. These involve multiplications by $\alpha$ and $[\Delta^j \alpha]$ for $1\leq j \leq p-2$. We readily note that $[\Delta^j\alpha]\iota_0=[\Delta^{j-1} \alpha]\Delta \iota_0$ for $1\leq j \leq p-2$. From the computation of the $E_\infty$-page, we note that $\beta^n \Delta^j\kappa_{2n,\alpha}\neq 0$ for $0\leq j \leq p-3$ and belongs to the image $EO_n \to EO_n\wedge C(\alpha_1)$, so that from degree reasons, it must equal an unit multiple of $[\Delta^{j+1}\alpha]\iota_0$. 

Note that $\alpha [\Delta^{p-3}\alpha]=0$ implies that $\alpha \Delta^{p-3}\kappa_{2n,\alpha}$ belongs to the image $EO_n \to EO_n \wedge C(\alpha_1)$, and is the Toda bracket $\langle \alpha, \alpha, [\Delta^{p-3}\alpha]\rangle \iota_0$, which in turn equals an unit multiple of $\beta^{n^2-2n+1}$ by \eqref{todabreon}.  

On the other hand under the map $EO_n \wedge C(\alpha_1)\to \Sigma^{2n} EO_n$, the class $[\Delta \alpha] \Delta^{p-3}\kappa_{2n,\alpha}$ maps to $\alpha [\Delta^{p-2}\alpha]\iota_{2n}$ which equals an unit multiple of $\beta^{n^2-n+1}\iota_{2n}$ by \eqref{extnalpheon}. As $\beta^n \iota_{2n}$ is represented by $\Delta \iota_0$, this implies $[\Delta \alpha] \Delta^{p-3}\kappa_{2n,\alpha}\simeq \beta^{n^2-2n+1}\Delta\iota_0$. We summarize all these multiplicative extensions in the equation below.
\begin{myeq}\label{multcalph}
\begin{aligned}
[\Delta^j \alpha]\Delta \iota_0 = [\Delta^{j+1}\alpha] \iota_0 \simeq \beta^n \Delta^j \kappa_{2n,\alpha} \mbox{ for } 0 \leq j \leq p-3, \\
\alpha \Delta^{p-3}\kappa_{2n,\alpha} \simeq \beta^{n^2 -2n+1}\iota_0, \quad [\Delta\alpha] \Delta^{p-3}\kappa_{2n,\alpha} \simeq \beta^{n^2 -2n+1} \Delta \iota_0.
\end{aligned}
\end{myeq}

\end{mysubsection}

\vspace{0.2cm}

\begin{mysubsection}{The $EO_n$-homology of iterated cones of $\alpha_1$} \label{itercones}
We next compute the $EO_n$-homology of finite spectra obtained as iterated cones of $\alpha_1$. These spectra have a cell complex structure with cells in dimensions $0, 2n, 4n, \dots , 2kn$ for $1\leq k \leq p-1$ such that the attaching map of the $2jn$-cell onto the $2(j-1)n$-cell is a unit multiple of $\alpha_1$. We denote such a finite spectrum by $C^{[k]}$. As in the $p=3$ case, these spectra are not unique and our computational results depend only on $k$. Their existence may be inductively deduced from $\delta \circ \alpha_1 =0$ in Figure \ref{C(alpha)-diagram} using the fact that $\langle \alpha_1, \dots, \alpha_1 \rangle = 0$ when the number of $\alpha_1$ in the expression is $\leq p-1$. It follows that $\alpha_1$ lifts to a map $S^{2kn} \to C^{[k-1]}$ and taking the homotopy cofiber we obtain a $C^{[k]}$.

\footnotesize	
\begin{figure}[h]
	\centering
	\begin{tikzpicture} [
		>=Stealth,
		every node/.style={font=\small},
		scale=1.1
		]
		
		%------------------------------------------------
		% Top row
		%------------------------------------------------
		
		\node (A) at (0,0)
		{$C^{[k-2]}$};
		
		\node (B) at (3,0)
		{$C^{[k-1]}$};
		
		\node (C) at (6,0)
		{$S^{2(k-1)n}$};
		
		\node (D) at (9,0)
		{$\Sigma C^{[k-2]}$};
		
		\node (E) at (11,1.2)
		{$\Sigma C^{[k-3]}$};
		
		\node (F) at (13,2.2)
		{$S^{1}$};
		
		%------------------------------------------------
		% Bottom row
		%------------------------------------------------
		
		\node (G) at (6,-2)
		{$S^{2kn-1}$};
		
		\node (H) at (9,-2)
		{$S^{2(k-2)n+1}$};
		
		\node (I) at (11,-1.3)
		{$S^{2(k-3)n-1}$};
		
		%------------------------------------------------
		% Horizontal arrows
		%------------------------------------------------
		
		\draw[->] (A) -- (B);
		
		\draw[->] (B) -- (C);
		
		\draw[->] (C) -- node[below] {$\delta$} (D);
		
		\draw[->] (G) -- node[above]{$\simeq \alpha_1^2$}(H);
		
		%------------------------------------------------
		% Vertical arrows
		%------------------------------------------------
		
		\draw[->] (G) -- node[left] {$\alpha_1$} (C);
		
		\draw[->] (D) -- (H);
		
		\draw[->] (E) -- (I);
		
		%------------------------------------------------
		% Diagonal arrows
		%------------------------------------------------
		
		\draw[->] (E) -- (D);
		
		%\draw[->,dashed] (H) -- (I);
		
		\draw[->,dashed] (F) -- (E);
		
		%------------------------------------------------
		% Labels on lower horizontal map
		%------------------------------------------------

		%------------------------------------------------
		% Dashed curved brackets
		%------------------------------------------------
		
		\draw[->,dashed]
		(G)
		.. controls (7.5,-4) and (10,-4) ..
		node[above]{$\simeq \langle \alpha_1,\alpha_1,\alpha_1\rangle$}
		(I);
		
		\draw[->,dashed]
		(G)
		.. controls (8,-5.2) and (13,-5.7) ..
		node[right]{$\simeq \langle \alpha_1,\dots,\alpha_1\rangle$}
		(F);
		
	\end{tikzpicture}
	\caption{}
	\label{C(alpha)-diagram}
\end{figure}\normalsize	

\vspace{0.2cm}

\begin{rmk}
The spectra $X_k$ of \cite[\S 2]{Cha20} are examples of $C^{[k]}$. The $X_k$ satisfy an additional hypothesis specifying the comodule structure over the dual Steenrod algebra $\pmod{p}$, which allows one to prove the uniqueness of $X_k$ in the category of $p$-complete spectra. 
\end{rmk}

\vspace{0.2cm}

The filtration spectral sequence has $\EE_2$-term 
\[\EE_2(EO_n \wedge C^{[k]}) \cong H^\ast(G;{E_n}_\ast)\{\iota_0,\iota_{2n},\dots, \iota_{2kn}\}\]
with differential 
\[d_{2n}(\iota_{jn})\simeq \alpha \iota_{(j-1)n}, \mbox{ for } 1\leq j \leq k.\]
For $k<p-1$, there are no further differentials. It follows that 
\[\EE_\infty(EO_n\wedge C^{[k]}) \cong \frac{H^*(G;{E_n}_*)\{[\iota_{0}],[p\iota_{2n}],\dots, [p\iota_{2kn}], [\alpha \iota_{2kn}]\}}{\big(\alpha[\iota_{0}],\alpha[\alpha\iota_{2kn}],(\alpha,\beta)[p\iota_{2jn}]\big)}.\]
As in the $k=1$ case, we use the notation 
\[ \kappa_{2nk,\alpha} := [\alpha \iota_{2nk}] \in E_2^{1,2nk+2n-1}(EO_n\wedge C^{[k]}), \quad \chi_{2jn} := [p\iota_{2jn}] \in E_2^{0,2jn}(EO_n \wedge C^{[k]}).\]
Denote by $I_k$ the ideal $(M_p, \chi_{2jn}\mid 1\leq j \leq k)$, which is a submodule of the $0$-line of the homotopy fixed point spectral sequence consisting of permanent cycles. We work modulo $I_k$, and so 
\[E_2(EO_n \wedge C^{[k]}) \cong \F_p[\Delta^{\pm},\beta]\{\iota_0, \kappa_{2nk,\alpha}\} \pmod{I_k}.\] 
The following proposition describes the differentials in the homotopy fixed point spectral sequence. 
\begin{prop}\label{hfpssc[k]}
For $k<p-1$, the differentials in the homotopy fixed point spectral sequence of $EO_n\wedge C^{[k]}$ is given by 
\[d_{2(k+1)n+1}(\Delta^j \iota_0)\simeq \Delta^{j-k-1} \beta^{(k+1)n} \kappa_{2nk,\alpha} \mbox{ for } k+1 \leq j \leq p-1, \]
\[d_{2n^2-2kn+1}(\Delta^{j+p-k-1}\kappa_{2nk,\alpha}) \simeq \beta^{n^2-kn+1}\Delta^j \iota_0 \mbox{ for } 0\leq j \leq k.\]
As a consequence the $E_\infty$-page is described as 
\[E_\infty(EO_n\wedge C^{[k]}) \cong \F_p[\Delta^{\pm p},\beta] \otimes  \Big[ \frac{\F_p\{\Delta^j \iota_0, 1\leq j \leq k\}}{(\beta^{n^2-kn+1})} \oplus  \frac{\F_p\{\Delta^j \kappa_{2nk,\alpha} \mid 0\leq j \leq p-k-2 \}}{(\beta^{(k+1)n})}\Big].\]
The multiplicative extensions are given by 
\begin{align*}
[\Delta^j \alpha]\Delta^k \iota_0= [\Delta^{j+1}\alpha]\Delta^{k-1}\iota_0=\dots = [\Delta^{j+k}\alpha] \iota_0 \simeq \beta^{kn} \Delta^j \kappa_{2nk,\alpha} \mbox{ for } 0 \leq j \leq p-k-2, \\
\alpha \Delta^{p-k-2}\kappa_{2nk,\alpha} \simeq \beta^{n^2 -(k+1)n+1}\iota_0, \quad [\Delta^j \alpha] \Delta^{p-k-2}\kappa_{2nk,\alpha} \simeq \beta^{n^2 -(k+1)n+1} \Delta^j \iota_0 \mbox{ for } 1\leq j \leq k.
\end{align*}
In addition, $\alpha \Delta^j \kappa_{2nk,\alpha}=0$ for $j\leq p-k-3$. 
\end{prop}

\begin{proof}
We proceed by induction on $k$ observing that the $k=1$ case has been proved in \eqref{compc[k]at1} and \eqref{multcalph}. Assume that the result has been proved for $k-1$. We use the cofiber sequence
\begin{myeq} \label{cofk1}
EO_n\wedge C^{[k-1]}\xrightarrow{\tau_{k-1}} EO_n\wedge C^{[k]} \xrightarrow{q_k} EO_n\wedge S^{2nk}\xrightarrow{\delta_k} \Sigma EO_n \wedge C^{[k-1]},
\end{myeq}
and observe that $\Delta^j \iota_0$ is a permanent cycle for $j<k$ as it lies in the image of $\tau_{k-1}$. Also note $\delta_k(\iota_{2nk})=\kappa_{2n(k-1),\alpha}$ so that 
\[
 \Ker(\delta_k) = \F_p[\Delta^{\pm p}]\otimes \big[ \F_p\{\beta^{kn}\iota_{2nk},\dots, \beta^{n^2}\iota_{2nk}\} \oplus  \frac{\F_p[\beta]}{(\beta^n)}\{\alpha \iota_{2nk}, [\Delta \alpha]\iota_{2nk}, \dots, [\Delta^{p-k-2}\alpha]\iota_{2nk} \} \big].
\]
As $\Delta^j \kappa_{2nk,\alpha}$ represents the class $[\Delta^j\alpha]\iota_{2nk}$, it follows that $\Delta^j\kappa_{2nk,\alpha}$ is a permanent cycle if $0\leq j \leq p-k-2$. The class $\beta^{kn}\iota_{2nk} \in \Ker(\delta_k)$ is an even degree element of $\pi_\ast(EO_n\wedge C^{[k]})$ and is thus represented in $E_2(EO_n\wedge C^{[k]})$ by a multiple of $\iota_0$. For degree reasons, this must be a unit multiple of $\Delta^k \iota_0$, which is therefore, a permanent cycle. This is directly observed from the fact that $\frac{|\beta|}{2}$ and $\frac{|\Delta|}{2}$ are coprime, and so other entries in the $E_2$-page that can represent the class are in higher $s$-filtration than $2kn$. Note that this also means $\beta^{n^2 - kn +1}\Delta^k \iota_0=0$. 
Next we observe 
\[\delta_k([\Delta^{p-k-1+j}\alpha] \iota_{2nk})= [\Delta^j\alpha] \Delta^{p-k-1}\kappa_{2n(k-1),\alpha}= \beta^{n^2-kn+1}\Delta^j \iota_0 \mbox{ for } 0\leq j\leq k-1.\]
Thus, $\beta^{n^2-kn+1}\Delta^j \iota_0$ must be killed by a differential for $0\leq j\leq k$. For degree reasons, the only possibility is 
\begin{myeq}\label{diff1}
d_{2n^2-2kn+1}(\Delta^{j+p-k-1} \kappa_{2nk,\alpha})\simeq\beta^{n^2-kn+1}\Delta^j \iota_0 \mbox{ for } 0\leq j \leq k.
\end{myeq}
We also observe inductively that in $EO_n \wedge C^{[k]}$, any odd degree class multiplied by $\beta^{(k+1)n}$ is $0$. This is true for $k=0$ as $\beta^n[\Delta^j \alpha]=0$ in $\pi_\ast EO_n$. Now consider \eqref{cofk1} and an odd degree class $\chi$ in $\pi_\ast (EO_n \wedge C^{[k]})$. We note that $\beta^nq_k(\chi)=0$ which implies $\beta^n \chi$ is an odd degree class in the image of $\tau_{k-1}$. It follows by the induction hypothesis  that 
\[\beta^{(k+1)n} \chi = \beta^{kn} \tau_{k-1}(\chi')=\tau_{k-1}(\beta^{kn}\chi') =0.\]
Thus, $\beta^{(k+1)n}\Delta^j \kappa_{2nk,\alpha}$ must be killed by a differential for $0\leq j\leq p- k-2$. For degree reasons, the only possibility is 
\begin{myeq}\label{diff2}
d_{2(k+1)n+1}(\Delta^j \iota_0)\simeq \beta^{(k+1)n} \Delta^{j-k-1} \kappa_{2nk,\alpha} \mbox{ for } k+1\leq j \leq p-1.
\end{myeq}
From \eqref{diff1} and \eqref{diff2}, observe that all classes in the $E_r$-page for $r\geq \mbox{max}\{2(k+1)n+2, 2n^2-2kn+2\}$ are permanent cycles. Therefore, the $E_\infty$-page is as described in the Proposition. It remains to prove the multiplicative extensions.

Note that for degree reasons, $\alpha \Delta^j \kappa_{2nk,\alpha}=0$ if $j\leq p-k-3$. The inclusion of the $0$-cell in $C^{[k]}$ induces the cofiber sequence
 \begin{myeq} \label{cofk2}
EO_n\xrightarrow{\tau_0} EO_n\wedge C^{[k]} \xrightarrow{q_0} \Sigma^{2n} EO_n\wedge C^{[k-1]}\xrightarrow{\delta_0} \Sigma EO_n,
\end{myeq}
where we observe that $\delta_0(\iota_{2n})= \alpha \iota_0$ and $q_0(\kappa_{2nk,\alpha})=\kappa_{2nk,\alpha}$. Note that by the induction hypothesis, $q_0 (\alpha \Delta^{p-k-2}\kappa_{2nk,\alpha})=0$ and so, $\alpha \Delta^{p-k-2}\kappa_{2nk,\alpha}$ belongs to the image of $\tau_0$ of \eqref{cofk2}. The identification of higher matric Massey products with higher Toda brackets recognizes this class as an element of $\langle \alpha, \dots, \alpha, [\Delta^{p-k-2}\alpha]\rangle \iota_0$ with $k+1$ terms equal to $\alpha$, which in turn equals an unit multiple of $\beta^{n^2-(k+1)n+1}$ by \eqref{todabreon}.  This implies 
\[\alpha \Delta^{p-k-2}\kappa_{2nk,\alpha}\simeq \beta^{n^2-(k+1)n+1}\iota_0.\]
Further observe that $\delta_0(\beta^n \Delta^j \iota_{2n})=0$ for $0\leq j\leq k-1$, and thus, belongs to the image of $q_0$. For degree reasons, the corresponding element in $\pi_\ast(EO_n \wedge C^{[k]})$ is represented in the $E_\infty$-page by $\Delta^{j+1} \iota_0$. 
On the other hand under the map $q_0$ of \eqref{cofk2},
\begin{align*}
q_0({[\Delta^j \alpha] \Delta^{p-k-2}\kappa_{2nk,\alpha}}) &= {[\Delta^{j-1} \alpha] \Delta^{p-k-1}\kappa_{2nk,\alpha}}\\
                                                                                       & \simeq \beta^{n^2-kn+1}\Delta^{j-1}\iota_{2n} \\
                                                                                      &\simeq q_0(\beta^{n^2-(k+1)n+1}\Delta^j \iota_0) \mbox{ for } 1\leq j \leq k.
\end{align*}
In the degree where this equality occurs, $q_0$ is injective from our computation of the $E_\infty$-term, so the equation $[\Delta^j \alpha]\Delta^{p-k-2}\kappa_{2nk,\alpha}\simeq \beta^{n^2-(k+1)n+1}\Delta^j\iota_0$ follows. 

Finally, consider the class $\beta^{kn}\Delta^j\kappa_{2nk,\alpha}\neq 0$ in the $E_\infty$-page. The map $q_0$ of \eqref{cofk2} sends this class to $0$, as $\beta^{kn}$ annihilates every odd degree class in $EO_n\wedge C^{[k-1]}$. Therefore, it represents a non-zero class in the image of $\tau_0$, which for degree reasons is forced to be an unit multiple of $[\Delta^{j+k}\alpha]\iota_0$. This completes the proof of the proposition. 
\end{proof}

\vspace{0.2cm}

Now we consider the case $k=p-1$. We use the cofiber \eqref{cofk1} and note that Proposition \ref{hfpssc[k]} implies that in $EO_n\wedge C^{[p-2]}$, $\alpha  \kappa_{2n(p-2),\alpha}=\beta \iota_0$. This is demonstrated in Figure \ref{C(p-1)-diagram}, which implies in the filtration spectral sequence
\[d_{2n(p-1)}(\alpha\iota_{2n(p-1)})=\beta \iota_0.\]
The Figure \ref{C(p-1)-diagram} demonstrates the differential in the $EO_n$-Atiyah Hirzebruch spectral sequence, which must appear in the filtration spectral sequence, as in the homotopy fixed point spectral sequence, the differentials change the $s$-degree by at least $2$. 
\begin{figure}[h]
	\centering
	\begin{tikzpicture} [
		>=Stealth,
		every node/.style={font=\small},
		scale=1.1
		]
		
		%------------------------------------------------
		% Top row
		%------------------------------------------------
		
		\node (A) at (0,0)
		{$\Sigma EO_n \wedge C^{[p-2]}$};
		
		\node (B) at (3,0)
		{$\Sigma EO_n \wedge C^{[p-3]}$};
		
		\node (C) at (6,0)
		{$\dots $};
		
		\node (D) at (9,0)
		{$\Sigma EO_n \wedge S^0$};

		%------------------------------------------------
		% Bottom row
		%------------------------------------------------
		
		\node (G) at (0,-2)
		{$EO_n \wedge S^{2n(p-1)}$};

                \node(I) at (1.5,-2)		
              {$\ni$};

		\node (H) at (2.3,-2.05)
		{$ \alpha \iota_{2n(p-1)}$};
		
               \node(J) at (1,-.5)
              {$\delta_{p-1}(\alpha \iota_0)$};

 \node(K) at (9,-.5)
              {$\beta \iota_{2n}$};
		
		%------------------------------------------------
		% Horizontal arrows
		%------------------------------------------------
		
		\draw[->] (B) -- (A);
		
		\draw[->] (C) -- (B);
		
		\draw[->] (D) -- (C);
		
		\draw[->] (G) -- node[left] {$\delta_{p-1}$} (A);
 
\draw[|->] (H) -- (J);
		
		%------------------------------------------------
		% Vertical arrows
		%------------------------------------------------
		
		\draw[|->] (K) --  (J);

	\end{tikzpicture}
	\caption{}
	\label{C(p-1)-diagram}
\end{figure}

\vspace{0.2cm}

As a consequence, we have
	\[\EE_{\infty}(EO_n \wedge C^{[p-1]})\cong \frac{H^*(G;{E_n}_\ast)}{(\alpha,\beta)}\{\iota_{0},[p\iota_{2nk}]\} \text{ where } 1 \leq k \leq n.\]   
There are no non-trivial differentials in the homotopy fixed point spectral sequence, and this is also the $E_\infty$-page. 
\end{mysubsection}

\vspace{0.2cm}

\begin{mysubsection}{The $EO_n$-cohomology of complex projective spaces}
We use the ``Adams splitting" of the spectrum $\Sigma^\infty_+ {\C P^\infty}^\land_p$ which we first recall. The space ${\C P^\infty}_p^\wedge \simeq K(\Z_p,2)$, and an element $\nu\in \Z_p$ induces a map $f_\nu : {\C P^\infty}_p^\wedge \to {\C P^\infty}_p^\wedge$ which is multiplication by $\nu$ in $\pi_2$. For an odd prime $p$, the group of units in the $p$-adic integers is,
\[\Z_p^{\times}=\Z/(p-1) \times \Z_p.\]
Fix a $(p-1)^{th}$-root of unity $\zeta_p$ of $\Z_p$, and let $e_p^{(0)}$ denote the stable self map (on $\Sigma^\infty_+  {\C P^\infty}_p^\wedge $)
\[e_p^{(0)}= \frac{1+f_{\zeta_p}+f_{\zeta_p}^2 +\dots + f_{\zeta_p}^{p-2}}{p-1},\]
which is an idempotent. Using the fact that the self maps of $\Sigma^\infty {\C P^\infty}_p^\wedge$ is a module over $\pi_0(\Sp_p^\wedge)\cong \Z_p$, we also have idempotents $e_p^{(i)}$ for $1\leq i \leq p-2$ given by the formula
\[e_p^{(i)}= \frac{1+\zeta_p^{-i}f_{\zeta_p}+\zeta_p^{-2i}f_{\zeta_p}^2 +\dots + \zeta_p^{-(p-2)i}f_{\zeta_p}^{p-2}}{p-1}.\]
These idempotents satisfy
\[e_p^{(i)}e_p^{(j)}=0 \mbox{ for } i\neq j, \quad \sum_{i=0}^{p-2} e_p^{(i)}=1.\]
This implies a splitting
\begin{myeq}\label{splitcpodd}
\Sigma^\infty_+ {\C P^\infty}_p^\wedge \simeq P_0 \vee P_1 \vee \dots \vee P_{p-2},
\end{myeq}
where $P_j = \Tel(e_p^{(j)})$. The $\F_p$-cohomology of these summands as a subset of $H^\ast(\C P^\infty; \F_p)$ are given by
\[H^*(P_i;\F_p)=\F_p\{x^{i+kn}, k \in \Z_{\geq 0}\}.\]
The skeletal filtration of $\Sigma^\infty_+ {\C P^\infty}_p^\wedge$ also induces a filtration on the $P_i$. We now use the fact that the action of $\PP^1$ on the classes $x^i$ is given by $\PP^1(x^i)=ix^{i+p-1}$, and that $\PP^1$ detects $\alpha_{1}$. This implies the following identifications 
\begin{myeq}\label{skelPi}
\begin{aligned}
&P_{i}^{(2ip)}\simeq \Sigma^{2i}C^{[i]},\quad \frac{P_i^{(2(i+(r+1)n)p)}}{P_i^{(2(i+rn)p)}}\simeq\Sigma^{2((i+rn)p+n)}C^{[p-1]}, \\
&\mbox{For } r\geq 0, \mbox{ } \frac{P_i^{(2m)}}{P_i^{(2(i+rn)p)}}\simeq\Sigma^{2((i+rn)p+n)}C^{\big[[\frac{m-ip}{n}]-rp\big]} \mbox{ if } i+rn\leq \frac{m}{p} < i+(r+1)n,\\
& \mbox{For } m< ip, \mbox{ } P_i^{(2m)} \simeq \Sigma^{2i} C^{\big[[\frac{m-i}{n}]\big]}.
\end{aligned}
\end{myeq}
We also have that the Spanier-Whitehead dual of a $C^{[i]}$ is also a $C^{[i]}$ in the sense that $D(C^{[i]})\simeq \Sigma^{-2in} C^{[i]}$.  

\vspace{0.2cm}

The filtration spectral sequence for the $EO_n$-cohomology of $\C P^m$ has the form,
\[\EE_2(EO_n^\ast(\C P^m))\cong H^\ast (G;\pi_\ast E_n)[x]/(x^{m+1}) \implies H^\ast (G;E^\ast_n(\C P^m)),\]
for $1\leq m \leq \infty$. The formulas for the differentials in the filtration spectral sequences of $C^{[j]}$ for $1\leq j \leq p-1$ implies the following differentials 
\[d_{2n}(x)=\alpha x^p, \quad d_{2n^2}(\alpha x^{pj+n})\simeq\beta x^{p(j+n)}.  \]
Analogous arguments as in the $p=3$ case (see the discussion before Proposition \ref{einffsseo2cp}) imply that there cannot be any further differentials in the filtration spectral sequence. For $m=\infty$, we have the following formula
\begin{myeq}\label{fsscpinfty}
\begin{aligned}
\EE_{\infty}(EO_n^*(\C P^\infty)) &= H^*(G;{E_n}_*)\{1\} \bigoplus_{k=1}^{n-1} \frac{H^*(G;{E_n}_*)}{(\alpha,p)}\{[\alpha x^k]\}\bigoplus_{k=1}^{n-1}\frac{H^*(G;{E_n}_*)}{(\alpha)}\{x^{kp}\}\\
&\bigoplus \frac{H^*(G;{E_n}_*)}{(\alpha,\beta)} \{x^{pj} \mbox{ for } j\geq n, [p x^j] \mbox{ for } p\nmid j\},	
\end{aligned}
\end{myeq}
with the evident multiplicative structure. If $m<\infty$, we can apply \eqref{skelPi} to deduce the expression of $\EE_\infty(EO_n^\ast(\C P^m))$. 
 For $2n \leq m < \infty$, we have
%\[\EE_\infty(EO_n^\ast(\C P^m)) \cong \sum_{i=1}^n \Big(\EE_\infty(\Sigma^{2i}EO_n \wedge C^{[i]}) + \EE_\infty(\Sigma^{2((i+[\frac{m-i}{n}]n)p +n)} EO_n \wedge C^{\big[[\frac{m-ip}{n}]-[\frac{m-i}{n}]p\big]})+  \sum_{r=0}^{[\frac{m-pi}{n}]} \EE_\infty(\Sigma^{2((i+rn)p+n)} EO_n \wedge C^{[p-1]})    \Big) \] 
\begin{myeq}\label{fsscpmlarge}
\begin{aligned}
\EE_{\infty}(EO_n^*(\C P^m)) &= \frac{\EE_\infty(EO_n^\ast(\C P^\infty))}{(x^{m+1})} \bigoplus  H^*(G;{E_n}_*)\{x^{[\frac{m+1}{p}]p-1} \mbox{ if } m \neq -2 \bmod{p}\}\\
%H^*(G;{E_n}_*)\{1\} \bigoplus_{k=1}^{n-1} \frac{H^*(G;{E_n}_*)}{(\alpha,p)}\{[\alpha x^k]\}\bigoplus_{k=1}^{n-1}\frac{H^*(G;{E_n}_*)}{(\alpha)}\{x^{kp}\}\\
& \bigoplus_{k=[\frac{m}{p}]-n +1}^{[ \frac{m-2n}{p}]} \frac{H^*(G;{E_n}_*)}{(\alpha,p)}\{[\alpha x^{kp+n}]\}\bigoplus_{\substack{m-n< k \leq m,\\ k \neq 0,-1 \bmod{p}}}\frac{H^*(G;{E_n}_*)}{(\alpha)}\{x^{k}\},
%& \bigoplus  H^*(G;{E_n}_*)\{x^{[\frac{m+1}{p}]p-1} \mbox{ if } m \neq -2 \bmod{p}\},
%&\bigoplus \frac{H^*(G;{E_n}_*)}{(\alpha,\beta)} \{x^{pj} \mbox{ for } \frac{m}{p}\geq j\geq n, [p x^j] \mbox{ for } p\nmid j \mbox{ and } j\leq m\},	
\end{aligned}
\end{myeq}
and for $m<2n$ (leaving out terms of the form $[px^j]$), 
\begin{myeq}\label{fsscpmsmall}
\begin{aligned}
\EE_{\infty}(EO_n^*(\C P^m)) &= H^*(G;{E_n}_*)\{1, x^{s } \mbox{ for } m-n<s\leq n\}\\
&\bigoplus_{k\leq m-n} \frac{H^*(G;{E_n}_*)}{(\alpha,p)}\{[\alpha x^k]\} \bigoplus_{n< k\leq m}\frac{H^*(G;{E_n}_*)}{(\alpha)}\{x^{k}\}.
\end{aligned}
\end{myeq}

\vspace{0.2cm}

The next step is to work out the homotopy fixed point spectral sequence in the case $m=\infty$ using the computation of the filtration spectral sequence \eqref{fsscpinfty}. As in the $p=3$ case, the formula for the differentials follow from those of $E_\ast(EO_n \wedge C^{[i]})$ which are described in Proposition \ref{hfpssc[k]}. Working modulo $M_p$, we see that the following are permanent cycles in $E_\ast (EO_n^\ast(\C P^\infty))$.
\begin{align*}
&[px^j] \mbox{ and }  x^{jp} \mbox{ for } j\geq 1, \quad [\Delta^i x^{kp}] \mbox{ for } 1\leq k\leq n-1, 0\leq i\leq k, \\
&[\Delta^i \alpha x^k] \mbox{ for } 1\leq k\leq n-1, 0\leq i \leq p-k-2 .
\end{align*}

%\begin{notation} \label{noteoncp}
We use the following notations
\begin{myeq}\label{noteoncp}
 x_{(k,\Delta^i\alpha)} := [\Delta^i\alpha x^k], \quad x_{(j,p)}:= [p x^j] , \quad x_{(kp,\Delta^i)}:=[\Delta^i x^{kp}],
\end{myeq}
%\[ x_{(j)}:= [x^j] \in E_2^{0,-2j}(EO_2^\ast(\C P^m)) \mbox{ if } j=n \mbox{ or } n-1, 3\nmid j,\]
% \[ x_{(3k-1, \alpha)}:= [\alpha x^{3k-1}] \in E_2^{1,5-6k}(EO_2^\ast(\C P^{3k+2})) \mbox{ and } E_2^{1,5-6k}(EO_2^\ast(\C P^{3k+1})),\]
%\end{notation}
and describe the $EO_n$-cohomology in the theorem below.
\begin{thm} \label{eoncpinfty}
The $EO_n$-cohomology of $\C P^\infty$ is described (modulo $M_p$) by 
\begin{align*}
EO_n^\ast(\C P^\infty) & \cong {EO_n}_\ast\{1\} + \frac{{EO_n}_\ast}{([\Delta^j\alpha],\beta)} \{x^{jp} \mbox{ for } j\geq n, x_{(j,p)} \mbox{ } p\nmid j \} +\\
&\sum_{k=1}^{n-1}\Big( \frac{\F_p[\Delta^{\pm p},\beta]}{(\beta^{n^2-kn+1} )}\{x^{kp}\} +\frac{\F_p[\Delta^{\pm p},\beta]}{(\beta^{n^2-kn+1} )}\{x_{(kp,\Delta^i)} \mbox{ for } 1\leq i\leq k\}\Big) + \\
&\sum_{k=1}^{n-1} \frac{\F_p[\Delta^{\pm p},\beta]}{(\beta^{(k+1)n} )}\{x_{(k,\Delta^i\alpha)} \mbox{ for } 0\leq i\leq p-k-2\}.
\end{align*}
The $\beta$-multiplication is evident from the expression above. The non-trivial $[\Delta^j\alpha]$ products on these classes are given by
\begin{align*}
\resizebox{\linewidth}{!}{$[\Delta^j \alpha]x_{(kp,\Delta^k)}= [\Delta^{j+1}\alpha]x_{(kp,\Delta^{k-1})}=\dots = [\Delta^{j+k}\alpha]x^{kp} \simeq \beta^{kn} x_{(k,\Delta^j \alpha)}~ 1\leq k <n, 0 \leq j \leq p-k-2,$} \\
\resizebox{\linewidth}{!}{$\alpha x_{(k,\Delta^{p-k-2}\alpha)} \simeq \beta^{n^2 -(k+1)n+1}x^{kp}, \quad [\Delta^j \alpha] x_{(k,\Delta^{p-k-2}\alpha)} \simeq \beta^{n^2 -(k+1)n+1} x_{(kp,\Delta^j)} 1\leq k <n, 1\leq j \leq k.$}
\end{align*}
\end{thm}  

\vspace*{0.2cm}

The expression in Theorem \ref{eoncpinfty} may be used to write $EO_n^\ast(\C P^\infty)$ in terms of generators and relations as in the $p=3$ case. We do not provide these here to avoid further notational complexities. We use Theorem \ref{eoncpinfty} to point out a splitting of $EO_n\wedge \C P^\infty_+$.
\begin{thm}\label{spliteoncpinfty}
There is a splitting
\[EO_n \wedge \C P^\infty_+ \simeq EO_n \vee \Big(\bigvee_{i=1}^{p-2} \Sigma^{2i} EO_n \wedge C^{[i]} \Big) \vee \bigvee_{j\geq 0} \Sigma^{2jp+2(p-1)} E_{n}^{h(G/C_p)}, \]
in the category of $EO_n$-modules. 
\end{thm}
\begin{proof}
The splitting \eqref{splitcpodd} implies that 
\[ EO_n \wedge \C P^\infty \simeq \bigvee_{j=0}^{p-1} EO_n \land P_j.\]
The spectra $P_j$ have a cell structure with cells in dimensions $2j + 2k(p-1)$ for $k\geq 0$, and therefore, $P_j$ is a $(2p-2)$-sparse spectrum in the sense of \cite[Theorem 5.13]{Cha20}. Thus, it has algebraic $EO_n$-theory which means that $EO_n\wedge P_j$ is a wedge of suspensions of $EO_n \wedge C^{[k]}$s. From the expression in Theorem \ref{eoncpinfty}, we directly observe (by comparing the smallest power of $\beta$ which annihilates $x^{kp}$), 
\[ EO_n \land P_j \simeq \Sigma^{2j} EO_n \land C^{[j]} \vee \bigvee_{i\geq 0} \Sigma^{2(j+pi)p +2(p-1)}E_n^{h(G/C_p)}.\]  
This implies the theorem. 
\end{proof}

\vspace{0.2cm}

Now we proceed to the $m<\infty$ case. We follow the same method as in the $m=\infty$ case, using the computations of the filtration spectral sequence \eqref{fsscpmlarge}, %\eqref{fsscpmmedium}, 
\eqref{fsscpmsmall} and the formulas of the differentials in the homotoply fixed point spectral sequence in Proposition \ref{hfpssc[k]}. Write $[k]_p=n-k+[\frac{k}{p}]p$ which equals $n-s$, where $s$ is the remainder when $k$ is divided by $p$. In terms of the notation \eqref{noteoncp} we get the following computation.   
\begin{thm} \label{eoncpm}
The $EO_n$-cohomology of $\C P^m$ for $m \geq pn$ is described (modulo $M_p$) by 
\begin{align*}
EO_n^\ast(\C P^m) & \cong \frac{{EO_n}^\ast(\C P^\infty)}{(x^{m+1})} + {EO_n}_\ast\{x^{[\frac{m+1}{p}]p-1} \mbox{ if } m \neq -2 \bmod{p}\} +\\
&\sum_{\substack{m-n< k \leq m,\\ k \neq 0,-1 \bmod{p}}}\Big( \frac{\F_p[\Delta^{\pm p},\beta]}{(\beta^{n^2-[k]_pn+1} )}\{x^{k}\} +\frac{\F_p[\Delta^{\pm p},\beta]}{(\beta^{n^2-[k]_pn+1} )}\{x_{(k,\Delta^i)} \mbox{ for } 1\leq i\leq [k]_p\}\Big) + \\
& \sum_{\substack{m-n< k \leq m,\\ k \neq 0,-1 \bmod{p}}} \frac{\F_p[\Delta^{\pm p},\beta]}{(\beta^{([k]_p+1)n} )}\{x_{(k-[k]_p n,\Delta^i\alpha)} \mbox{ for } 0\leq i\leq p-[k]_p-2\}.
\end{align*}
For $2n\leq m < pn$, we write $\{k\}_p = \min({\lceil\frac{k}{n}\rceil-1,[k]_p})$ and have
\begin{align*}
EO_n^\ast(\C P^m) & \cong \resizebox{12.5cm}{!}{$\frac{{EO_n}^\ast(\C P^\infty)}{(x^{m+1}, \beta^{([\frac{m-k}{n}]+1)n}x_{(k,\Delta^i\alpha)}, x_{(kp,\Delta^i)} \mbox{\small if } kp>m)} + {EO_n}_\ast\{x^{[\frac{m+1}{p}]p-1} \mbox{\small if } m \neq -2 \bmod{p}\} +$}\\
&\sum_{\substack{m-n< k \leq m,\\ k \neq 0,-1 \bmod{p}}}\Big( \frac{\F_p[\Delta^{\pm p},\beta]}{(\beta^{n^2-\{k\}_p n+1} )}\{x^{k}\} +\frac{\F_p[\Delta^{\pm p},\beta]}{(\beta^{n^2-\{k\}_pn+1} )}\{x_{(k,\Delta^i)} \mbox{ for } 1\leq i\leq \{k\}_p\}\Big) + \\
& \sum_{\substack{m-n< k \leq m,\\ k \neq 0,-1 \bmod{p},\\ k-[k]_p n >n}} \frac{\F_p[\Delta^{\pm p},\beta]}{(\beta^{([k]_p+1)n} )}\{x_{(k-[k]_p n,\Delta^i\alpha)} \mbox{ for } 0\leq i\leq p-[k]_p-2\}.
\end{align*}
For $m<2n$, we have
\begin{align*}
EO_n^\ast(\C P^m) & \cong {EO_n}_\ast\{1, x^{s } \mbox{ for } m-n<s\leq n\} +\sum_{k=1}^{m-n} \frac{\F_p[\Delta^{\pm p},\beta]}{(\beta^{2n} )}\{x_{(k,\Delta^i\alpha)} \mbox{ for } 0\leq i\leq p-3\}\\
&+\sum_{k=n+1}^{m} \frac{\F_p[\Delta^{\pm p},\beta]}{(\beta^{n^2-n+1} )}\{x^{k}, x_{(k,\Delta)}\}.
\end{align*}
The products with $\beta$ and $[\Delta^j\alpha]$ are analogous to Theorem \ref{eoncpinfty}.
\end{thm}  

\vspace{0.2cm}

As in the $m=\infty$ case, we have an induced splitting of $EO_n \wedge \C P^m_+$ and $F(\C P^m, EO_n)\simeq EO_n \wedge D(\C P^m)$. We note this in the following theorem. 
\begin{thm}\label{spliteoncpm}
 As a $EO_n$-module, we have the following splittings for $m=pt+r$ and $ 0\leq r \leq p-1$. If $t\geq p-2$, we have
\[
F(\C P^m, EO_n)\simeq  \bigvee_{i=1}^{p-2} \Sigma^{-2pi} EO_n \wedge C^{[i]} \vee \bigvee_{i=p-1}^t \Sigma^{-2pi} E_n^{h(G/C_p)} \vee  \bigvee_{\substack{j=0\\ j\neq r}}^{p-2} \Sigma^{-2(m-j)}EO_n \wedge C^{[[m-j]_p]}.  
\]
If $2\leq t<p-2$, we have 
\[F(\C P^m, EO_n)\simeq \bigvee_{i=1}^{t} \Sigma^{-2pi} EO_n \wedge C^{[i]} \vee \bigvee_{\substack{j=0\\ j\neq r}}^{p-3} \Sigma^{-2(m-j)}EO_n \wedge C^{[\{m-j\}_p]}.  \]
Finally if $t= 0 $ or $1$, 
 \[F(\C P^m, EO_n)\simeq \bigvee_{\substack{s=m-n+1,\\ 1\leq s\leq m}}^{n} \Sigma^{-2s} EO_n \vee \bigvee_{s=n+1}^{m} \Sigma^{-2s}EO_n \wedge C^{[1]}.  \]
\end{thm}
\begin{proof}
The splitting of $\C P^\infty$ in \eqref{splitcpodd} is cellular in the sense that $\C P^m$ splits as the sum of the $2m$-skeleta of the $P_j$s. Again by the arguments in \cite[Theorem 5.13]{Cha20} implies that $\C P^m$ has algebraic $EO_n$-theory. Now comparing the power of $\beta$ which annihilates $x^{kp}$ and $x^s$ for $m-n < s < m$ implies the splitting of $EO_n\wedge \C P^m$ analogous to Theorem \ref{spliteoncpinfty}. As these are compact $EO_n$-modules, we obtain an analogous splitting for $EO_n \wedge D\C P^m$.  
\end{proof}

\vspace{0.2cm}

Let $q_m : \C P^m \to S^{2m}$ be the degree $1$ map obtained by quotienting out a copy of $\C P^{m-1}$. From Theorem \ref{eoncpm}, we deduce the following corollary.
\begin{cor}\label{powbetanonzero}
Let $m=pt+r$ with $0\leq r \leq p-1$. Write $EO_n^\ast(S^{2m})={EO_n}_\ast\{1,\epsilon_{2m}\}$ where $\epsilon_{2m} \in \widetilde{ EO_n}^{2m}(S^{2m})$ is a generator. Then, 
\[q_m^\ast(\beta^s \epsilon_{2m})= \beta^s x^m \neq 0 \mbox{  if } \begin{cases} s\leq rn & \mbox{if } t \geq n-r \\
      s\leq (n-t)n & \mbox{if } t< n-r. \end{cases} \]
\end{cor}
\begin{proof}
This follows from Theorem \ref{eoncpm} via the fact that $[m]_p = n-r$, $\lceil \frac{m}{n}\rceil = t + \lceil \frac{t+r}{n} \rceil$, and the latter equals $t+1$ if $t+r <n$. 
\end{proof}

\vspace{0.2cm}

The classes $\beta^s=\iota(\beta_1^s)$ belong to the Hurewicz image $\pi_\ast \Sp \xrightarrow{\iota} \pi_\ast EO_n$. The following corollary detects when $\beta_1^s\in \pi_{2s(pn-1)}\Sp$ maps non-trivially under the map 
\[ \C P^{s(pn-1)} \xrightarrow{q_{s(pn-1)}} S^{|\beta^s|} \xrightarrow{\beta_1^s} \Sp \xrightarrow{\iota} EO_n.\]
\begin{cor}\label{detectbeta}
The image $q_{\frac{|\beta^s|}{2}}^\ast(\iota(\beta_1^s))$ is $\neq 0$ if and only if 
\[s=r+kp, \mbox{ where } \begin{cases} k\in \{1,\dots, r-1\} & \mbox{ if } r\leq \frac{p-1}{2}\\
 k\in \{1,\dots r-2\} & \mbox{ if } r> \frac{p-1}{2}. \end{cases} \]  
\end{cor}

\begin{proof}
We will apply Corollary \ref{powbetanonzero} in the case $s=m$. Observe that $|\beta^s|=2spn-2s > 2pn$ if $s\geq 2$.  Let $s=r+kp$, so that the condition under which $q_{s(pn-1)}^\ast(\beta^s) \neq 0$ is $s\leq rn$. This implies the given restrictions on $k$. 
\end{proof}

\end{mysubsection}

\vspace{0.2cm}

\begin{mysubsection}{The $EO_n$-cohomology of quaternionic projective spaces}
Analogously as in the $p=3$ case, the $EO_n$-cohomology of $\Hyp P^m$ may be deduced from $EO_n^\ast(\C P^{2m})$ by looking at the even powers of $x$ therein.  The analogue of \eqref{splitcpodd} is 
\[ \Sigma^\infty_+ {\Hyp P^\infty}_p^\land \simeq P_0 \vee P_2 \vee \dots \vee P_{p-3}.\]
We write $y \in E_n^4(\Hyp P^\infty)$ which maps to $x^2$ under the map $\C P^\infty \to \Hyp P^\infty$. Using the notations
\[y_{(k,\Delta^i\alpha)} := [\Delta^i\alpha y^k], \quad y_{(j,p)}:= [p y^j] , \quad y_{(kp,\Delta^i)}:=[\Delta^i y^{kp}],\]
we describe the $EO_n$-cohomology of $\Hyp P^m$, which follows from the computations in Theorem \ref{eoncpinfty} and Theorem \ref{eoncpm}.
\begin{thm} \label{eonhpm}
The $EO_n$-cohomology of $\Hyp P^\infty$ is described (modulo $M_p$) by 
\begin{align*}
EO_n^\ast(\Hyp P^\infty) & \cong {EO_n}_\ast\{1\} + \frac{{EO_n}_\ast}{([\Delta^j\alpha],\beta)} \{y^{jp} \mbox{ for } j\geq \frac{n}{2}, y_{(j,p)} \mbox{ } p\nmid j \} +\\
&\sum_{k=1}^{\frac{n}{2}-1}\Big( \frac{\F_p[\Delta^{\pm p},\beta]}{(\beta^{n^2-2kn+1} )}\{y^{kp}\} +\frac{\F_p[\Delta^{\pm p},\beta]}{(\beta^{n^2-2kn+1} )}\{y_{(kp,\Delta^i)} \mbox{ for } 1\leq i\leq 2k\}\Big) + \\
&\sum_{k=1}^{\frac{n}{2}-1} \frac{\F_p[\Delta^{\pm p},\beta]}{(\beta^{(2k+1)n} )}\{y_{(k,\Delta^i\alpha)} \mbox{ for } 0\leq i\leq p-2k-2\}.
\end{align*}
The $EO_n$-cohomology of $\Hyp P^m$ for $2m \geq pn$ is described (modulo $M_p$) by 
\begin{align*}
EO_n^\ast(\Hyp P^m) & \cong \frac{{EO_n}^\ast(\Hyp P^\infty)}{(y^{m+1})} + {EO_n}_\ast\{y^{\frac{[\frac{2m+1}{p}]p-1}{2}} \mbox{ if } m \neq -1 \bmod{p}\} +\\
&\sum_{\substack{m-\frac{n}{2}< k \leq m,\\ 2k \neq 0,-1 \bmod{p}}}\Big( \frac{\F_p[\Delta^{\pm p},\beta]}{(\beta^{n^2-[2k]_pn+1} )}\{y^{k}\} +\frac{\F_p[\Delta^{\pm p},\beta]}{(\beta^{n^2-[2k]_pn+1} )}\{y_{(k,\Delta^i)} \mbox{ for } 1\leq i\leq [2k]_p\}\Big) + \\
& \sum_{\substack{m-\frac{n}{2}< k \leq m,\\ 2k \neq 0,-1 \bmod{p}}} \frac{\F_p[\Delta^{\pm p},\beta]}{(\beta^{([2k]_p+1)n} )}\{y_{(k-[2k]_p \frac{n}{2},\Delta^i\alpha)} \mbox{ for } 0\leq i\leq p-[2k]_p-2\}.
\end{align*}
For $2n\leq 2m < pn$, we have
\begin{align*}
EO_n^\ast(\Hyp P^m) & \cong \resizebox{12.5cm}{!}{$\frac{{EO_n}^\ast(\Hyp P^\infty)}{(y^{m+1}, \beta^{([\frac{2m-2k}{n}]+1)n}y_{(k,\Delta^i\alpha)}, y_{(kp,\Delta^i)} \mbox{\small if } kp>2m)} + {EO_n}_\ast\{y^{\frac{[\frac{2m+1}{p}]p-1}{2}} \mbox{\small if } m \neq -1 \bmod{p}\} +$}\\
&\sum_{\substack{m-\frac{n}{2}< k \leq m,\\ 2k \neq 0,-1 \bmod{p}}}\Big( \frac{\F_p[\Delta^{\pm p},\beta]}{(\beta^{n^2-\{2k\}_p n+1} )}\{y^{k}\} +\frac{\F_p[\Delta^{\pm p},\beta]}{(\beta^{n^2-\{2k\}_pn+1} )}\{y_{(k,\Delta^i)} \mbox{ for } 1\leq i\leq \{2k\}_p\}\Big) + \\
& \sum_{\substack{m-\frac{n}{2}< k \leq m,\\ 2k \neq 0,-1 \bmod{p},\\ 2k-[2k]_p n >n}} \frac{\F_p[\Delta^{\pm p},\beta]}{(\beta^{([2k]_p+1)n} )}\{y_{(k-[2k]_p \frac{n}{2},\Delta^i\alpha)} \mbox{ for } 0\leq i\leq p-[2k]_p-2\}.
\end{align*}
For $m<n$, we have
\begin{align*}
EO_n^\ast(\Hyp P^m) & \cong {EO_n}_\ast\{1, y^{s } \mbox{ for } m-\frac{n}{2}<s\leq \frac{n}{2}\} +\sum_{k=1}^{m-\frac{n}{2}} \frac{\F_p[\Delta^{\pm p},\beta]}{(\beta^{2n} )}\{y_{(k,\Delta^i\alpha)} \mbox{ for } 0\leq i\leq p-3\}\\
&+\sum_{k=\frac{n}{2}+1}^{m} \frac{\F_p[\Delta^{\pm p},\beta]}{(\beta^{n^2-n+1} )}\{y^{k}, y_{(k,\Delta)}\}.
\end{align*}

\end{thm}  

\vspace{0.2cm}

We also have the following splittings of $EO_n \wedge \Hyp P^\infty_+$ and $F(\Hyp P^m,EO_n)$ analogous to Theorem \ref{spliteoncpinfty} and Theorem \ref{spliteoncpm} using the expressions in Theorem \ref{eonhpm}.

\begin{thm}\label{spliteonhpinfty}
There is a splitting
\[EO_n \wedge \Hyp P^\infty_+ \simeq EO_n \vee \Big(\bigvee_{i=1}^{\frac{p-3}{2}} \Sigma^{4i} EO_n \wedge C^{[2i]} \Big) \vee \bigvee_{j\geq 0} \Sigma^{4jp+2(p-1)} E_{n}^{h(G/C_p)}, \]
in the category of $EO_n$-modules. For $m=pt+r$, with $t\geq \frac{n}{2}-1$, we have
\[
F(\Hyp P^m, EO_n)\simeq  \bigvee_{i=1}^{\frac{p-3}{2}} \Sigma^{-4pi} EO_n \wedge C^{[2i]} \vee \bigvee_{i=\frac{n}{2}}^t \Sigma^{-4pi} E_n^{h(G/C_p)} \vee  \bigvee_{\substack{j=0\\ j\neq r}}^{\frac{n}{2}-1} \Sigma^{-4(m-j)}EO_n \wedge C^{[[2(m-j)]_p]}.  
\]
If $1\leq t<\frac{n}{2}-1$, we have 
\[F(\Hyp P^m, EO_n)\simeq \bigvee_{i=1}^{t} \Sigma^{-4pi} EO_n \wedge C^{[2i]} \vee \bigvee_{\substack{j=0\\ j\neq r}}^{\frac{n}{2}-1} \Sigma^{-4(m-j)}EO_n \wedge C^{[\{2(m-j)\}_p]}.  \]
Finally if $m=r < p $, 
 \[F(\Hyp P^m, EO_n)\simeq \bigvee_{\substack{s=m-\frac{n}{2}+1,\\ 1\leq s\leq m}}^{\frac{n}{2}} \Sigma^{-4s} EO_n \vee \bigvee_{s=\frac{n}{2}+1}^{m} \Sigma^{-4s}EO_n \wedge C^{[1]}.  \]
\end{thm}

\vspace{0.2cm}

The following corollary to Theorem \ref{eonhpm} is analogous to Corollary \ref{detectbeta} and detects when $\beta_1^s\in \pi_{2s(pn-1)}\Sp$ for even $s$ maps non-trivially under the map 
\[ \Hyp P^{\frac{s(pn-1)}{2}} \xrightarrow{q_{\frac{s(pn-1)}{2}}} S^{2s(pn-1)} \xrightarrow{\beta_1^s} \Sp \xrightarrow{\iota} EO_n.\]
\begin{cor}\label{detectbetahpm}
For $s$ even, the image $q_{\frac{s(pn-1)}{2}}^\ast(\iota(\beta_1^s))$ is $\neq 0$ if and only if 
\[s=r+kp, \mbox{ where }  k\in \{1,\dots r-2\}, \mbox{ and } k+r \mbox{ is even}.  \]  
\end{cor}
\begin{proof}
The result follows by applying the condition $s$ is even to the criterion in Corollary \ref{detectbeta}.
\end{proof}
\end{mysubsection}

\vspace*{0.2cm}

	\section{Consequences for exotic smooth structures} \label{applnexotic}
We now point out some consequences for exotic smooth structures derived from the computations of $EO_n^\ast(\C P^m)$. We examine the inertia group of $\C P^m$, and then observe that these provide  exotic smooth structures on complex hyperbolic manifolds. We also deduce some free smooth $S^1$ and $S^3$ actions on exotic spheres. 

\begin{mysubsection}{The inertia group of $\C P^m$} 
Recall that the inertia group $I(M)$ of a smooth manifold $M$ of dimension $m$ is the subgroup of the group of exotic spheres $\Theta_m$, which on connected sum do not change the smooth structure on $M$. The inertia group possesses various subgroups of importance formed according to the structure of the identification $M\# \Sigma \cong M$ for $[\Sigma] \in I(M)$. The concordance inertia group $I_c(M)$ is the subgroup of those $\Sigma$ such that the diffeomorphism $M\# \Sigma \cong M$ is concordant to the identity. Via the identification $\Theta_m \cong [S^m, \Top/O]$ \cite{KM63}, we have \cite{KS77}
\[ I_c(M) = \Ker([S^m, \Top/O] \xrightarrow{q^\ast} [M,\Top/O]),\]
where $M$ is  a closed $m$-manifold, and $q : M \to S^m$ is the map which quotients out the lower skeleton. For $\C P^m$, we have $I(\C P^m) = I_c(\C P^m)$\cite[Theorem 4.2]{Kas15}. The inertia group of $\C P^m$ is largely unknown, and explicit computations are available only in low dimensional cases (\cite{Kaw68}, \cite{BK18}). 

Our approach towards the inertia group involves using the map $\Top/O\to G/O$, and the fact \cite[Lemma 2.6]{BK18} 
\[ [\C P^m , \Top/O] \to [\C P^m, G/O]\]
is injective.  Combining this with the fact that the map $\Theta_{2m} \cong [S^{2m},\Top/O] \to [S^{2m}, G/O]$ is also injective, we have the following proposition. 
\begin{prop}\label{cokerjinertcpn}
Let $\xi \in \pi_{2m}^s/\Im(J)$ correspond to an exotic sphere $\Sigma$ of dimension $2m$ where $2m \not \in \{2,6,14,30, 62, 126\}$. Let $\tilde{\xi}$ be a lift of $\xi$ in $\pi_{2m}^s$. Suppose that the image of $\tilde{\xi}$ under the map 
\[ q^\ast  : \pi_{2m}^s \to [\C P^m, \Sp]\]
is non-zero. Then, $\Sigma \not\in I(\C P^m)$.  
\end{prop}

\vspace*{0.2cm}

In order to leverage Proposition \ref{cokerjinertcpn}, we use the unit map $\iota: \Sp \to EO_n$ and detect elements in the image 
\[ h_\iota: [\C P^m, \Sp] \to [\C P^m, EO_n].\]
For the prime $3$, we have the following result recalling the expression of Theorem \ref{eo2cpn}. 
\begin{thm}\label{hurcp3}
The image of the map 
\[ h_\iota: [\C P^m, \Sp]_\ast \to [\C P^m, EO_2]_\ast\cong EO_2^{-\ast}(\C P^m)\]
contains the classes (for $t\geq 0$)
\[\resizebox{\linewidth}{!}{$\begin{cases} 
\F_3\{x_{(3k-1)},\alpha x_{(3k-1)},\Delta^{3t}\alpha\beta x_{(3k-1)},\Delta^{3t}\beta b x_{(3k-1)}, \Delta^{3t} \beta^i x_{(3k-1)}, 1\leq i \leq 4 \} & \mbox{ if } m=3k,\\
\\
\F_3\{x_{(3k+1)}, \Delta^{3t}\beta x_{(3k+1)}, \Delta^{3t} \beta^2 x_{(3k+1)},x_{(3k-1,\alpha)}, \Delta^{3t}\beta^i x_{(3k-1,\alpha)}, 1\leq i \leq 3\} & \mbox{ if } m= 3k+1,\\
\\
\F_3\{x_{(3k+2)},\alpha x_{(3k+2)}, \Delta^{3t}\alpha\beta x_{(3k+2)},\Delta^{3t}\beta b x_{(3k+2)},\Delta^{3t} \beta^i x_{(3k+2)}, 1\leq i \leq 4 \} + \\
\F_3\{x_{(3k+1)}, \Delta^{3t}\beta x_{(3k+1)}, \Delta^{3t} \beta^2 x_{(3k+1)},x_{(3k-1,\alpha)}, \Delta^{3t}\beta^i x_{(3k-1,\alpha)}, 1\leq i \leq 3\} 
&\mbox{ if } m=3k+2.\end{cases} $}\]
\end{thm}

\begin{proof}
We first consider the case $m=3k$. Note that after $3$-completion, $\C P^{3k}/\C P^{3k-2} \simeq S^{6k} \vee S^{6k-2}$. Now we consider 
\[\xymatrix{[S^{6k-2}, \Sp]_\ast \ar[d]^{h_\iota}  \ar@/^1.5pc/[r]^{r}& [\C P^{3k}/\C P^{3k-2}, \Sp]_\ast \otimes \Z^\wedge_3 \ar[l] \ar[r] \ar[d]^{h_\iota} & [\C P^{3k}, \Sp]_\ast \otimes \Z_3^\wedge \ar[d]^{h_\iota}   \\
[S^{6k-2}, EO_2]_\ast \ar@/^1,5pc/[r]^{r}   & [\C P^{3k}/\C P^{3k-2}, EO_2]_\ast \ar[l] \ar[r] & [\C P^{3k}, EO_2]_\ast, }\]
where $r$ is the retraction of $\C P^{3k}/\C P^{3k-2}$ onto $S^{6k-2}$. We readily observe that $x_{(3k-1)} \xi$ where $\xi$ belongs to the Hurewicz image $\pi_\ast(\Sp) \to \pi_\ast(EO_2)$, lies in the image of the composite 
\[ [S^{6k-2},\Sp]_\ast \xrightarrow{h_\iota} [S^{6k-2}, EO_2]_\ast \xrightarrow{r} [\C P^{3k}/\C P^{3k-2},EO_2]_\ast \to [\C P^{3k}, EO_2]_\ast.\]
This directly implies the statement for $m=3k$. 

For $m=3k+1$, we have $\Big(\C P^{3k+1}/\C P^{3k-2}\Big)^{\wedge}_3 \simeq \Sigma^{6k-2} C(\alpha_1)\vee \Big(S^{6k}\Big)^\wedge_3$.  We compute the map $h_\iota$ for $C(\alpha_1)$ via the following diagram 
\begin{myeq}\label{eqcalph3}
\xymatrix{ [C(\alpha_1), \Sp]_\ast\otimes \Z^{\wedge}_3 \ar[d]^{h_\iota} \ar[r]^{\cong} & \pi_\ast(\Sigma^{-4} C(\alpha_1))\otimes \Z^{\wedge}_3  \ar[d]^{h_\iota}\\
[C(\alpha_1), EO_2]_\ast\ar[r]^{\cong} & {EO_2}_\ast(\Sigma^{-4} C(\alpha_1)) }
\end{myeq}
using the fact that the Spanier Whitehead dual $D(C(\alpha_1))\simeq \Sigma^{-4} C(\alpha_1)$. Using Theorem \ref{hureo2} and \eqref{einftyhfpsseo2calph}, the following classes belong to the image of $h_\iota$ (the right vertical arrow of \eqref{eqcalph3})
\begin{myeq}\label{hurcalph3}
\begin{aligned}
[\Sigma^{-4}\iota_0], \Delta^{3t}\beta[\Sigma^{-4}\iota_0], \Delta^{3t}\beta^2[\Sigma^{-4}\iota_0], [\Sigma^{-4}\kappa_7], \\
 \Delta^{3t}\beta[\Sigma^{-4}\kappa_7], \Delta^{3t}\beta^2[\Sigma^{-4}\kappa_7], \Delta^{3t}\beta^3[\Sigma^{-4}\kappa_7], & \mbox{ for } t\geq 0.
\end{aligned} 
\end{myeq}
We now consider the diagram
\[\xymatrix{[\Sigma^{6k-2}C(\alpha_1), \Sp]_\ast \ar[d]^{h_\iota}  \ar@/^1.5pc/[r]^{r}& [\C P^{3k+1}/\C P^{3k-2}, \Sp]_\ast \otimes \Z^\wedge_3 \ar[l] \ar[r] \ar[d]^{h_\iota} & [\C P^{3k+1}, \Sp]_\ast \otimes \Z_3^\wedge \ar[d]^{h_\iota}   \\
[\Sigma^{6k-2} C(\alpha_1), EO_2]_\ast \ar@/^1,5pc/[r]^{r}   & [\C P^{3k+1}/\C P^{3k-2}, EO_2]_\ast \ar[l] \ar[r] & [\C P^{3k+1}, EO_2]_\ast, }\]
where $r$ is the retraction of $\Big(\C P^{3k+1}/\C P^{3k-2}\Big)^{\wedge}_3$ onto $\Sigma^{6k-2}C(\alpha_1)$. The image of the composite from the top left corner along the bottom row 
 \[ [\Sigma^{6k-2} C(\alpha_1),\Sp]_\ast \xrightarrow{h_\iota} [\Sigma^{6k-2}C(\alpha_1), EO_2]_\ast \xrightarrow{r} [\C P^{3k+1}/\C P^{3k-2},EO_2]_\ast \to [\C P^{3k+1}, EO_2]_\ast,\]
may be computed using the formula of $h_\iota$ from \eqref{hurcalph3}. We compute 
\[[\Sigma^{-6k-2}\iota_0]\mapsto x_{(3k+1)}, \quad \mbox{ and }  [\Sigma^{-6k-2}\kappa_7]\mapsto x_{(3k-1,\alpha)},\]
which allows us to conclude the result for $m=3k+1$. 

The remaining case for $m=3k+2$ follows from the decomposition 
\[\Big(\C P^{3k+2}/\C P^{3k-2}\Big)^{\wedge}_3 \simeq \Sigma^{6k-2} C(\alpha_1) \vee \Big(S^{6k} \vee S^{6k+4}\Big)^\wedge_3,\]
and the computations in the previous cases.
\end{proof}

\vspace*{0.2cm}

We now obtain consequences for the inertia group applying Theorem \ref{hurcp3} for $p=3$ and Corollary \ref{detectbeta} for $p\geq 5$. 
\begin{thm}\label{inertcpn}
a) Suppose $m=36t+5$, $m=36t +10$, or $m=36t+20$. There are $2$ exotic spheres $\Sigma_1$ and $\Sigma_2$ of dimension $2m$ such that $\Sigma_i \not\in I(\C P^m)$, and  $\{\C P^m, \C P^m\# \Sigma_1, \C P^m \# \Sigma_2\}$ are pairwise non-diffeomorphic, and in fact not concordant, to each other.\\
b) Suppose 
\[m=(r+kp)(p(p-1)-1), \mbox{ where } \begin{cases} k\in \{1,\dots, r-1\} & \mbox{ if } r\leq \frac{p-1}{2}\\
 k\in \{1,\dots r-2\} & \mbox{ if } r> \frac{p-1}{2}, \end{cases} \]  
for $p\geq 5$. There are $(p-1)$ exotic spheres $\Sigma_1, \dots, \Sigma_{p-1}$ of dimension $2m$ such that $\Sigma_i \not\in I(\C P^m)$ and  $\{\C P^m, \C P^m\# \Sigma_1, \dots,\C P^m \# \Sigma_{p-1}\}$ are pairwise non-diffeomorphic.
\end{thm}

\begin{proof}
We apply Proposition \ref{cokerjinertcpn} and check that there are suitable $\xi \in \pi_{2m}^s/\Im(J)$ such that the image of $\tilde{\xi}$ in the composite
\[\pi_{2m}^s \xrightarrow{q^\ast} [\C P^m, \Sp] \xrightarrow{h_\iota} [\C P^m, EO_n]\]
is non-zero. The proof of b) follows directly from Corollary \ref{detectbeta} as in the given degrees $\F_p\{\beta_1^m\}\subset \pi_{2m}(\Sp)$ maps non-trivially to $EO_n^\ast(\C P^m)$ under the map $h_\iota\circ q^\ast$. We now apply Theorem \ref{hurcp3} to prove a). As the classes detected are all $3$-local, in the given degrees, the $\Im(J)$ is $0$. Recall from Theorem \ref{hureo2} that the classes $\Delta^{3t}\beta^i$ for $t\geq 0$ and $1\leq i \leq 4$, lie in the image of the Hurewicz homomorphism $h_\iota : \pi_{2m}(\Sp_3^\wedge) \to \pi_{2m}EO_2$. We denote the corresponding elements in $\pi_{2m}^s$ by $\widehat{\Delta^{3t} \beta^i}$. We now consider the class $\widehat{\Delta^{3t}\beta}$ which maps to the non-zero class $\Delta^{3t}\beta x_{(36t+5)}$,  the class $\widehat{\Delta^{3t}\beta^2}$  which maps to the non-zero class $\Delta^{3t}\beta^2 x_{(36t+10)}$, and the class $\widehat{\Delta^{3t}\beta^4}$  which maps to the non-zero class $\Delta^{3t}\beta^4 x_{(36t+20)}$. Therefore, in these degrees there is a copy of $\F_3 \subset \pi_{2m}^s$ which maps non-trivially into $[\C P^m, EO_2]$. The conclusion follows.  
\end{proof}

\end{mysubsection}

\vspace*{0.2cm}

\begin{mysubsection}{Exotic complex hyperbolic manifolds} Farrell and Jones \cite{FJ94} proved that there are compact manifolds $\MM^{8n+2}$ and $\NN^{8n+2}$ in dimensions $8n+2$ satisfying \\
(1) $\MM$ and $\NN$ are homeomorphic but not diffeomorphic. \\
(2) $\MM$ is a complex hyperbolic manifold with complex dimension $4n+1$. \\
(3) $\NN$ supports a negatively curved Riemannian metric. In fact, for any $\epsilon >0$, there is a Riemannian metric on $\NN$ whose sectional curvatures lie in the interval $[-4-\epsilon, -1+\epsilon]$. 

The manifold $\NN$ is constructed from $\MM$ by taking connected sum with an exotic sphere. The condition on the sectional curvature may be achieved for any $\MM\# \Sigma$ where $\Sigma\in \Theta_{2m}$, where $\MM$ is a complex hyperbolic manifold of complex dimension $m$ (Ref : Farrell Jones Proposition 3.19). Farell and Jones prove the following result which allows them to construct examples such that $\MM$ and $\NN$ are not diffeomorphic if $m=4k+1$. 
\begin{prop}(Farrell-Jones \cite[Corollary 3.14]{FJ94} (see also Okun \cite{Oku02})) \label{fjthm}
Let $M$ be a compact complex $m$-dimensional manifold. There exists a finite sheeted cover $\MM$ of $M$ such that for any further finite sheeted cover $N$ of $\MM$ and $\Sigma_1,\Sigma_2 \in \Theta_{2m}$, if $N\# \Sigma_1$ is diffeomorphic to $N\# \Sigma_2$, then $\C P^m \# \Sigma_1$ is concordant to either $\C P^m \# \Sigma_2$ or $\C P^m \# (- \Sigma_2)$.  
\end{prop}

\vspace*{0.2cm}

We now apply Proposition \ref{fjthm} and Theorem \ref{inertcpn} to produce exotic smooth structures on complex hyperbolic manifolds. 
\begin{thm}\label{exstrcxhyp}
a) Let $m=36t+5$. Then, there is a complex hyperbolic manifold $\MM$ of complex dimension $m$ and $\Sigma_1, \Sigma_2 \in \Theta_{2m}$ such that \\
(i) The manifolds $\{\MM, \MM\# \Sigma_1, \MM \# \Sigma_2, \MM \# \Sigma_1 \# \Sigma_2\}$ are pairwise non-diffeomorphic. \\
(ii) Each of the manifolds $\MM\# \Sigma_1$, $\MM \# \Sigma_2$, $\MM \# \Sigma_1 \# \Sigma_2$ support negatively curved Riemannian metrics whose sectional curvatures may be chosen to lie in $[-4-\epsilon, -1+\epsilon]$ (for any given $\epsilon >0$). \\
b) Let $m=36t +10$, or $m=36t+20$. Then, there is a complex hyperbolic manifold $\MM$ of complex dimension $m$ and $\Sigma \in \Theta_{2m}$ such that $\MM$ and $\MM \# \Sigma$ are not diffeomorphic, and $\MM\# \Sigma$ supports a negatively curved Riemannian metric whose sectional curvatures may be chosen to lie in $[-4-\epsilon, -1+\epsilon]$ (for any given $\epsilon >0$).\\
c) Let 
\[m=(r+kp)(p(p-1)-1), \mbox{ where } \begin{cases} k\in \{1,\dots, r-1\} & \mbox{ if } r\leq \frac{p-1}{2}\\
 k\in \{1,\dots r-2\} & \mbox{ if }  \frac{p-1}{2}< r \leq p-1, \end{cases} \]  
for $p\geq 5$.  Then, there is a complex hyperbolic manifold $\MM$ of complex dimension $m$ and $\Sigma_i \in \Theta_{2m}$ for $1\leq i \leq \frac{p-1}{2}$ such that $\{\MM,\MM \# \Sigma_1,\dots, \MM\# \Sigma_{\frac{p-1}{2}}\}$ are pairwise non-diffeomorphic, and $\MM\# \Sigma_i$ supports a negatively curved Riemannian metric whose sectional curvatures may be chosen to lie in $[-4-\epsilon, -1+\epsilon]$ (for any given $\epsilon >0$).

\end{thm}

\begin{proof}
The statement in b) directly follows from Theorem \ref{inertcpn} and Proposition \ref{fjthm}. For the statement in a), we have a $\Sigma_1 \in \Theta_{2m}$ of order $3$ using Theorem \ref{inertcpn}. The sphere $\Sigma_2\in \Theta_{2m}$ is the one obtained in \cite[Theorem 3.20]{FJ94} which is of order $2$. It follows that $\Sigma_1$ and $\Sigma_2$ generate a subgroup of order $6$ in $\Theta_{2m}$ whose intersection with the inertia group of $\C P^m$ is $0$. The statement in a) now follows from Proposition \ref{fjthm}. 
\end{proof}

\end{mysubsection}

\vspace*{0.2cm}

\begin{mysubsection}{The inertia group of $\Hyp P^m$} 
We demonstrate that, as in the case of $\C P^m$, the computation of $EO_n$-cohomology of $\Hyp P^m$ may be used to construct exotic spheres lying outside the concordance inertia group. Let $q:\Hyp P^m \to S^{4m}$. The analogue of Proposition \ref{cokerjinertcpn} in this case is the following proposition. 
\begin{prop}\label{cokerjinerthpn}
Let $\xi \in \pi_{4m}^s/\Im(J)$ correspond to an exotic sphere $\Sigma$ of dimension $4m$. Let $\tilde{\xi}$ be a lift of $\xi$ in $\pi_{4m}^s$. Suppose that the image of $\tilde{\xi}$ under the map 
\[ q^\ast  : \pi_{4m}^s \to [\Hyp P^m, \Sp]\]
is non-zero. Then, $\Sigma \not\in I_c(\Hyp P^m)$.  
\end{prop}

\vspace*{0.2cm}

We apply Proposition \ref{cokerjinerthpn} by detecting elements in the image 
\[ h_\iota: [\Hyp P^m, \Sp] \to [\Hyp P^m, EO_{n}].\]
For the prime $3$, we have the following result recalling the expression of Theorem \ref{eo2hpn}. 
\begin{thm}\label{hurhp3}
The image of the map 
\[ h_\iota: [\Hyp P^m, \Sp]_\ast \to [\Hyp P^m, EO_2]_\ast\cong EO_2^{-\ast}(\Hyp P^m)\]
contains the classes (for $t\geq 0$)
\[\begin{cases} 
\F_3\{y^k\} & \mbox{if } m=3k,\\
\F_3\{x_{(6k+2)},\alpha x_{(6k+2)}, \Delta^{3t}\alpha\beta x_{(6k+2)},\Delta^{3t}\beta b x_{(6k+2)},\Delta^{3t} \beta^i x_{(6k+2)}, 1\leq i \leq 4 \} & \mbox{if } m=3k+1,\\ 
\F_3\{x_{(6k+4)}, \Delta^{3t}\beta x_{(6k+4)}, \Delta^{3t} \beta^2 x_{(6k+4)},x_{(6k+2,\alpha)}, \Delta^{3t}\beta^i x_{(6k+2,\alpha)}, 1\leq i \leq 3\} & \mbox{if } m= 3k+2.
\end{cases} \]
\end{thm}

\begin{proof}
The proof is exactly the same as Theorem \ref{hurcp3}, based on the equivalences 
\[ \Hyp P^{3k+1}/\Hyp P^{3k} \simeq S^{6k+2}, \quad \Big(\Hyp P^{3k+2}/\Hyp P^{3k}\Big)^\wedge_3 \simeq  C(\alpha_1), \]
and the results of Theorem \ref{eo2hpn}. 
\end{proof}

\vspace*{0.2cm}

The following consequences for the concordance inertia groups of $\Hyp P^m$ follow from Theorem \ref{hurhp3}. 
\begin{thm}\label{inerthpn}
a) Suppose $m=18t+5$, or $m=18t+10$. There are $2$ exotic spheres $\Sigma_1$ and $\Sigma_2$ of dimension $4m$ such that $\Sigma_i \not\in I_c(\Hyp P^m)$, and  $\{\Hyp P^m, \Hyp P^m\# \Sigma_1, \Hyp P^m \# \Sigma_2\}$ are pairwise non-concordant to each other. \\
b) Let $m=(\frac{r+kp}{2})(p(p-1)-1)$ for a prime $p\geq 5$ and $r<p$, where  $ k\in \{1,\dots r-2\}$, and $k+r$ is even.   There are $p-1$ exotic spheres $\Sigma_1, \dots, \Sigma_{p-1}$ of dimension $4m$ such that $\Sigma_i \not\in I_c(\Hyp P^m)$, and  $\{\Hyp P^m, \Hyp P^m\# \Sigma_1, \dots, \Hyp P^m \# \Sigma_{p-1}\}$ are pairwise non-concordant to each other.
\end{thm}

\begin{proof}
The proof is exactly the same as that of Theorem \ref{inertcpn}, applying Proposition \ref{cokerjinerthpn}. We deduce that the composite 
\[\pi_{4m}^s \xrightarrow{q^\ast} [\Hyp P^m, \Sp] \xrightarrow{h_\iota} [\Hyp P^m, EO_2]\]
sends 
\[ \widehat{\Delta^{3t} \beta^2} \mapsto \Delta^{3t} \beta^2 x_{(36t+10)} \neq 0 \in EO_2^\ast(\Hyp P^{18t+5}),\]
and
\[ \widehat{\Delta^{3t} \beta^4} \mapsto \Delta^{3t} \beta^4 x_{(36t+20)} \neq 0 \in EO_2^\ast(\Hyp P^{18t+10}).\]
This implies the result. 
\end{proof}

\vspace*{0.2cm}

Theorem \ref{inerthpn} may now be used to construct exotic smooth structures on quaternionic hyperbolic manifolds. Specifically, we apply \cite{AF04} and \cite[Theorem 3.6]{Oku02} to deduce the following theorem. 
\begin{thm}\label{exstrquathyp}
a) Let $m=18t +5$, or $m=18t+10$. Then, there is a quaternionic hyperbolic manifold $\MM$ of quaternionic dimension $m$ and $\Sigma \in \Theta_{4m}$ such that $\MM$ and $\MM \# \Sigma$ are not diffeomorphic, and $\MM\# \Sigma$ supports a negatively curved Riemannian metric. \\
b) Let $m=(\frac{r+kp}{2})(p(p-1)-1)$ for a prime $p\geq 5$ and $r<p$, with  $ k\in \{1,\dots r-2\}$, and $k+r$ is even. Then, there is a quaternionic hyperbolic manifold $\MM$ of quaternionic dimension $m$ and $\Sigma_1, \dots, \Sigma_{\frac{p-1}{2}} \in \Theta_{4m}$ such that $\{\MM, \MM\# \Sigma_1, \dots, \MM \# \Sigma_{\frac{p-1}{2}}\}$ are not diffeomorphic, and $\MM\# \Sigma_j$ supports a negatively curved Riemannian metric for $1\leq j \leq \frac{p-1}{2}$. 
\end{thm}

\end{mysubsection}

\vspace*{0.2cm}

\begin{mysubsection}{Free $S^1$ and $S^3$ actions on exotic spheres}
We now deduce results about the existence of free smooth  $S^1$-actions and $S^3$-actions on exotic spheres which do not bound parallelizable manifolds. Such an exotic sphere $\Sigma$ corresponds to $\xi \in \pi_m(\Top/O)$ such that  the map $\Theta_m\cong \pi_m( \Top/O)\to \pi_m(G/O)$ sends $\xi$ to an element in the image  \cite{KM63}
\begin{myeq}\label{imcokj}
 \Cok(J) \subset \pi_m^s \cong \pi_m(G) \to \pi_m(G/O).
\end{myeq} 
Here, we are assuming that $m\geq 5$ and $m\not \in \{5, 13, 29, 61, 125\}$. Let $\eta_{\C P^m} : S^{2m+1} \to \C P^m$ and $\nu_{\Hyp P^m} : S^{4m+3} \to \Hyp P^m$ be the usual quotient maps. We have the following result from \cite[Lemma 4.3]{BK18} and \cite{BQ26}.
\begin{lemma} \label{freeact}
a) Suppose that an $\alpha \in \pi_{2m+1}(Top/O)$ such that the image of $\alpha$ in $\pi_{2m+1}(G/O)$ lies in the image of \eqref{imcokj}, and that $\alpha$ lies in the image of  
\[[\C P^m, \Top/O] \xrightarrow{\eta_{\C P^m}^\ast} [S^{2m+1},\Top/O] \cong \pi_{2m+1}(\Top/O).\]
Then, there is a free smooth $S^1$-action on the exotic sphere  $\Sigma$ corresponding to $\alpha \in \Theta_{2m+1}$. \\
a) Suppose that a $\beta \in \pi_{4m+3}(Top/O)$ such that the image of $\beta$ in $\pi_{4m+3}(G/O)$ lies in the image of \eqref{imcokj}, and that $\beta$ lies in the image of  
\[[\Hyp P^m, \Top/O] \xrightarrow{\nu_{\Hyp P^m}^\ast} [S^{4m+3},\Top/O] \cong \pi_{4m+3}(\Top/O).\]
Then, there is a free smooth $S^3$-action on the exotic sphere  $\Sigma$ corresponding to $\beta \in \Theta_{4m+3}$. 
\end{lemma}
 
\vspace*{0.2cm}

We now use Lemma \ref{freeact} and the computations of $EO_n^\ast(\C P^m)$ in Theorem \ref{eo2cpn} and $EO_n^\ast(\Hyp P^m)$  in Theorem \ref{eo2hpn} to deduce the following theorem about smooth actions on exotic spheres. 
\begin{thm}\label{smactexotic}
a) Let $l\in \{72t + 13, 72t+23, 72t+43\}$. There exists an exotic sphere $\Sigma$ of dimension $l$ which does not bound a parrellizable manifold, and has a free smooth $S^1$-action. 
\\
b) Let $l=72t + 23$ or $l=72t+43$. There exists an exotic sphere $\Sigma$ of dimension $l$ which does not bound a parrellizable manifold, and has a free smooth   $S^3$-action. 
\end{thm}

\begin{proof}
For a), consider the diagram 
\[ \xymatrix{  & & & & EO_2^0(\C P^{m+1})\ar[d] \\
[\C P^m, \Top/O] \ar[r] \ar[d]^{\eta_{\C P^m}^\ast}  & [\C P^m, G/O] \ar[d]^{\eta_{\C P^m}^\ast} & [\C P^m, G] \ar[l] \ar[d]^{\eta_{\C P^m}^\ast} \ar[r]^{\cong} & [\C P^m, \Sp] \ar[d]^{\eta_{\C P^m}^\ast} \ar[r]^{h_\iota} & EO_2^0 (\C P^m) \ar[d]^{\eta_{\C P^m}^\ast}   \\
[S^{2m+1}, \Top/O] \ar[r]  & [S^{2m+1}, G/O] & [S^{2m+1}, G] \ar[l] \ar[r]^{\cong} & [S^{2m+1}, \Sp]\ar[r]^{h_\iota} & EO_2^0(S^{2m+1}). } \]
The proof of a) will be deemed complete once we show that there is a class $\chi \in h_\iota([\C P^m, \Sp])$ such that $\eta^\ast_{\C P^m}(\chi)\neq 0$, and equals $h_\iota(\xi)$ for some $\xi \in \Cok(J)$. This is due to the fact that $[\C P^m, G] \to [\C P^m, G/O]$ is monic \cite[pg 77]{Bru69} and that 
\[ \Im([\C P^m, G] \to [\C P^m, G/O]) = \Im([\C P^m, \Top/O] \to [\C P^m, G/O]).\]
Now we verify the conditions for the given values of $l$ using the formula in Theorem \ref{hurcp3} and \eqref{hurewiczeo2}. 
\begin{myeq}\label{exotics1}
\begin{aligned} 
& l= 72t+13, m= 36t+6 : \quad \chi= \Delta^{3t}\beta x_{(36t+5)}, \quad \xi = \begin{cases} \beta_{\frac{9t}{2} + 1}  & t \mbox{ even,} \\
\beta_{\big(\frac{9(t-1)}{2}+3\big)/3} & t \mbox{ odd.}  \end{cases}  \\
& l= 72t+23, m= 36t+11 : \quad \chi= \Delta^{3t}\beta^2 x_{(36t+10)}, \quad \xi = \begin{cases} \beta_1\beta_{\frac{9t}{2} + 1}  & t \mbox{ even,} \\
\beta_1\beta_{\big(\frac{9(t-1)}{2}+3\big)/3} & t \mbox{ odd.}  \end{cases}  \\
& l= 72t+43, m= 36t+21 : \quad \chi= \Delta^{3t}\beta^4 x_{(36t+20)}, \quad \xi = \begin{cases} \beta_1^3\beta_{\frac{9t}{2} + 1}  & t \mbox{ even,} \\
\beta_1^3 \beta_{\big(\frac{9(t-1)}{2}+3\big)/3} & t \mbox{ odd.}  \end{cases}  \\
\end{aligned}
\end{myeq}
The fact that these values of $\xi$ lie in $\Cok(J)$ follows from the computations in \cite{Rav86}. The proof of b) is entirely analogous, where we apply Theorem \ref{hurhp3} to show that the second and third examples of \eqref{exotics1} are also present in the $\Hyp P^m$ case.
\end{proof}

\end{mysubsection}


\vspace*{0.2cm}
	
\begin{thebibliography}{10}

\bibitem{Ada66}
{\sc J.~F. Adams}, {\em On the groups {$J(X)$}. {IV}}, Topology, 5 (1966),
  pp.~21--71.

\bibitem{Ada74}
\leavevmode\vrule height 2pt depth -1.6pt width 23pt, {\em Stable homotopy and
  generalised homology}, Chicago Lectures in Mathematics, University of Chicago
  Press, Chicago, Ill.-London, 1974.

\bibitem{AF03}
{\sc C.~S. Aravinda and F.~T. Farrell}, {\em Exotic negatively curved
  structures on {C}ayley hyperbolic manifolds}, J. Differential Geom., 63
  (2003), pp.~41--62.

\bibitem{AF04}
\leavevmode\vrule height 2pt depth -1.6pt width 23pt, {\em Exotic structures
  and quaternionic hyperbolic manifolds}, in Algebraic groups and arithmetic,
  Tata Inst. Fund. Res., Mumbai, 2004, pp.~507--524.

\bibitem{BK18}
{\sc S.~Basu and R.~Kasilingam}, {\em Inertia groups of high-dimensional
  complex projective spaces}, Algebr. Geom. Topol., 18 (2018), pp.~387--408.

\bibitem{BK22}
\leavevmode\vrule height 2pt depth -1.6pt width 23pt, {\em Inertia groups and
  smooth structures on quaternionic projective spaces}, Forum Math., 34 (2022),
  pp.~369--383.

\bibitem{Bau03}
{\sc T.~Bauer}, {\em Elliptic cohomology and projective spaces - a
  computation}, 2003.

\bibitem{BQ26}
{\sc T.~Bauer and J.~D. Quigley}, {\em Infinite families of very exotic spheres
  with free $s^1$- and $s^3$-actions}, 2026.

\bibitem{Beh20}
{\sc M.~Behrens}, {\em Topological modular and automorphic forms}, in Handbook
  of homotopy theory, CRC Press/Chapman Hall Handb. Math. Ser., CRC Press, Boca
  Raton, FL, [2020] \copyright 2020, pp.~221--261.

\bibitem{BS23}
{\sc E.~Belmont and K.~Shimomura}, {\em Beta families arising from a {$v_2^9$}
  self-map on {$S/(3,v_1^8)$}}, Algebr. Geom. Topol., 23 (2023),
  pp.~4349--4378.

\bibitem{BC22}
{\sc P.~Bhattacharya and H.~Chatham}, {\em On the {EO}-orientability of vector
  bundles}, J. Topol., 15 (2022), pp.~2017--2044.

\bibitem{Bru69}
{\sc G.~Brumfiel}, {\em {D}ifferentiable {$S^1$}-actions on homotopy spheres},
  preprint, Princeton University,  (1969).

\bibitem{CE99}
{\sc H.~Cartan and S.~Eilenberg}, {\em Homological algebra}, vol.~19, Princeton
  university press, 1999.

\bibitem{Cha20}
{\sc H.~Chatham}, {\em An {O}rientation {M}ap for {H}eight p -1 {R}eal {E}
  {T}heory}, ProQuest LLC, Ann Arbor, MI, 2020.
\newblock Thesis (Ph.D.)--Massachusetts Institute of Technology.

\bibitem{DH04}
{\sc E.~S. Devinatz and M.~J. Hopkins}, {\em Homotopy fixed point spectra for
  closed subgroups of the {M}orava stabilizer groups}, Topology, 43 (2004),
  pp.~1--47.

\bibitem{Dug03}
{\sc D.~Dugger}, {\em Multiplicative structures on homotopy spectral sequences
  {II}}, preprint,  (2003).

\bibitem{FJ89c}
{\sc F.~T. Farrell and L.~E. Jones}, {\em Compact negatively curved manifolds
  (of dim {$\neq 3, 4$}) are topologically rigid}, Proc. Nat. Acad. Sci.
  U.S.A., 86 (1989), pp.~3461--3463.

\bibitem{FJ89}
\leavevmode\vrule height 2pt depth -1.6pt width 23pt, {\em Negatively curved
  manifolds with exotic smooth structures}, J. Amer. Math. Soc., 2 (1989),
  pp.~899--908.

\bibitem{FJ89b}
\leavevmode\vrule height 2pt depth -1.6pt width 23pt, {\em A topological
  analogue of {M}ostow's rigidity theorem}, J. Amer. Math. Soc., 2 (1989),
  pp.~257--370.

\bibitem{FJ90}
\leavevmode\vrule height 2pt depth -1.6pt width 23pt, {\em Rigidity and other
  topological aspects of compact nonpositively curved manifolds}, Bull. Amer.
  Math. Soc. (N.S.), 22 (1990), pp.~59--64.

\bibitem{FJ94}
\leavevmode\vrule height 2pt depth -1.6pt width 23pt, {\em Complex hyperbolic
  manifolds and exotic smooth structures}, Invent. Math., 117 (1994),
  pp.~57--74.

\bibitem{GHMR05}
{\sc P.~Goerss, H.-W. Henn, M.~Mahowald, and C.~Rezk}, {\em A resolution of the
  {$K(2)$}-local sphere at the prime 3}, Ann. of Math. (2), 162 (2005),
  pp.~777--822.

\bibitem{GH04}
{\sc P.~G. Goerss and M.~J. Hopkins}, {\em Moduli spaces of commutative ring
  spectra}, in Structured ring spectra, vol.~315 of London Math. Soc. Lecture
  Note Ser., Cambridge Univ. Press, Cambridge, 2004, pp.~151--200.

\bibitem{GM95}
{\sc J.~P.~C. Greenlees and J.~P. May}, {\em Generalized {T}ate cohomology},
  Mem. Amer. Math. Soc., 113 (1995), pp.~viii+178.

\bibitem{HMS17}
{\sc D.~Heard, A.~Mathew, and V.~Stojanoska}, {\em Picard groups of higher real
  {$K$}-theory spectra at height {$p-1$}}, Compos. Math., 153 (2017),
  pp.~1820--1854.

\bibitem{HS14}
{\sc D.~Heard and V.~Stojanoska}, {\em {$K$}-theory, reality, and duality}, J.
  K-Theory, 14 (2014), pp.~526--555.

\bibitem{Hen07}
{\sc H.-W. Henn}, {\em On finite resolutions of {$K(n)$}-local spheres}, in
  Elliptic cohomology, vol.~342 of London Math. Soc. Lecture Note Ser.,
  Cambridge Univ. Press, Cambridge, 2007, pp.~122--169.

\bibitem{Hen14}
{\sc A.~G. Henriques}, {\em The homotopy groups of tmf and of its
  localizations}, vol.~201 of Mathematical Surveys and Monographs, American
  Mathematical Society, Providence, RI, 2014, pp.~xxxii+318.

\bibitem{Her91}
{\sc L.~Hern\'andez}, {\em K\"ahler manifolds and {$1/4$}-pinching}, Duke Math.
  J., 62 (1991), pp.~601--611.

\bibitem{Hew95}
{\sc T.~Hewett}, {\em Finite subgroups of division algebras over local fields},
  J. Algebra, 173 (1995), pp.~518--548.

\bibitem{Hsi66}
{\sc W.-c. Hsiang}, {\em A note on free differentiable actions of {$S\sp{1}$}
  and {$S\sp{3}$} on homotopy spheres}, Ann. of Math. (2), 83 (1966),
  pp.~266--272.

\bibitem{Kaw68}
{\sc K.~Kawakubo}, {\em Inertia groups of low dimensional complex projective
  spaces and some free differentiable actions on spheres. {I}}, Proc. Japan
  Acad., 44 (1968), pp.~873--875.

\bibitem{KM63}
{\sc M.~A. Kervaire and J.~W. Milnor}, {\em Groups of homotopy spheres. {I}},
  Ann. of Math. (2), 77 (1963), pp.~504--537.

\bibitem{KS77}
{\sc R.~C. Kirby and L.~C. Siebenmann}, {\em Foundational essays on topological
  manifolds, smoothings, and triangulations}, vol.~No. 88 of Annals of
  Mathematics Studies, Princeton University Press, Princeton, NJ; University of
  Tokyo Press, Tokyo, 1977.
\newblock With notes by John Milnor and Michael Atiyah.

\bibitem{May69}
{\sc J.~P. May}, {\em Matric {M}assey products}, J. Algebra, 12 (1969),
  pp.~533--568.

\bibitem{Mcc01}
{\sc J.~McCleary}, {\em A user's guide to spectral sequences}, vol.~58 of
  Cambridge Studies in Advanced Mathematics, Cambridge University Press,
  Cambridge, second~ed., 2001.

\bibitem{Mos68}
{\sc G.~D. Mostow}, {\em Quasi-conformal mappings in {$n$}-space and the
  rigidity of hyperbolic space forms}, Inst. Hautes \'Etudes Sci. Publ. Math.,
  (1968), pp.~53--104.

\bibitem{Mos73}
\leavevmode\vrule height 2pt depth -1.6pt width 23pt, {\em Strong rigidity of
  locally symmetric spaces}, vol.~No. 78 of Annals of Mathematics Studies,
  Princeton University Press, Princeton, NJ; University of Tokyo Press, Tokyo,
  1973.

\bibitem{Nav99}
{\sc L.~S. Nave}, {\em The cohomology of finite subgroups of {M}orava
  stabilizer groups and {S}mith-{T}oda complexes}, ProQuest LLC, Ann Arbor, MI,
  1999.
\newblock Thesis (Ph.D.)--University of Washington.

\bibitem{Nav10}
{\sc L.~S. Nave}, {\em The {S}mith-{T}oda complex {$V((p+1)/2)$} does not
  exist}, Ann. of Math. (2), 171 (2010), pp.~491--509.

\bibitem{Oku02}
{\sc B.~Okun}, {\em Exotic smooth structures on nonpositively curved symmetric
  spaces}, Algebr. Geom. Topol., 2 (2002), pp.~381--389.

\bibitem{Pra73}
{\sc G.~Prasad}, {\em Strong rigidity of {${\bf Q}$}-rank {$1$} lattices},
  Invent. Math., 21 (1973), pp.~255--286.

\bibitem{Kas15}
{\sc K.~Ramesh}, {\em Farrell-{J}ones spheres and inertia groups of complex
  projective spaces}, Forum Math., 27 (2015), pp.~3005--3015.

\bibitem{Rav78}
{\sc D.~C. Ravenel}, {\em The non-existence of odd primary {A}rf invariant
  elements in stable homotopy}, Math. Proc. Cambridge Philos. Soc., 83 (1978),
  pp.~429--443.

\bibitem{Rav86}
\leavevmode\vrule height 2pt depth -1.6pt width 23pt, {\em Complex cobordism
  and stable homotopy groups of spheres}, vol.~121 of Pure and Applied
  Mathematics, Academic Press, Inc., Orlando, FL, 1986.

\bibitem{Rez98}
{\sc C.~Rezk}, {\em Notes on the {H}opkins-{M}iller theorem}, in Homotopy
  theory via algebraic geometry and group representations ({E}vanston, {IL},
  1997), vol.~220 of Contemp. Math., Amer. Math. Soc., Providence, RI, 1998,
  pp.~313--366.

\bibitem{Ser79}
{\sc J.-P. Serre}, {\em Local fields}, vol.~67 of Graduate Texts in
  Mathematics, Springer-Verlag, New York-Berlin, 1979.
\newblock Translated from the French by Marvin Jay Greenberg.

\bibitem{Tod67}
{\sc H.~Toda}, {\em An important relation in homotopy groups of spheres}, Proc.
  Japan Acad., 43 (1967), pp.~839--842.

\bibitem{Tod68}
\leavevmode\vrule height 2pt depth -1.6pt width 23pt, {\em Extended {$p$}-th
  powers of complexes and applications to homotopy theory}, Proc. Japan Acad.,
  44 (1968), pp.~198--203.

\bibitem{Yam07}
{\sc A.~Yamaguchi}, {\em Real {$K$}-homology of complex projective spaces}, J.
  Math. Kyoto Univ., 47 (2007), pp.~203--222.

\bibitem{YZ91}
{\sc S.-T. Yau and F.~Zheng}, {\em Negatively {$\frac14$}-pinched {R}iemannian
  metric on a compact {K}\"ahler manifold}, Invent. Math., 103 (1991),
  pp.~527--535.

\end{thebibliography}
\end{document}